\documentclass[11pt]{article}

\pdfoutput=1

\RequirePackage{my_packages}
\RequirePackage{my_theorems}
\RequirePackage{my_macros}
\RequirePackage{my_tikz}

\usepackage[backend = biber,        
            language = catalan,
            style = numeric-comp,   
            giveninits = true,
            isbn = false,
            url = false,
            doi = false,
            sorting = nyt,
            maxnames = 6,
            backref=true
            ]{biblatex}

\renewbibmacro{in:}{%
  \ifentrytype{article}{}{%
  \printtext{\bibstring{in}\intitlepunct}}}
\AtEveryBibitem{%
  \clearfield{note}%
  \clearfield{edition}%
  \clearlist{language}
}

\DeclareNameAlias{labelname}{given-family}

\title{Autòmats de Stallings, un camí d'anada i tornada
}

\author{Jordi Delgado\thanks{\url{jorge.delgado@upc.edu}}}
\author{Enric Ventura\thanks{\url{enric.ventura@upc.edu}}}

\affil{Departament de Matem\`atiques, Universitat Polit\`ecnica de Catalunya, i Institut de Matem\`atiques de la UPC-BarcelonaTech (Barcelona)}

\newenvironment{cita}[1]                                                     
  {\xdef\autorcita{ #1}
   \leftskip=0.56\textwidth
   \parindent=0pt                               
   \small\sffamily}
  {\par \vspace{-0.3\baselineskip}                                           
  \begin{flushright}\bfseries\autorcita\end{flushright}
    \vspace{\baselineskip}                                     
  \xdef\autorcita{}}

\begin{document}

\maketitle
\begin{abstract}
En aquest article revisem algunes de les propietats fonamentals del grup lliure i fem una exposició detallada de la teoria dels autòmats de Stallings, una interpretació geomètrica dels seus subgrups que ha estat (i segueix essent) immensament fructífera, tant com a mitjà per entendre resultats clàssics, com com a font de nous resultats. N'expliquem alguns dels més rellevants.
\end{abstract}


\tableofcontents
\vfill
\hfill \textsf{Versió 1.00}

\newpage

\begin{cita}{Sophie Germain (1776--1831)}
L'algèbre n'est qu'une géometrie écrite;\\
la géométrie n'est qu'une algèbre figurée.
\end{cita}


\section{Introducció}

La relació entre l'àlgebra i la geometria és, probablement, una de les més fructíferes entre les diferents àrees matemàtiques. Val a dir que es tracta d'una relació bidireccional. Si bé inicialment l'àlgebra es va constituir com un potent recurs per resoldre (entre d'altres) problemes geomètrics (per exemple, amb el treball d'Euclides al Llibre II dels \emph{Elements}, molts segles més tard, de forma sistemàtica, amb el descobriment de Descartes de la geometria analítica, o, més modernament, amb el desenvolupament de la topologia algebraica), a mesura que aquesta guanyava en abstracció, i es constituïa en una disciplina matemàtica \emph{per se}, s'ha anant configurant també una relació fructífera en sentit contrari: els arguments geomètrics i topològics han esdevingut un instrument essencial per resoldre problemes d'origen i motivació algebraics.

Aquesta influència ha estat especialment acusada durant el darrer segle i escaig en l'àmbit de la teoria de grups infinits, on l'enfoc combinatori (basat en la idea de presentació d'un grup) ha donat pas a una explosió de mètodes i arguments geomètrics que han acabat conformant l'àrea anomenada \emph{Teoria Geomètrica de Grups}.

El concepte de presentació d'un grup (vegeu la~\Cref{def: presentacio}) es basa en el de \emph{grup lliure}, una mena de llenç\footnote{Com veurem a la \Cref{sec: grup lliure}, un full d'origami, del que sorgeixen tots els altres grups després de ``plegar-lo'' adequadament, és potser una metàfora més acurada.} on tots els grups poden ser dibuixats. No és, per tant, gens sorprenent la rellevància que ha tingut l'estudi del grup lliure en el desenvolupament de la teoria de grups, especialment dels grups infinits. Per als no avesats, convé recordar que el grup lliure admet una interpretació molt senzilla i intuïtiva com a ``conjunt de paraules'' usant un conjunt predefinit de lletres (vegeu la~\Cref{sec: grup lliure}).

La concisió de la descripció combinatòria donada per les presentacions és especialment propícia per a la formulació de qüestions algorísmiques, que, durant el darrer segle, han esdevingut un tercer ingredient estretament lligat a l'àlgebra i la geometria dels grups. 

Malgrat la seva aparent simplicitat, reforçada pel fet que els seus subgrups són també lliures (Teorema de Nielsen–Schreier), les relacions entre els subgrups del grup lliure són intricades i amaguen sorpreses interessants. Per exemple, veurem que --- en contrast amb el que succeeix en els ambients abelians clàssics --- el grup lliure de rang $2$ té subgrups de qualsevol rang finit, i fins i tot de rang infinit. I estudiarem el comportament de les seves interseccions, que va donar lloc a la coneguda conjectura de Hanna Neumann, un problema obert durant més de mig segle, i resolt fa uns anys independentment per Friedman a~\cite{friedman_sheaves_2015}, i per Mineyev a~\cite{mineyev_submultiplicativity_2012} (vegeu també les remarcables simplificacions de Dicks a~\cite[Appendix B]{friedman_sheaves_2015} i~\cite{dicks_simplified_2012}).

Tot i la innegable complexitat del seu reticle de subgrups, veurem que, des del punt de vista algorísmic, els grups lliures són, en general, molt ben comportats. Multitud de problemes algorísmics naturals (incloent els problemes clàssics de Dehn, així com els de la pertinença,\footnote{Membership Problem (\MP), en anglès.} la intersecció de subgrups o la finitud de l'índex) són decidibles al grup lliure $\Fn$; molts d'ells gràcies, precisament, a tècniques basades en els autòmats de Stallings, que expliquem en aquest article.  Veure \cite{delgado_list_2022} per a un inventari (no exhaustiu) de problemes decidibles usant autòmats de Stallings, els protagonistes d'aquest article.
Cal assenyalar però, que els comportaments díscols tampoc no estan lluny\ldots\ és suficient fer un producte directe de dos grups lliures per començar a trobar problemes molt naturals que són algorísmicament indecidibles (vegeu~\cite{mikhailova_occurrence_1958}).

En aquest article començarem fent una breu introducció al grup lliure i les seves propietats principals. Tot seguit, exposem els fonaments de la teoria dels autòmats de Stallings, una interpretació geomètrica molt aclaridora dels seus subgrups que, a més, ha resultat immensament fructífera en termes d'aplicacions. 
A la \Cref{sec: aplicacions} en revisarem algunes de les més importants.

\subsection{Notació, terminologia i convenis}

La major part de la notació i la terminologia utilitzades en aquest article és estàndard; tanmateix, a continuació aclarim alguns aspectes per evitar possibles confusions.

El conjunt dels nombres naturals, designat per $\NN$, inclou el zero, i especifiquem condicions sobre ell mitjançant subíndexs; per exemple, designem per $\NN_{\geq 1}$ el conjunt de nombres naturals estrictament positius.
El cardinal d'un conjunt S es designa per $\card S$, mentre que la notació $|{\cdot}|$ es reserva per indicar longitud (en diferents contextos).
Escrivim $[m,n] = \set{k \in \NN \st m \leq k \leq n}$ el conjunt de nombres naturals entre $m$ i~$n$ (ambdós inclosos) i $[0,\aleph_0] = \NN \cup \set{\aleph_0}$.

Designem per $\Free$ un grup lliure genèric, mentre que les notacions $\Free[A]$ i $\Free[\kappa]$ (\resp~$\Free[n])$ s'usem per emfatitzar una base $A$ i el rang $\kappa$ (\resp $n$, si és finit) de $\Free$, respectivament.
Les primeres lletres de l'alfabet llatí ($a,b,c,\ldots$) s'usen normalment per designar els símbols dels alfabets formals, mentre que les últimes ($u,v,w,\ldots$) solen designar paraules formals o elements del grup lliure.

Les funcions actuen per la dreta. És a dir, designem per $(x)\varphi$ (o simplement per $x \varphi$) la imatge de l'element $x$ per la funció $\varphi$, i designem per $\phi \psi$ la composició \smash{$A \xto{\phi} B \xto{\psi} C$}. En conseqüència, escrivim, per exemple, $g^{h} = h^{-1} gh$ (el conjugat de $g$ per $h$) i $[g,h] = g^{-1} h^{-1} g h$ (el commutador de $g$ i $h$). 

Designem la conjugació (d'elements o subgrups) per $\sim$, i escrivim $H \leqslant G$, $H \normaleq G$, $H \leqslant\fg G$, $H \leqslant\ff G$, $H \leqslant\fin G$ per designar que $H$ és subgrup, subgrup normal, subgrup finitament generat, factor lliure, i subgrup d'índex finit, respectivament.

\section{Grup lliure} \label{sec: grup lliure}

Comencem recordant les definicions de monoide i de grup. Un \defin{monoide} és un parell $(G,\cdot)$ on $G$ és un conjunt arbitrari, i $\cdot \colon G\times G \to G$, $(g_1,g_2) \mapsto g_1 \cdot g_2$ és una operació associativa amb element neutre (que representarem per $\trivial$). Si, a més, tot element $g\in G$ té un (únic) invers $g^{-1}$ (\ie tal que $g\cdot g^{-1} = g^{-1}\cdot g =\trivial$) diem que $(G,\cdot)$ és un \defin{grup}.
Si $A \subseteq G$, designarem per $A^{-1}$ el conjunt d'inversos d'elements de~$A$, \ie ${A^{-1}=\set{a^{-1} \st a\in A}}$. Usualment ometrem el símbol per a l'operació i escriurem $g\cdot h=gh$. Si, a més, $gh=hg$ per a tot $g,h\in G$, direm que el monoide (o el grup) és \defin{commutatiu} o \defin{abelià}; en aquest cas l'operació se sol designar per `$+$' i l'element neutre per $0$ (\emph{notació additiva}). 

Hi ha moltes maneres de definir els grups lliures i les seves bases: geomètriques, algebraiques, combinatòries, categòriques, \ldots\ La més concreta és, potser, la que adapta la noció estàndard de base de l'àlgebra lineal al context de grups, mentre que la més abstracta és probablement la categòrica.\footnote{Vegeu \url{https://en.wikipedia.org/wiki/Abstract_nonsense}.} Al llarg d'aquesta secció les desenvoluparem totes dues i veurem com estan relacionades. Abans, però, emfatitzem el polimorfisme del grup lliure amb dues caracteritzacions extra, de caràcter geomètric, que enunciem sense demostració.

\begin{thm} \label{thm: free geom}
Sigui $F$ un grup. Aleshores, els següents enunciats són equivalents:
\begin{enumerate}[dep]
\item $F$ és un grup lliure;
\item \label{item: free iff graph} $F$ és el grup fonamental d'un graf connex;\,\footnote{Vegeu la \Cref{def: fundamental group} per a la definició de grup fonamental, i l'\Cref{rem: rose free} i la \Cref{prop: S_T}.\ref{item: S_T lliure} per una reformulació en termes d'autòmats.}
\item \label{item: free iff tree} $F$ actua lliurement i sense inversions sobre les arestes d'un arbre.\footnote{Un arbre és un graf connex sense cicles. Vegeu \cite[capítol 3]{clay_office_2017}.}
\end{enumerate}
\end{thm}

Tot i la indubtable rellevància d'aquest resultat, hem preferit fer una presentació `algebraica' del grup lliure, amb l'esperança que el lector la trobi més natural. Comencem notant que és immediat de la definició de grup que els factors trivials $\trivial$ i els productes $g g^{-1}$ (anomenats \defin{cance\l.lacions}) són superflus en qualsevol producte $g_1 \cdots g_n$ d'elements d'un grup $G$ (per exemple, $h\trivial k=hk$ i $hgg^{-1}k=hk$). Un producte en el qual no apareixen factors trivials ni cance\l.lacions s'anomena \defin{producte reduït} (incloent el producte buit, que representa l'element neutre --- també anomenat \defin{element trivial} en aquest context). És clar que, eliminant repetidament factors trivials i cance\l.lacions, tot producte es pot convertir en un producte reduït, amb el mateix resultat dins de $G$. 

\begin{defn}\label{def: base}
Sigui $A\subseteq G$. Diem que $A$ és un subconjunt \emph{lliure} (o \emph{independent}) en~$G$ si dos productes reduïts diferents d'elements de $A^{\pm}=A\cup A^{-1}$ sempre donen resultats diferents a~$G$ (o, equivalentment, si el buit és l'únic producte reduït que dóna l'element trivial de~$G$). Diem que $A$ \emph{genera} (o que és un \defin{conjunt de generadors} per a) $G$ si tot element de $G$ és igual a un producte (que podem suposar reduït) d'elements d'$A^{\pm}$. Finalment, diem que $A$ és una \emph{base} de $G$ i que $G$ és un \emph{grup lliure} (sobre~$A$) si $A$ és lliure en $G$ i genera $G$. Farem servir la notació $\Free$ per referir-nos als grups lliures. 
\end{defn}

\begin{defn}
S'anomena \defin{rang} d'un grup $G$, designat per $\rk(G)$, a la mínima cardinalitat d'un conjunt de generadors per $G$. Si un grup admet un conjunt de generadors finit
direm que és \defin{finitament generat}.
\end{defn}

\begin{exm}
Considerem el grup dels enters $\mathbb{Z}$ amb notació additiva.\footnote{En notació additiva, $1$ és el generador natural del grup $(\ZZ,+)$ i no pas el neutre, designat per $0$.} Els conjunts~$\set{1}$ i $\set{-1}$ en són bases ja que, clarament, generen $\mathbb{Z}$ i cap expressió de la forma 
 \begin{equation*}
1+1+ \overset{n)}{\cdots} + 1 \qquad \text{ ó }\qquad (-1)+(-1)+\overset{n)}{\cdots}+(-1)  
 \end{equation*}
amb $n\neq 0$ dóna igual a zero. En particular, $\set{1}$ és un conjunt de generadors (òbviament mínim) de $\ZZ$, i per tant $\rk (\ZZ) = 1$.
\end{exm}

\begin{exm}
En canvi, $\set{\overline{1}}=\set{1+n\ZZ}$ \emph{no} és una base de $\mathbb{Z}/n\mathbb{Z}$, per a $n\geq 2$, ja que tot i ser-ne generador, no és un subconjunt lliure, degut a la \emph{relació} $\overline{1} + \oset{n\scriptscriptstyle{)}}{\cdots}+\overline{1}=\overline{0}$. De fet, \emph{cap} subconjunt $A\subseteq \mathbb{Z}/n\mathbb{Z}$ és base ja que el subconjunt buit no genera, i qualsevol element $\overline{a}\in \mathbb{Z}/n\mathbb{Z}$ compleix la relació: \smash{$\overline{a}+ \oset{n\scriptscriptstyle{)}}{\cdots}+\overline{a} =\overline{0}$}; per tant, el grup $\mathbb{Z}/n\mathbb{Z}$ no és lliure.
\end{exm}

Noteu que el grup trivial és lliure (amb base $\emptyset$ i, per tant, amb rang $0$) i, com hem vist a l'exemple anterior, el grup dels enters $\mathbb{Z}$ també és lliure (i té dues bases, $\{1\}$ i~$\{-1\}$). És immediat de la definició que aquests són els únics grups lliures abelians. D'altra banda, hem vist que $\mathbb{Z}/n\mathbb{Z}$ no és lliure per a cap $n\geq 2$. Amb el mateix argument es pot demostrar que cap grup finit és lliure, excepte el trivial. 

\begin{rem} \label{rem: base => gen minimal}
Si $A\subseteq \Free$ és base d'un grup (lliure) $\Free$, aleshores $A$ és un sistema de generadors \emph{minimal} de $\Free$; és a dir, $\gen{A}=\Free$, però $\gen{A\setmin S}\neq \Free$ per a tot $\emptyset \neq S\subseteq A$.
\end{rem}

\begin{proof}
Si existís un $S \neq \emptyset$ tal que $A\setmin S$ generés tot $\Free$ aleshores, per qualsevol $s\in S$, tindríem $s\in \Free=\gen{A\setmin S}$, i l'element $s\in G$ el podríem obtenir com una lletra de $A$ i, a l'hora, com un producte reduït en $A\setmin S$, contradient la hipòtesi que $A$ és una família lliure. 
\end{proof}

Tot seguit construirem grups lliures amb bases de cardinal arbitrari.

\begin{defn}
Sigui $A$ un conjunt (finit o infinit), que anomenarem \emph{alfabet}, o conjunt de \emph{lletres} elementals. Una \emph{paraula} sobre $A$ és una seqüència ordenada i finita de lletres, $a_1 a_2\cdots a_n$, on $n\geqslant 0$, i $a_i\in A$, amb possibles repeticions. Al nombre total de lletres $n$ d'una paraula se l'anomena \emph{longitud}, i escrivim $|a_1 a_2\cdots a_n|=n$. Com a conveni, designarem per 1 l'única paraula de longitud zero, o \defin{paraula buida}. Al conjunt de totes les paraules sobre $A$ el designem per $A^*$.
\end{defn}

\begin{rem}
Per a tot alfabet $A$, el conjunt $A^*$ és un monoide amb l'operació de concatenació, $u\cdot v=uv$, $u,v\in A^*$. A més, $|uv|=|u|+|v|$ i, per tant, l'únic element invertible és el neutre, 1.
\end{rem}


Aprofitant la notació exponencial per a productes successius d'un element amb sí mateix, $u^n=u\oset{n\scriptscriptstyle{)}}{\cdots}u$, podem abreujar les paraules de $A^*$ tot escrivint la corresponent potència cada vegada que una lletra apareix diverses vegades consecutivament; per exemple, $aabaaabbab=a^2ba^3b^2ab$. Si fem això agrupant al màxim, cadascuna de les potències de lletres que apareixen s'anomena una \emph{sí\l.laba}; per exemple, la paraula anterior té longitud $2+1+3+2+1+1=10$, i 6 sí\l.labes.


Observem que, sobre un alfabet d'una sola lletra, $A=\{a\}$, el monoide $A^*$ és isomorf al dels nombres naturals (en efecte, $A^*=\{1, a, a^2, a^3, \ldots\}$ i $a^n\cdot a^m =a^{n+m}$, per $n,m\geqslant 0$). En canvi, sobre un alfabet $A$ de dues o més lletres, el monoide $A^*$ és força més complicat; en particular, com que $ab\neq ba$, no és un monoide commutatiu.

Anem en la bona direcció per aconseguir que $A$ sigui una base de $A^*$: efectivament $A$ genera $A^*$, i cap producte de longitud $n\neq 0$ d'elements de $A$ dóna mai igual a~1. El problema és que $A^*$ només és un monoide (de fet, un monoide lliure); però és molt lluny de ser un grup ja que cap element, llevat del neutre 1, té invers. Per convertir-lo en un grup, haurem d'introduir els inversos de tots els elements. Com veurem, serà suficient amb introduir `inversos formals' per a les lletres elementals, $A^{-1} = \set{a^{-1} \st a\in A}$. És a dir, doblem l'alfabet $A$ amb una lletra nova per cadascuna de les antigues, i el designem per $A^{\pm}=A\sqcup A^{-1}=\{a,\, a^{-1} \mid a\in A\}$. La notació usada suggereix que cada $a^{-1}$ és l'invers del corresponent $a$, i així serà en acabar la construcció, però de moment és només una qüestió notacional: $a$ i $a^{-1}$ són simplement dues lletres diferents del nou alfabet $A^{\pm}$ (inverses formals, si voleu).

Considerem el monoide $(A^{\pm})^*$, és a dir, el conjunt de totes les paraules sobre l'alfabet $A^{\pm}$ amb l'operació de concatenació. Per exemple, si $A=\{a,b\}$ tindrem $A^{\pm}=\{a, a^{-1}, b, b^{-1}\}$ i $(A^{\pm})^*$ conté exactament $4^2=16$ paraules de longitud 2, a saber $a^2$, $aa^{-1}$, $ab$, $ab^{-1}$, $a^{-1}a$, $a^{-2}$, $a^{-1}b$, $a^{-1}b^{-1}$, $ba$, $ba^{-1}$, $b^2$, $bb^{-1}$, $b^{-1}a$, $b^{-1}a^{-1}$, $b^{-1}b$, i $b^{-2}$ (on $x^{-n}=(x^{-1})^n$, per $n\geqslant 0$). Per aconseguir que, per cada lletra $a\in A$, l'element $a^{-1}$ sigui realment l'invers de $a$, cal fer alguna cosa per forçar que $aa^{-1}=1=a^{-1}a$. Una manera d'aconseguir-ho és fent el conjunt quocient per la relació d'equi\-va\-lèn\-cia apropiada.

\begin{defn}
Anomenem \emph{reducció elemental}, designada per $\rightsquigarrow$, a la transformació consistent en eliminar una cance\l.lació dins d'una paraula; és a dir, si~$a\in A$ i $u,v \in (A^{\pm})^*$, aleshores 
 \begin{equation}
uaa^{-1}v\rightsquigarrow uv\quad \text{i} \quad ua^{-1}av\rightsquigarrow uv \,.
 \end{equation}
A la transformació inversa l'anomenem \emph{inserció elemental}, i a la seva clausura simètrica (designada per $\leftrightsquigarrow$), \defin{transformació elemental}. És a dir,
 \begin{equation}
w\leftrightsquigarrow w'\quad \Leftrightarrow \quad w\rightsquigarrow w'\  \text{ ó }\  w'\rightsquigarrow w.
 \end{equation}
Òbviament, si $w\leftrightsquigarrow w'$ llavors $|w|=|w'|\pm 2$. Finalment, definim $\sim$ com la clausura reflexo-transitiva de $\leftrightsquigarrow$; és a dir, per a tot $w,w' \in (A^{\pm})^*$
 \begin{enumerate}[ind]
\item $w \sim w $ \ i
\item $w\sim w' \ \Leftrightarrow \ \exists w_0=w, w_1, \ldots ,w_n=w' \ \text{t.q.}\  w_0 \leftrightsquigarrow w_1 \leftrightsquigarrow \cdots \leftrightsquigarrow w_n$,
 \end{enumerate}
és a dir, $w\sim w'$ si es pot passar de $w$ a $w'$ fent una quantitat finita de reduccions i/o insercions elementals, entenent que això inclou el cas amb $n=0$ passes. Per exemple, $a^2a^{-1}bb^{-1}a^{-2}b\sim a^2c^{-1}ca^{-3}b$ ja que
 $$
a^2a^{-1}bb^{-1}a^{-2}b\rightsquigarrow abb^{-1}a^{-2}b \rightsquigarrow aa^{-2}b\leftsquigarrow a^2a^{-3}b \leftsquigarrow a^2c^{-1}ca^{-3}b.
 $$
\end{defn}

És immediat de la definició que $\sim$ és una relació d'equivalència al conjunt $\IM{A}$. Podem considerar doncs el conjunt quocient, que designem per $\Free[A] = \IM{A}/\sim$. És clar que a la classe $[w]\in \Free[A]$, hi ha, precisament totes les paraules `iguals' a~${w\in \IM{A}}$ mòdul les igualtats elementals desitjades $aa^{-1}=1=a^{-1}a$, $a\in A$.
Finalment, definim a  $\Free[A]=\IM{A}/\sim$ una operació binària (clarament ben definida) adaptant de manera natural la concatenació de $\IM{A}$: per a tot $u,v \in \IM{A}$,
 \begin{equation} \label{eq: op free}
[u]\cdot [v]=[uv].
 \end{equation}

\begin{prop}
El conjunt $\Free[A]$ amb l'operació \eqref{eq: op free} és un grup.
\end{prop}

\begin{proof}
La propietat associativa és conseqüència immediata de l'associativitat de la concatenació en $\IM{A}$; l'element neutre és $[1]$; i l'invers d'una classe $[a_{i_1}^{\epsilon_1}\cdots a_{i_n}^{\epsilon_n}]\in \Free[A]$ (on $\epsilon_j = \pm 1$) és la classe $[a_{i_n}^{-\epsilon_n}\cdots a_{i_1}^{-\epsilon_1}]$. Per tant, $\Free[A]$ amb l'operació~\eqref{eq: op free} té estructura de grup.
\end{proof}

Ja tenim el grup $\Free[A]$ construït. Vegem ara que $\Free[A]$ és lliure amb base $A$. La manera natural de veure $A$ dins de $\Free[A]$ és com el conjunt de classes de les paraules positives de longitud 1, és a dir $\{[a] \st a\in A\}$. Però hi ha un petit problema tècnic aquí: estem segurs que dues lletres positives diferents estan sempre en classes diferents? en altres paraules, estem segurs que l'aplicació $\iota_A\colon A\to \Free[A]$, $a\mapsto [a]$, és injectiva? Intuïtivament sembla clar que, si $a,b\in A$ són dues lletres diferents, llavors $a\nsim b$ ja que no sembla possible transformar la lletra $a$ en una de diferent $b$, simplement usant reduccions i insercions elementals. Per demostrar-ho rigorosament usarem la proposició següent, que és important també per a altres qüestions.

\begin{defn}
Una paraula $w\in \IM{A}$ és \emph{reduïda} si no conté cap parell de lletres consecutives mútuament inverses formals. És a dir, $w=a_{i_1}^{\epsilon_1}a_{i_2}^{\epsilon_2}\cdots a_{i_n}^{\epsilon_n}$ és reduïda si sempre que dues lletres consecutives coincideixin, $a_{i_j}=a_{i_{j+1}}$, els seus signes també, $\epsilon_j=\epsilon_{j+1}$. Al conjunt de paraules reduïdes en $A$ el designarem per $R(A)\subseteq \IM{A}$.
\end{defn}

\begin{prop}\label{prop: unicitat reduida}
Tota classe d'equivalència $[w]\in \Free[A]$ conté una i només una paraula reduïda (designada per $\red{w}$).
\end{prop}

\begin{proof}
És clar que tota classe conté alguna paraula reduïda, ja que podem prendre un representant qualsevol $w\in [w]$ i (si no és ja reduït) reduir-lo aplicant successives reduccions elementals fins que no en quedi cap de disponible. Aquest procés sempre acaba en un nombre finit de passos, ja que $|w|$ és finit i a cada pas la longitud disminueix en dues unitats.

Per veure la unicitat, suposem que dues paraules reduïdes diferents ${w,w'\in R(A)}$ pertanyen a la mateixa classe i busquem una contradicció. Sigui $w=w_0 \leftrightsquigarrow
w_1 \leftrightsquigarrow
\cdots
\leftrightsquigarrow w_{n-1}
\leftrightsquigarrow w_n=w'$ una successió de transformacions elementals minimitzant $N=\sum_{i=0}^n |w_i|$. Com que $w\neq w'$ i ambdues són reduïdes, no pot ser ni $n=0$, ni $n=1$, \mbox{ni $n=2$}; per tant, $n\geqslant 3$. A més, $|w_0|<|w_1|$ i $|w_{n-1}|>|w_n|$ i, per tant, hi ha un ${j \in [1,n-1]}$ tal que $|w_{j-1}|<|w_j|>|w_{j+1}|$. Fixem-nos ara en les reduccions elementals $w_{j-1} \leftsquigarrow w_j \rightsquigarrow w_{j+1}$, i en les dues lletres de $w_j$ afectades per cadascuna d'aquestes reduccions. Si n'hi ha una (o dues) en comú, aleshores $w_{j-1}=w_{j+1}$ contradient la minimalitat de~$N$. I si els parells de lletres eliminades en ambdues reduccions són disjunts, aleshores $w_j =xa^{\epsilon}a^{-\epsilon}y b^{\delta}b^{-\delta}z$, per a certes paraules $x,y,z\in \IM{A}$, certes lletres $a,b\in A$, i certs signes $\epsilon, \delta =\pm 1$; canviant $w_{j-1} \leftsquigarrow w_j \rightsquigarrow w_{j+1}$ per $w_{j-1}=xyb^{\delta}b^{-\delta}z \rightsquigarrow xyz \leftsquigarrow xa^{\epsilon}a^{-\epsilon}yz=w_{j+1}$ ó $w_{j-1}=xa^{\epsilon}a^{-\epsilon}yz \rightsquigarrow xyz\leftsquigarrow xyb^{\delta}b^{-\delta}z=w_{j+1}$ segons convingui, reduïm el valor de $N$, en contradicció també amb la seva minimalitat. 
\end{proof}

\begin{rem}
Donat que a tota classe $[w] \in \Free[A]$ hi ha una única paraula reduïda $\red{w}\in R(A)$ (tal que $[w]=[\red{w}]$), podem pensar el grup $\Free[A]$, alternativament, com el conjunt de paraules reduïdes $R(A)$ amb l'operació $u\cdot v=\red{uv}$. Sovint aplicarem aquesta interpretació i escriurem simplement $w\in \Free[A]$ per referir-nos a l'element $[w]=[\red{w}]\in \Free[A]$.
\end{rem}

Com que les lletres (les paraules de llargada $1$, per ser precisos) són clarament reduïdes, el resultat següent és immediat de la \Cref{prop: unicitat reduida}.

\begin{cor}
L'aplicació $\iota_A\colon A\to \Free[A]$, $a\mapsto [a]$ és injectiva. \qed
\end{cor}

Ara, pensant $A\subseteq \Free[A]$ via $\iota_A$, és clar que tota paraula reduïda és producte d'elements de $A^{\pm}$ (és a dir, $A$ genera $\Free[A]$), i que el buit és l'únic producte reduït d'elements de $A^{\pm}$ que dóna $1\in \Free[A]$ (és a dir, $A$ és un subconjunt lliure de $\Free[A]$). 

\begin{cor}
Per a tot alfabet $A$, el grup $\Free[A]$ és lliure amb base $A$. \qed
\end{cor}

Hem provat l'existència de grups lliures amb bases de qualsevol cardinal. La pregunta natural ara és quan dos d'aquests grups, $\Free[A]$ i $\Free[A\!']$, són isomorfs (com a grups). I és força raonable pensar que serà, precisament, quan $A$ i $A\!'$ tinguin el mateix cardinal. Efectivament, així és; i per demostrar-ho ens serà de molta utilitat la caracterització següent de ``grup lliure'' que, en realitat, n'és la definició estàndard en termes categòrics.

\begin{prop}\label{def: grup lliure cat}
Siguin $F$ un grup, $A\subseteq F$, i designem per $\iota_A\colon A\into F$ la inclusió natural. Aleshores, $F$ és lliure amb base $A$ si i només si, per a tot grup $G$, i tota aplicació (de conjunts) $\varphi \colon A\to G$, existeix un únic morfisme de grups $\widetilde{\varphi}\colon F\to G$ tal que $\iota_A \widetilde{\varphi} = \varphi$.
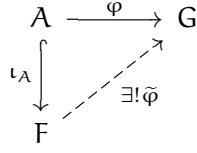
\begin{figure}[H]
\centering
\begin{tikzcd}
A \arrow[d, "\iota_A"', hook] \arrow[r, "\varphi"]     & G \\ F \arrow[ru, "\exists ! \widetilde{\varphi}"', dashed] &  
\end{tikzcd}
\caption{Definició categòrica de grup lliure $F$}
\end{figure}
\end{prop}

\begin{proof} Suposem que $F$ és lliure amb base $A$. Donats $G$ i $\varphi\colon A\to G$, el morfisme $\widetilde{\varphi}\colon F\to G$ que busquem ha de complir $a\widetilde{\varphi}=a\varphi$ per tot element $a\in A$ i, per tant, $(a_{i_1}^{\epsilon_1}\cdots a_{i_n}^{\epsilon_n})\widetilde{\varphi}=(a_{i_1} \widetilde{\varphi})^{\epsilon_1}\cdots (a_{i_n}\widetilde{\varphi})^{\epsilon_n}= (a_{i_1}\varphi)^{\epsilon_1}\cdots (a_{i_n}\varphi)^{\epsilon_n}$, per a tot producte reduït d'elements de $A$, $a_{i_1}^{\epsilon_1}\cdots a_{i_n}^{\epsilon_n}$. Com que $A$ genera $F$, el possible morfisme $\widetilde{\varphi}$ queda completament determinat per $\varphi$, la qual cosa ens dóna la unicitat. D'altra banda, com que $A$ és un subconjunt lliure de $F$ (\ie cap element de $F$ admet dues expressions diferents com a producte reduït d'elements de $A$) és clar que la igualtat anterior per a $\widetilde{\varphi}$ ens dóna una aplicació $F\to G$ ben definida. Finalment, veiem que $\widetilde{\varphi}$ és morfisme de grups: donats $x,y\in F$, considerem les seves (úniques) expressions reduïdes,
 \begin{equation*}
x=a_{i_n}^{\epsilon_n}\cdots a_{i_2}^{\epsilon_2} a_{i_1}^{\epsilon_1} \qquad \text{ i }\qquad y=b_{j_1}^{\delta_1} b_{j_2}^{\delta_2}\cdots b_{j_m}^{\delta_m},
\end{equation*}
amb $a_{i_1},\ldots ,a_{i_n}, b_{j_1},\ldots ,b_{j_m}\in A$ i $\epsilon_1,\ldots ,\epsilon_n, \delta_1, \ldots ,\delta_m=\pm 1$. Fixant-nos en les $0\leq r\leq \min\{n,m\}$ cance\l.lacions que apareixen en el producte $xy$, tenim que $a_{i_1}^{\epsilon_1}=b_{j_1}^{-\delta_1},\ldots, a_{i_r}^{\epsilon_r}=b_{j_r}^{-\delta_r}$ i $a_{i_{r+1}}^{\epsilon_{r+1}}\neq b_{j_{r+1}}^{-\delta_{r+1}}$. En aquesta situació, l'(única) expressió reduïda per a l'element $xy$ en termes de $A$ és $xy=a_{i_n}^{\epsilon_n}\cdots a_{i_{r+1}}^{\epsilon_{r+1}} b_{j_{r+1}}^{\delta_{r+1}} \cdots b_{j_m}^{\delta_m}$ i tenim 
 \begin{align*}
(xy)\widetilde{\varphi} &=( a_{i_n}^{\epsilon_n}\cdots a_{i_{r+1}}^{\epsilon_{r+1}} b_{j_{r+1}}^{\delta_{r+1}} \cdots b_{j_m}^{\delta_m} )\widetilde{\varphi} \\ &\,=\,(a_{i_n}\varphi)^{\epsilon_n}\cdots (a_{i_{r+1}}\varphi)^{\epsilon_{r+1}}\cdot (b_{j_{r+1}}\varphi)^{\delta_{r+1}}\cdots (b_{j_m}\varphi)^{\delta_m} \\
&\,=\,\big[ (a_{i_n}\varphi)^{\epsilon_n}\cdots (a_{i_{r+1}}\varphi)^{\epsilon_{r+1}}\big] \cdot \big[ (a_{i_r}\varphi)^{\epsilon_r}\cdots (a_{i_1}\varphi)^{\epsilon_1}\big] \cdot \\
& \hspace{2cm} \cdot \big[ (b_{j_1}\varphi)^{\delta_1}\cdots (b_{j_r}\varphi)^{\delta_r} \big] \cdot \big[ (b_{j_{r+1}}\varphi)^{\delta_{r+1}}\cdots (b_{j_m}\varphi)^{\delta_m} \big] \\ &\,=\,(a_{i_n}^{\epsilon_n}\cdots a_{i_{r+1}}^{\epsilon_{r+1}}\cdot a_{i_r}^{\epsilon_r}\cdots a_{i_1}^{\epsilon_1})\widetilde{\varphi} \cdot (b_{j_1}^{\delta_1}\cdots b_{j_r}^{\delta_r}\cdot b_{j_{r+1}}^{\delta_{r+1}}\cdots b_{j_m}^{\delta_m})\widetilde{\varphi} \\ &\,=\,(x)\widetilde{\varphi}\cdot (y)\widetilde{\varphi}.
 \end{align*}
 
Per veure l'altra implicació, suposem certa la propietat universal per a $F$.

Considerem $H=\langle A\rangle\leqslant F$, el subgrup de $F$ generat per $A$, i designem per $i\colon H\into F$ la inclusió. Aplicant la propietat universal a la inclusió $\varphi\colon A\into H$, obtenim un morfisme $\widetilde{\varphi}\colon F\to H$ complint $\iota_A \widetilde{\varphi}=\varphi$. Però ara, tornem a aplicar la propietat universal a la composició $\iota_A =\varphi i\colon A\into H\into F$: el morfisme $\widetilde{\varphi} i$ compleix $\iota_A (\widetilde{\varphi}i) =(\iota_A \widetilde{\varphi})i =\varphi i$, i la identitat $\id_F \colon F\to F$ també ho compleix $\iota_A \id_F =\iota_A =\varphi i$. Per tant, per la unicitat, $\id_F= \widetilde{\varphi} i$, d'on deduïm que $F=\im (\id_F)=\im (\widetilde{\varphi} i)=\im \widetilde{\varphi}\leqslant H=\langle A\rangle$ i, per tant, que $A$ genera $F$. 

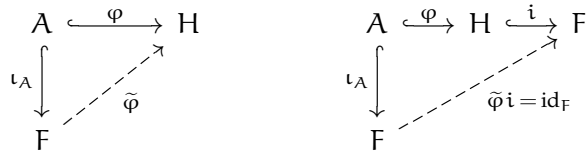
\begin{figure}[H]
\centering
\begin{tikzcd}
A \arrow[d, "\iota_A"', hook] \arrow[r, "\varphi", hook]     & H \\
F \arrow[ru, "\widetilde{\varphi}"', dashed] &  
\end{tikzcd}
\hspace{40pt}
\begin{tikzcd}[column sep=scriptsize]
A \arrow[d, "\iota_A"', hook] \arrow[r, "\varphi", hook] & H \arrow[r, "i", hook] & F \\ F \arrow[rru, "\widetilde{\varphi} i \,=\, \id_{F}"', dashed] &  
\end{tikzcd}
\end{figure}

Per altra banda, podem pensar en el conjunt $A$ com un alfabet formal i construir el grup $G=\Free[A]$; aplicant la propietat universal a la inclusió $\varphi\colon A \hookrightarrow \Free[A]$, $a\mapsto [a]$, existeix un morfisme $\widetilde{\varphi}\colon F\to \Free[A]$ complint $\iota_A \widetilde{\varphi}=\varphi$, és a dir, complint $(a_{i_1}^{\epsilon_1}\cdots a_{i_n}^{\epsilon_n})\widetilde{\varphi}= (a_{i_1}\varphi)^{\epsilon_1}\cdots (a_{i_n}\varphi)^{\epsilon_n}= [a_{i_1}]^{\epsilon_1}\cdots [a_{i_n}]^{\epsilon_n} =[a_{i_1}^{\epsilon_1}\cdots a_{i_n}^{\epsilon_n}]$, per a cada producte reduït $a_{i_1}^{\epsilon_1}\cdots a_{i_n}^{\epsilon_n}$ d'elements de $A$. Per tant, si dos productes reduïts d'elements de $A$, $a_{i_1}^{\epsilon_1}\cdots a_{i_n}^{\epsilon_n}$ i $b_{j_1}^{\delta_1}\cdots b_{j_m}^{\delta_m}$ donen el mateix resultat a $F$, $a_{i_1}^{\epsilon_1}\cdots a_{i_n}^{\epsilon_n} =_F b_{j_1}^{\delta_1}\cdots b_{j_m}^{\delta_m}$, llavors $[a_{i_1}^{\epsilon_1}\cdots a_{i_n}^{\epsilon_n}]=_{\Free[A]} [b_{j_1}^{\delta_1}\cdots b_{j_m}^{\delta_m}]$, i la~\Cref{prop: unicitat reduida}, ens assegura que són la mateixa expressió reduïda: $n=m$, $a_{i_1}=b_{j_1},\ldots ,a_{i_n}=b_{j_n}$, i $\epsilon_1=\delta_1,\ldots ,\epsilon_n=\delta_n$. Per tant, $A$ és un subconjunt lliure de $F$. \end{proof}

\begin{prop} \label{prop: isom iff =card}
Siguin $\Free$ un grup lliure amb base $A\subseteq \Free$, $\Free'$ un grup lliure amb base $A\!'\subseteq \Free'$, 
i designem per $\iota_A$ i $\iota_{A\!'}$ les respectives inclusions. Llavors, $\Free$ i $\Free'$ són grups isomorfs si i només si $A$ i $A\!'$ tenen el mateix cardinal: $\Free\simeq \Free' \Leftrightarrow \card A=\card A\!'$.
\end{prop}

\begin{proof}
Per la implicació cap a l'esquerra, suposem $\card A=\card A\!'$ i prenem una aplicació bijectiva $\eta\colon A\to A\!'$. Aplicant la propietat universal de $\Free$ a $\eta \iota_{A\!'}$ obtenim un morfisme $\alpha\colon \Free\to \Free'$ complint $\iota_A \alpha=\eta \iota_{A\!'}$; anàlogament, aplicant la propietat universal de $\Free'$ a $\eta^{-1} \iota_A$, obtenim un morfisme $\beta\colon \Free'\to \Free$ complint $\iota_{A\!'}\beta =\eta^{-1} \iota_A$; vegeu la part superior de~\eqref{eq: =card}. Però ara, apliquem la propietat universal de $\Free$ a la inclusió $\iota_A$: òbviament la identitat compleix $\iota_A \id_{\Free}=\iota_A$; però el morfisme $\alpha\beta\colon \Free\to \Free'\to \Free$ també: $\iota_A \alpha\beta =\eta \iota_{A\!'}\beta =\eta \eta^{-1} \iota_A=\iota_A$. Per unicitat, deduïm que $\alpha\beta=\id_{\Free}$. I amb un argument simètric, $\beta\alpha=\id_{\Free'}$. Per tant, $\Free\simeq \Free'$.
 \begin{equation} \label{eq: =card}
\setlength{\tabcolsep}{20pt}
\renewcommand{\arraystretch}{5}
\begin{tabular}{cc}
\begin{tikzcd}[column sep=scriptsize]
A \arrow[d, "\iota_A"', hook] \arrow[r, "\eta", hook] & A\!' \arrow[r, "\iota_{A\!'}", hook]& \Free' \\ \Free \arrow[rru, "\alpha"', dashed] &  
\end{tikzcd}
&
\begin{tikzcd}[column sep=scriptsize]
A\!' \arrow[d, "\iota_{A\!'}"', hook] \arrow[r, "\eta^{-1}", hook] & A \arrow[r, "\iota_A", hook]& \Free \\ \Free' \arrow[rru, "\beta"', dashed] &  
\end{tikzcd}\\[5pt]
\begin{tikzcd}
A \arrow[d, "\iota_A"', hook] \arrow[r, "\iota_A", hook] & \Free \\
\Free \arrow[ru, "\alpha \beta\,=\,\id_{\Free}"', dashed] &  
\end{tikzcd}
&
\begin{tikzcd}
A\!' \arrow[d, "\iota_{A\!'}"', hook] \arrow[r, "\iota_{A\!'}", hook] & \Free' \\
\Free' \arrow[ru, "\beta \alpha\,=\,\id_{\Free'}"', dashed] &  
\end{tikzcd}
\end{tabular}
 \end{equation}

Per la implicació contrària farem un argument de naturalesa diferent. És fàcil veure que, si $\card A\geqslant \aleph_0$, llavors $\card \Free[A]=\card A$. Per tant, si $\card A,\, \card A\!'\geqslant \aleph_0$ el resultat és evident. Restringim-nos, doncs, al cas $\card A<\infty$.

Observeu que, per cada grup $G$, la propietat universal de $\Free$ ens dóna una bijecció,  $\operatorname{Map}(A,G)\to \operatorname{Hom}(\Free, G)$, $\varphi\mapsto \widetilde{\varphi}$, entre el conjunt d'aplicacions de $A$ a $G$, i el conjunt de morfismes de grups de $\Free$ a $G$ (la inversa és la restricció, $\phi_{|A}\mapsfrom \phi$). Per tant, aquests dos conjunts tenen el mateix cardinal, $\card \operatorname{Map}(A,G)=\card \operatorname{Hom}(\Free,G)$. Si $\Free$ i $\Free'$ són grups isomorfs llavors, mirant aplicacions i morfismes a $\ZZ/2\ZZ$, tenim
 \begin{align*}
2^{\card A} &=\card \operatorname{Map}(A,\ZZ/2\ZZ) =\card \operatorname{Hom}(\Free,\ZZ/2\ZZ) \\ &=\card \operatorname{Hom}(\Free',\ZZ/2\ZZ) =\card \operatorname{Map}(A\!',\ZZ/2\ZZ) =2^{\card A\!'},
 \end{align*}
d'on es dedueix que $\card A=\card A\!'$. 
\end{proof}




\begin{rem} 
Un grup lliure $\Free$, doncs, pot ser-ho sobre diversos subconjunts seus $A, A\!'\subseteq \Free$; però, en aquest cas, tots hauran de tenir el mateix cardinal $\card A=\card A\!'$. Donat que les bases de $\Free$ són sistemes de generadors minimals, veiem ara que, de fet, són sistemes de generadors de cardinal mínim: si $A$ és base de $\Free$, aleshores $\rk(\Free)=\card A$. Escriurem $\Free[\kappa]$ per designar el grup lliure de rang $\kappa$ (si no volem fer referència a cap base concreta).
\end{rem}

Aquesta situació és anàloga a la de l'àlgebra lineal: un $K$-espai vectorial $E$ sobre un cos $K$ pot tenir diverses bases, però totes han de tenir sempre el mateix cardinal (la $K$-dimensió de $E$). Com ja hem vist, hi ha grups lliures de qualsevol rang, de la mateixa manera que hi ha $K$-espais vectorials de qualsevol dimensió. Una diferència important, però, és que en àlgebra lineal tot $K$-espai vectorial té bases (i, per tant, dimensió), mentre que hi ha molts grups que no tenen cap base, és a dir que \emph{no són lliures} sobre cap subconjunt (els grups finits no-trivials, per exemple).

Algebraicament, la família dels grups lliures és d'una gran rellevància. Com que, llevat d'isomorfia, n'hi ha un de sol per a cada rang $\kappa$, el designarem per $\Free[\kappa]$ si volem obviar la referència a l'alfabet concret $A$ usat per la seva construcció. Així tindrem:
 \begin{equation*}
     \Free[0]\simeq 1,\,\, \Free[1]\simeq \ZZ,\,\, \Free[2],\,\, \Free[3],\,\, \ldots \,\, ,\Free[\aleph_0],\,\, \overset{\textstyle{\footnotemark}}{\ldots}
 \end{equation*}
 \footnotetext{Qui vulgui que siguin\ldots\ vegeu \url{https://ca.wikipedia.org/wiki/Hipòtesi_del_continu}}
Clarament, els únics grups commutatius són els dos primers, força especials dins d'aquesta família. La rellevància d'aquests grups rau en el fet que, d'alguna manera,
contenen \emph{tota} la informació possible sobre \emph{tots} els grups possibles. Això es concreta a través del resultat clàssic següent, força senzill de demostrar, i alhora de fonamental importància per a la Teoria de Grups.

\begin{thm} \label{th: G = F/N}
Tot grup $G$ és quocient d'un grup lliure. És a dir, per a tot grup $G$ existeix un cardinal $\kappa$ i un subgrup normal $N\normaleq \Free[\kappa]$ tals que $G\simeq \Free[\kappa]/N$.
\end{thm}

\begin{proof}
Sigui $A\subseteq G$ un conjunt de generadors de $G$ (sempre es pot prendre $A=G$), i sigui $\kappa =\card A$ el seu cardinal. Pensem $A$ com un conjunt abstracte, i considerem el grup lliure $\Free[A]$. Per la propietat universal aplicada a la inclusió $\varphi\colon A\to G$, existeix un (únic) morfisme de grups $\widetilde{\varphi}\colon \Free[A] \to G$ tal que $[a]\widetilde{\varphi}=a$, per a tot $a\in A$. Com que $A$ genera $G$, $\widetilde{\varphi}$ és exhaustiu i, pel primer teorema d'isomorfia, $N=\ker \widetilde{\varphi}$ és un subgrup normal de $\Free[A]$ complint $\Free[A]/N \simeq \operatorname{Im}(\widetilde{\varphi})=G$.
\end{proof}

Si al \Cref{th: G = F/N} fixem una base per $\Free[\kappa ]$, i una família de generadors de $N$ \emph{com a subgrup normal} de $\Free[\kappa ]$, aleshores obtenim el concepte de presentació, essencial en teoria de grups.

\begin{defn} \label{def: presentacio}
Sigui $G$ un grup. Una \defin{presentació} per a $G$ és un parell $(A,R)$ on~$A$ és un conjunt de símbols, $R$ és un subconjunt de $\Free[A]$, i $G\isom \Free[A]/\ncl{R}$.\footnote{La \defin{clausura normal} $\ncl{R}$ és el subgrup (normal) generat per tots els conjugats d'elements de $R$.} Abusant lleugerament del llenguatge, se sol escriure $G=\pres{A}{R}$, i es diu que els elements de $A$ (\resp $R$) són els \defin{generadors}\footnote{Formalment, els generadors de $G$ donats per la presentació $G=\pres{A}{X}$ són les imatges~$[a]\widetilde{\varphi}$ dels elements $a \in A$ per la seqüència d'homomorfismes exhaustius $(A^{\pm})^*\onto \Free[A] \onto G$.} (\resp \defin{relators}) de $G$ donats per la presentació $\pres{A}{R}$, i que $w \in (A^{\pm})^*$ és una \defin{paraula en els generadors} de $G$.
\end{defn}

\begin{cor}
El grup lliure amb base $A$ admet la presentació $\Free[A]=\pres{A}{-}$. 
\end{cor}

Noteu que el \Cref{th: G = F/N} ens diu que tot grup admet una (de fet, infinites) presentacions. Diem que una presentació $\pres{A}{R}$ és \defin{finita}
si tant $A$ com $R$ són conjunt finits.
Es diu que un grup és \defin{finitament 
presentat} si admet una presentació finita.

El \Cref{th: G = F/N} admet dues interpretacions contraposades. Per un costat, implica que, co\lgem ectivament, els reticles de subgrups (normals) dels grups lliures contenen tota la informació algebraica possible sobre qualsevol grup concebible; això els converteix en uns candidats naturals i molt concrets on buscar informació sobre qualsevol grup. Però, per altra banda, per aquesta mateixa raó, aquests reticles han de ser extraordinàriament complicats, i és utòpic pretendre entendre'ls completament.

Les presentacions constitueixen la manera natural de descriure els grups des del punt de vista de la teoria combinatòria i algorísmica de grups. La contrapartida geomètrica més elemental és l'anomenat dígraf de Cayley, definit a continuació.

\begin{defn}
Siguin $G$ un grup i $S\subseteq G$ un conjunt de generadors per $G$. Aleshores, el \defin{dígraf de Cayley de $G$ respecte $S$}, designat per $\cayley(G,S)$, és el graf dirigit (dígraf) amb conjunt de vèrtexs $G$, i un arc $g\xarc{s\,} gs$ per a cada element $g\in G$ i cada generador $s\in S^{\pm}$.
\end{defn}

Observeu que els dígrafs de Cayley (respecte a conjunts generadors) són connexos: com que $\gen{S}=G$, qualsevol $g\in G$ es pot escriure en la forma $g=s_{i_1}^{\epsilon_1} \cdots s_{i_k}^{\epsilon_k}$, per a certs $k\geq 0$,  $s_{i_j}\in S$, i $\epsilon_j=\pm 1$ i, llavors, el camí
 \begin{equation}
1 \! \xarc{\raisebox{0.6ex}{$\scriptstyle{s_{i_1}^{\epsilon_1}}$}} s_{i_1}^{\epsilon_1} \xarc{\raisebox{0.6ex}{$\scriptstyle{s_{i_2}^{\epsilon_2}}$}} s_{i_1}^{\epsilon_1}s_{i_2}^{\epsilon_2} \xarc{\raisebox{0.6ex}{$\scriptstyle{s_{i_3}^{\epsilon_3}}$}} \ \cdots \ \xarc{\raisebox{0.6ex}{$\scriptstyle{s_{i_k}^{\epsilon_k}}$}} g
 \end{equation}
connecta el vèrtex $1\in G$ amb el vèrtex genèric $g\in G$. Per altra banda, els dígrafs de Cayley són també $(2 \,\card S$)-regulars:\footnote{Un digraf $\Ati$ és \defin{regular} si els nombres  d'arcs entrant a i sortint d'un vèrtex són independents del vèrtex  $\verti \in \Verts \Ati$.} per definició, de cada vèrtex $g\in G$ en surt exactament una aresta per cada $s\in S^{\pm}$ (anant al vèrtex $gs\in G$) i n'arriba també exactament una per cada $s\in S^{\pm}$ (venint de $gs^{-1}\in G$); en general, algunes d'aquestes arestes poden ser para\lgem eles, o llaços.  

Recalquem que $\cayley(G,S)$ depèn fortament del conjunt de generadors escollit~$S$ (vegeu \Cref{fig: Cayley(Z)}). Per tant, no és un objecte geomètric canònicament associat al grup $G$, sinó al grup $G$ \emph{i a un conjunt de generadors concret} $S$. 

\begin{exm}
El dígraf de Cayley del grup lliure $\Free[\{a\}]\simeq \ZZ$ respecte la base $\{a\}$ (és a dir, $\{1\}$ en notació additiva) és un camí dirigit amb infinits vèrtexs consecutius, connectats cadascun al següent per una $a$-aresta (i a l'anterior per una $(a^{-1})$-aresta). A la~\Cref{fig: Cayley(Z)} podeu veure dibuixats els dígrafs de Cayley $\cayley(\ZZ,\set{1})$ i $\cayley(\ZZ,\set{2,3})$ on, com és habitual amb els dígrafs involutius (vegeu la \Cref{def: positive part}), dibuixem només els arcs positius, i assenyalem amb $\bp$ el vèrtex corresponent a l'element neutre.
\end{exm}

\begin{figure}[H] 
\centering
  \begin{tikzpicture}[shorten >=1pt, node distance=1 and 1, on grid,auto,>=stealth']
  \begin{scope}
  
   \node[state,accepting] (0) {};
   \node[state] (-1) [left = of 0]{};
   \node[state] (-2) [left = of -1]{};
   \node[state] (-3) [left = of -2]{};
   \node[] (-4) [left = 1 of -3]{$\cdots$};
   \node[] (-5) [left = of -4]{};
   \node[state] (1) [right = of 0]{};
   \node[state] (2) [right = of 1]{};
   \node[state] (3) [right = of 2]{};
   \node[] (4) [right = 1 of 3]{$\cdots$};
   \node[] (5) [right = of 4]{};

    \path[->] (-4) edge[blue] (-3);        
    \path[->] (-3) edge[blue] (-2);
    \path[->] (-2) edge[blue] (-1);
    \path[->] (-1) edge[blue] (0);
    \path[->] (0) edge[blue] node[above] {\scriptsize{$1$}} (1);
    \path[->] (1) edge[blue] (2);
    \path[->] (2) edge[blue] (3);
    \path[->] (3) edge[blue] (4);
    \end{scope}

   \begin{scope}[yshift=-1.5 cm]
   
   \node[state,accepting] (0) {};
   \node[state] (-1) [left = of 0]{};
   \node[state] (-2) [left = of -1]{};
   \node[state] (-3) [left = of -2]{};
   \node[] (-4) [left = 1 of -3]{$\cdots$};
   \node[] (-5) [left = of -4]{};
   \node[] (-6) [left = of -5]{};
   \node[state] (1) [right = of 0]{};
   \node[state] (2) [right = of 1]{};
   \node[state] (3) [right = of 2]{};
   \node[] (4) [right = 1 of 3]{$\cdots$};
   \node[] (5) [right = of 4]{};
   \node[] (6) [right = of 5]{};
    
    \path[->]
        (-4) edge[Sepia, bend left = 30]
        (-2);

    \path[->]
        (-3) edge[Sepia, bend left = 30]
        (-1);
    
    \path[->]
        (-2) edge[Sepia, bend left = 30]
        (0);

    \path[->]
        (-1) edge[Sepia, bend left = 30]
        (1);

    \path[->]
        (0) edge[Sepia, bend left = 30]
            node[above] {\scriptsize{$2$}}
            (2);
            
    \path[->]
        (1) edge[Sepia, bend left = 30]
        (3);
            
    \path[->]
        (2) edge[Sepia, bend left = 30]
        (4);
    
    \path[->]
    (-4) edge[OliveGreen, bend right = 30]
    (-1);
     
    \path[->]
    (-3) edge[OliveGreen, bend right = 30]
    (0); 
     
    \path[->]
    (-2) edge[OliveGreen, bend right = 30]
    (1);   
    
    \path[->]
    (-1) edge[OliveGreen, bend right = 30]
    (2);    
        
    \path[->]
    (0) edge[OliveGreen, bend right = 30] node[below] {\scriptsize{$3$}}
    (3);
        
    \path[->]
    (1) edge[OliveGreen, bend right = 30]
    (4);
    \end{scope}
\end{tikzpicture}
\vspace{-5pt}
\caption{$\cayley(\ZZ,\set{1})$ (a dalt) i $\cayley(\ZZ,\set{2,3})$ (a sota)}
\label{fig: Cayley(Z)}
\end{figure}

\begin{exm}
És fàcil veure que el dígraf de Cayley del grup lliure $\Free[\{a,b\}]$ respecte de la base $\{a,b\}$ té l'aspecte de la \Cref{fig: Sch(F_2)}. Per convèncer-nos-en, només cal tenir en compte les dues observacions anteriors (es tracta d'un dígraf connex, i $4$-regular), juntament amb els fets d'ésser infinit ($\card \Free[\{a,b\}]=\aleph_0$) i d'ésser un arbre: un hipotètic camí tancat ens donaria dues expressions diferents d'un mateix element $g$ com a producte reduït de les lletres $a,\, a^{-1},\, b,\, b^{-1}$, contradient el fet que $\Free[\{a,b\}]$ és lliure en $\{a,b\}$. Anàlogament, el dígraf de Cayley del grup lliure $\Free[\kappa]$ de rang $\kappa$ respecte una base és el corresponent arbre infinit $2\kappa$-regular.
\end{exm}

\begin{figure}[h]
\centering
  \begin{tikzpicture}[shorten >=1pt, node distance=1.2 and 1.2, on grid,auto,>=stealth',transform shape]
  
  \node[state,accepting] (bp) {};

  \begin{scope}
  
  \newcommand{\dx}{1.6}
  \newcommand{\dy}{1.6}
  \newcommand{\dxx}{\dx*0.5}
  \newcommand{\dxxx}{\dx*0.3}
  \newcommand{\dxxxx}{\dx*0.18}
  \newcommand{\dyy}{\dy*0.5}
  \newcommand{\dyyy}{\dy*0.3}
  \newcommand{\dyyyy}{\dy*0.18}

   \node[smallstate] (1) {};
   \node[smallstate] (a) [right = \dx of 1]{};
   \node[smallstate] (aa) [right = \dxx of a]{};
   \node[tinystate] (aaa) [right = \dxxx of aa]{};
   \node[emptystate] (aaaa) [right = \dxxxx of aaa]{};
   
   \node[tinystate] (ab) [above = \dyy of a]{};
   \node[tinystate] (aab) [above = \dyyy of aa]{};
   \node[emptystate] (aabb) [above = \dyyyy of aab]{};
   \node[emptystate] (aaab) [above = \dyyyy of aaa]{};
   \node[emptystate] (aaaab) [above = \dyyyy of aaaa]{};
   
   \node[tinystate] (aB) [below = \dyy of a]{};
   \node[tinystate] (aaB) [below = \dyyy of aa]{};
   \node[emptystate] (aaBB) [below = \dyyyy of aaB]{};
   \node[emptystate] (aaaB) [below = \dyyyy of aaa]{};
   \node[emptystate] (aaaaB) [below = \dyyyy of aaaa]{};
   
   \node[tinystate] (aba) [right = \dxxx of ab]{};
   \node[tinystate] (abA) [left = \dxxx of ab]{};
   \node[emptystate] (aaba) [right = \dxxxx of aab]{};
   \node[emptystate] (aabA) [left = \dxxxx of aab]{};

   \node[tinystate] (aBa) [right = \dxxx of aB]{};
   \node[tinystate] (aBA) [left = \dxxx of aB]{};
   \node[emptystate] (aaBa) [right = \dxxxx of aaB]{};
   \node[emptystate] (aaBA) [left = \dxxxx of aaB]{};

   \node[tinystate] (abb) [above = \dyyy of ab]{};
   \node[tinystate] (aBB) [below = \dyyy of aB]{};
   
   \node[emptystate] (abbA) [left = \dxxxx of abb]{};
   \node[emptystate] (abbb) [above = \dyyyy of abb]{};
   \node[emptystate] (abba) [right = \dxxxx of abb]{};
   
   \node[emptystate] (aBBA) [left = \dxxxx of aBB]{};
   \node[emptystate] (aBBB) [below = \dyyyy of aBB]{};
   \node[emptystate] (aBBa) [right = \dxxxx of aBB]{};
   
   \node[emptystate] (abAA) [left = \dxxxx of abA]{};
   \node[emptystate] (abAb) [above = \dyyyy of abA]{};
   \node[emptystate] (abAB) [below = \dxxxx of abA]{};
   
   \node[emptystate] (aBAA) [left = \dxxxx of aBA]{};
   \node[emptystate] (aBAb) [above = \dyyyy of aBA]{};
   \node[emptystate] (aBAB) [below = \dxxxx of aBA]{};
   
   \node[emptystate] (abaa) [right = \dxxxx of aba]{};
   \node[emptystate] (abab) [above = \dyyyy of aba]{};
   \node[emptystate] (abaB) [below = \dxxxx of aba]{};
   
   \node[emptystate] (aBaa) [right = \dxxxx of aBa]{};
   \node[emptystate] (aBab) [above = \dyyyy of aBa]{};
   \node[emptystate] (aBaB) [below = \dxxxx of aBa]{};
     
   \path[->] (a) edge[red] (ab);
   \path[->] (aa) edge[red] (aab);
   \path[->,densely dotted] (aab) edge[red] (aabb);
   \path[->,densely dotted] (aaa) edge[red] (aaab);
   
   \path[->] (aB) edge[red] (a);
   \path[->] (aaB) edge[red] (aa);
   \path[->,densely dotted] (aaBB) edge[red] (aaB);
   \path[->,densely dotted] (aaaB) edge[red] (aaa);
   
   \path[->] (ab) edge[red] (abb);
   \path[->] (aBB) edge[red] (aB);
   
   \path[->] (1) edge[blue] node[pos = 0.45,above] {\scriptsize{$a$}}(a);
   \path[->] (a) edge[blue] (aa);
   \path[->] (aa) edge[blue] (aaa);
   \path[->,densely dotted] (aaa) edge[blue] (aaaa);
   
   \path[->] (abA) edge[blue] (ab);
   \path[->] (ab) edge[blue] (aba);
   \path[->] (ab) edge[blue] (aba);
   
   \path[->,densely dotted] (aabA) edge[blue] (aab);
   \path[->,densely dotted] (aab) edge[blue] (aaba);
   
   \path[->] (aBA) edge[blue] (aB);
   \path[->] (aB) edge[blue] (aBa);
   \path[->] (aB) edge[blue] (aBa);
   
   \path[->,densely dotted] (aaBA) edge[blue] (aaB);
   \path[->,densely dotted] (aaB) edge[blue] (aaBa);
   
   \path[->,densely dotted] (abb) edge[red] (abbb);
   \path[->,densely dotted] (abbA) edge[blue] (abb);
   \path[->,densely dotted] (abb) edge[blue] (abba);
   
   \path[->,densely dotted] (aBBB) edge[red] (aBB);
   \path[->,densely dotted] (aBBA) edge[blue] (aBB);
   \path[->,densely dotted] (aBB) edge[blue] (aBBa);
   
   \path[->,densely dotted] (abA) edge[red] (abAb);
   \path[->,densely dotted] (abAA) edge[blue] (abA);
   \path[->,densely dotted] (abAB) edge[red] (abA);
   
   \path[->,densely dotted] (aBA) edge[red] (aBAb);
   \path[->,densely dotted] (aBAA) edge[blue] (aBA);
   \path[->,densely dotted] (aBAB) edge[red] (aBA);
   
   \path[->,densely dotted] (aba) edge[red] (abab);
   \path[->,densely dotted] (aba) edge[blue] (abaa);
   \path[->,densely dotted] (abaB) edge[red] (aba);
   
   \path[->,densely dotted] (aBa) edge[red] (aBab);
   \path[->,densely dotted] (aBa) edge[blue] (aBaa);
   \path[->,densely dotted] (aBaB) edge[red] (aBa);
  \end{scope}
  
  \begin{scope}[rotate=90]
  
  \newcommand{\dx}{1.6}
  \newcommand{\dy}{1.6}
  \newcommand{\dxx}{\dx*0.5}
  \newcommand{\dxxx}{\dx*0.3}
  \newcommand{\dxxxx}{\dx*0.18}
  \newcommand{\dyy}{\dy*0.5}
  \newcommand{\dyyy}{\dy*0.3}
  \newcommand{\dyyyy}{\dy*0.18}

   \node[smallstate] (1) {};
   \node[smallstate] (a) [right = \dx of 1]{};
   \node[smallstate] (aa) [right = \dxx of a]{};
   \node[tinystate] (aaa) [right = \dxxx of aa]{};
   \node[emptystate] (aaaa) [right = \dxxxx of aaa]{};
   
   \node[tinystate] (ab) [above = \dyy of a]{};
   \node[tinystate] (aab) [above = \dyyy of aa]{};
   \node[emptystate] (aabb) [above = \dyyyy of aab]{};
   \node[emptystate] (aaab) [above = \dyyyy of aaa]{};
   \node[emptystate] (aaaab) [above = \dyyyy of aaaa]{};
   
   \node[tinystate] (aB) [below = \dyy of a]{};
   \node[tinystate] (aaB) [below = \dyyy of aa]{};
   \node[emptystate] (aaBB) [below = \dyyyy of aaB]{};
   \node[emptystate] (aaaB) [below = \dyyyy of aaa]{};
   \node[emptystate] (aaaaB) [below = \dyyyy of aaaa]{};
   
   \node[tinystate] (aba) [right = \dxxx of ab]{};
   \node[tinystate] (abA) [left = \dxxxx of ab]{};
   \node[emptystate] (aaba) [right = \dxxxx of aab]{};
   \node[emptystate] (aabA) [left = \dxxxx of aab]{};

   \node[tinystate] (aBa) [right = \dxxx of aB]{};
   \node[tinystate] (aBA) [left = \dxxx of aB]{};
   \node[emptystate] (aaBa) [right = \dxxxx of aaB]{};
   \node[emptystate] (aaBA) [left = \dxxxx of aaB]{};

   \node[tinystate] (abb) [above = \dyyy of ab]{};
   \node[tinystate] (aBB) [below = \dyyy of aB]{};
   
   \node[emptystate] (abbA) [left = \dxxxx of abb]{};
   \node[emptystate] (abbb) [above = \dyyyy of abb]{};
   \node[emptystate] (abba) [right = \dxxxx of abb]{};
   
   \node[emptystate] (aBBA) [left = \dxxxx of aBB]{};
   \node[emptystate] (aBBB) [below = \dyyyy of aBB]{};
   \node[emptystate] (aBBa) [right = \dxxxx of aBB]{};
   
   \node[emptystate] (abAA) [left = \dxxxx of abA]{};
   \node[emptystate] (abAb) [above = \dyyyy of abA]{};
   \node[emptystate] (abAB) [below = \dxxxx of abA]{};
   
   \node[emptystate] (aBAA) [left = \dxxxx of aBA]{};
   \node[emptystate] (aBAb) [above = \dyyyy of aBA]{};
   \node[emptystate] (aBAB) [below = \dxxxx of aBA]{};
   
   \node[emptystate] (abaa) [right = \dxxxx of aba]{};
   \node[emptystate] (abab) [above = \dyyyy of aba]{};
   \node[emptystate] (abaB) [below = \dxxxx of aba]{};
   
   \node[emptystate] (aBaa) [right = \dxxxx of aBa]{};
   \node[emptystate] (aBab) [above = \dyyyy of aBa]{};
   \node[emptystate] (aBaB) [below = \dxxxx of aBa]{};
     
   \path[->] (ab) edge[blue] (a);
   \path[->] (aab) edge[blue] (aa);
   \path[->,densely dotted] (aabb) edge[blue] (aab);
   \path[->,densely dotted] (aaab) edge[blue] (aaa);
   
   \path[->] (a) edge[blue] (aB);
   \path[->] (aa) edge[blue] (aaB);
   \path[->,densely dotted] (aaB) edge[blue] (aaBB);
   \path[->,densely dotted] (aaa) edge[blue] (aaaB);
   
   \path[->] (abb) edge[blue] (ab);
   \path[->] (aB) edge[blue] (aBB);
   
   \path[->] (1) edge[red]  node[pos = 0.45,below] {\rotatebox[origin=c]{-90}{\scriptsize{$b$}}} (a);
   \path[->] (a) edge[red] (aa);
   \path[->] (aa) edge[red] (aaa);
   \path[->,densely dotted] (aaa) edge[red] (aaaa);
   
   \path[->] (abA) edge[red] (ab);
   \path[->] (ab) edge[red] (aba);
   \path[->] (ab) edge[red] (aba);
   
   \path[->,densely dotted] (aabA) edge[red] (aab);
   \path[->,densely dotted] (aab) edge[red] (aaba);
   
   \path[->] (aBA) edge[red] (aB);
   \path[->] (aB) edge[red] (aBa);
   \path[->] (aB) edge[red] (aBa);
   
   \path[->,densely dotted] (aaBA) edge[red] (aaB);
   \path[->,densely dotted] (aaB) edge[red] (aaBa);
   
   \path[->,densely dotted] (abbb) edge[blue] (abb);
   \path[->,densely dotted] (abbA) edge[red] (abb);
   \path[->,densely dotted] (abb) edge[red] (abba);
   
   \path[->,densely dotted] (aBB) edge[blue] (aBBB);
   \path[->,densely dotted] (aBBA) edge[red] (aBB);
   \path[->,densely dotted] (aBB) edge[red] (aBBa);
   
   \path[->,densely dotted] (abAb) edge[blue] (abA);
   \path[->,densely dotted] (abAA) edge[red] (abA);
   \path[->,densely dotted] (abA) edge[blue] (abAB);
   
   \path[->,densely dotted] (aBAb) edge[blue] (aBA);
   \path[->,densely dotted] (aBAA) edge[red] (aBA);
   \path[->,densely dotted] (aBA) edge[blue] (aBAB);
   
   \path[->,densely dotted] (abab) edge[blue] (aba);
   \path[->,densely dotted] (aba) edge[red] (abaa);
   \path[->,densely dotted] (aba) edge[blue] (abaB);
   
   \path[->,densely dotted] (aBab) edge[blue] (aBa);
   \path[->,densely dotted] (aBa) edge[red] (aBaa);
   \path[->,densely dotted] (aBa) edge[blue] (aBaB);
  \end{scope}
  
  \begin{scope}[rotate=180]
  \newcommand{\dx}{1.6}
  \newcommand{\dy}{1.6}
  \newcommand{\dxx}{\dx*0.5}
  \newcommand{\dxxx}{\dx*0.3}
  \newcommand{\dxxxx}{\dx*0.18}
  \newcommand{\dyy}{\dy*0.5}
  \newcommand{\dyyy}{\dy*0.3}
  \newcommand{\dyyyy}{\dy*0.18}

  \node[smallstate] (1) {};
  \node[smallstate] (a) [right = \dx of 1]{};
  \node[smallstate] (aa) [right = \dxx of a]{};
  \node[tinystate] (aaa) [right = \dxxx of aa]{};
  \node[emptystate] (aaaa) [right = \dxxxx of aaa]{};
   
  \node[tinystate] (ab) [above = \dyy of a]{};
  \node[tinystate] (aab) [above = \dyyy of aa]{};
  \node[emptystate] (aabb) [above = \dyyyy of aab]{};
  \node[emptystate] (aaab) [above = \dyyyy of aaa]{};
  \node[emptystate] (aaaab) [above = \dyyyy of aaaa]{};
   
  \node[tinystate] (aB) [below = \dyy of a]{};
  \node[tinystate] (aaB) [below = \dyyy of aa]{};
  \node[emptystate] (aaBB) [below = \dyyyy of aaB]{};
  \node[emptystate] (aaaB) [below = \dyyyy of aaa]{};
  \node[emptystate] (aaaaB) [below = \dyyyy of aaaa]{};
   
  \node[tinystate] (aba) [right = \dxxx of ab]{};
  \node[tinystate] (abA) [left = \dxxx of ab]{};
  \node[emptystate] (aaba) [right = \dxxxx of aab]{};
  \node[emptystate] (aabA) [left = \dxxxx of aab]{};

  \node[tinystate] (aBa) [right = \dxxx of aB]{};
  \node[tinystate] (aBA) [left = \dxxx of aB]{};
  \node[emptystate] (aaBa) [right = \dxxxx of aaB]{};
  \node[emptystate] (aaBA) [left = \dxxxx of aaB]{};

  \node[tinystate] (abb) [above = \dyyy of ab]{};
  \node[tinystate] (aBB) [below = \dyyy of aB]{};
   
  \node[emptystate] (abbA) [left = \dxxxx of abb]{};
  \node[emptystate] (abbb) [above = \dyyyy of abb]{};
  \node[emptystate] (abba) [right = \dxxxx of abb]{};
   
  \node[emptystate] (aBBA) [left = \dxxxx of aBB]{};
  \node[emptystate] (aBBB) [below = \dyyyy of aBB]{};
  \node[emptystate] (aBBa) [right = \dxxxx of aBB]{};
   
  \node[emptystate] (abAA) [left = \dxxxx of abA]{};
  \node[emptystate] (abAb) [above = \dyyyy of abA]{};
  \node[emptystate] (abAB) [below = \dxxxx of abA]{};
   
  \node[emptystate] (aBAA) [left = \dxxxx of aBA]{};
  \node[emptystate] (aBAb) [above = \dyyyy of aBA]{};
  \node[emptystate] (aBAB) [below = \dxxxx of aBA]{};
   
  \node[emptystate] (abaa) [right = \dxxxx of aba]{};
  \node[emptystate] (abab) [above = \dyyyy of aba]{};
  \node[emptystate] (abaB) [below = \dxxxx of aba]{};
   
  \node[emptystate] (aBaa) [right = \dxxxx of aBa]{};
  \node[emptystate] (aBab) [above = \dyyyy of aBa]{};
  \node[emptystate] (aBaB) [below = \dxxxx of aBa]{};
 
  \path[->] (ab) edge[red] (a);
  \path[->] (aab) edge[red] (aa);
  \path[->,densely dotted] (aabb) edge[red] (aab);
  \path[->,densely dotted] (aaab) edge[red] (aaa);
   
  \path[->] (a) edge[red] (aB);
  \path[->] (aa) edge[red] (aaB);
  \path[->,densely dotted] (aaB) edge[red] (aaBB);
  \path[->,densely dotted] (aaa) edge[red] (aaaB);
   
  \path[->] (abb) edge[red] (ab);
  \path[->] (aB) edge[red] (aBB);
   
  \path[->] (a) edge[blue] (1);
  \path[->] (aa) edge[blue] (a);
  \path[->] (aaa) edge[blue] (aa);
  \path[->,densely dotted] (aaaa) edge[blue] (aaa);
   
  \path[->] (ab) edge[blue] (abA);
  \path[->] (aba) edge[blue] (ab);
  \path[->] (aba) edge[blue] (ab);
   
  \path[->,densely dotted] (aab) edge[blue] (aabA);
  \path[->,densely dotted] (aaba) edge[blue] (aab);
   
  \path[->] (aB) edge[blue] (aBA);
  \path[->] (aBa) edge[blue] (aB);
  \path[->] (aBa) edge[blue] (aB);
   
  \path[->,densely dotted] (aaB) edge[blue] (aaBA);
  \path[->,densely dotted] (aaBa) edge[blue] (aaB);
   
  \path[->,densely dotted] (abbb) edge[red] (abb);
  \path[->,densely dotted] (abb) edge[blue] (abbA);
  \path[->,densely dotted] (abba) edge[blue] (abb);
   
  \path[->,densely dotted] (aBB) edge[red] (aBBB);
  \path[->,densely dotted] (aBB) edge[blue] (aBBA);
  \path[->,densely dotted] (aBBa) edge[blue] (aBB);
   
  \path[->,densely dotted] (abAb) edge[red] (abA);
  \path[->,densely dotted] (abA) edge[blue] (abAA);
  \path[->,densely dotted] (abA) edge[red] (abAB);
   
  \path[->,densely dotted] (aBAb) edge[red] (aBA);
  \path[->,densely dotted] (aBA) edge[blue] (aBAA);
  \path[->,densely dotted] (aBA) edge[red] (aBAB);
   
  \path[->,densely dotted] (abab) edge[red] (aba);
  \path[->,densely dotted] (abaa) edge[blue] (aba);
  \path[->,densely dotted] (aba) edge[red] (abaB);
   
  \path[->,densely dotted] (aBab) edge[red] (aBa);
  \path[->,densely dotted] (aBaa) edge[blue] (aBa);
  \path[->,densely dotted] (aBa) edge[red] (aBaB);
  \end{scope}
  
  \begin{scope}[rotate=270]
  \newcommand{\dx}{1.6}
  \newcommand{\dy}{1.6}
  \newcommand{\dxx}{\dx*0.5}
  \newcommand{\dxxx}{\dx*0.3}
  \newcommand{\dxxxx}{\dx*0.18}
  \newcommand{\dyy}{\dy*0.5}
  \newcommand{\dyyy}{\dy*0.3}
  \newcommand{\dyyyy}{\dy*0.18}

  \node[smallstate] (1) {};
  \node[smallstate] (a) [right = \dx of 1]{};
  \node[smallstate] (aa) [right = \dxx of a]{};
  \node[tinystate] (aaa) [right = \dxxx of aa]{};
  \node[emptystate] (aaaa) [right = \dxxxx of aaa]{};
   
  \node[tinystate] (ab) [above = \dyy of a]{};
  \node[tinystate] (aab) [above = \dyyy of aa]{};
  \node[emptystate] (aabb) [above = \dyyyy of aab]{};
  \node[emptystate] (aaab) [above = \dyyyy of aaa]{};
  \node[emptystate] (aaaab) [above = \dyyyy of aaaa]{};
   
  \node[tinystate] (aB) [below = \dyy of a]{};
  \node[tinystate] (aaB) [below = \dyyy of aa]{};
  \node[emptystate] (aaBB) [below = \dyyyy of aaB]{};
  \node[emptystate] (aaaB) [below = \dyyyy of aaa]{};
  \node[emptystate] (aaaaB) [below = \dyyyy of aaaa]{};
   
  \node[tinystate] (aba) [right = \dxxx of ab]{};
  \node[tinystate] (abA) [left = \dxxx of ab]{};
  \node[emptystate] (aaba) [right = \dxxxx of aab]{};
  \node[emptystate] (aabA) [left = \dxxxx of aab]{};

  \node[tinystate] (aBa) [right = \dxxx of aB]{};
  \node[tinystate] (aBA) [left = \dxxx of aB]{};
  \node[emptystate] (aaBa) [right = \dxxxx of aaB]{};
  \node[emptystate] (aaBA) [left = \dxxxx of aaB]{};

  \node[tinystate] (abb) [above = \dyyy of ab]{};
  \node[tinystate] (aBB) [below = \dyyy of aB]{};
   
  \node[emptystate] (abbA) [left = \dxxxx of abb]{};
  \node[emptystate] (abbb) [above = \dyyyy of abb]{};
  \node[emptystate] (abba) [right = \dxxxx of abb]{};
   
  \node[emptystate] (aBBA) [left = \dxxxx of aBB]{};
  \node[emptystate] (aBBB) [below = \dyyyy of aBB]{};
  \node[emptystate] (aBBa) [right = \dxxxx of aBB]{};
   
  \node[emptystate] (abAA) [left = \dxxxx of abA]{};
  \node[emptystate] (abAb) [above = \dyyyy of abA]{};
  \node[emptystate] (abAB) [below = \dxxxx of abA]{};
   
  \node[emptystate] (aBAA) [left = \dxxxx of aBA]{};
  \node[emptystate] (aBAb) [above = \dyyyy of aBA]{};
  \node[emptystate] (aBAB) [below = \dxxxx of aBA]{};
   
  \node[emptystate] (abaa) [right = \dxxxx of aba]{};
  \node[emptystate] (abab) [above = \dyyyy of aba]{};
  \node[emptystate] (abaB) [below = \dxxxx of aba]{};
   
  \node[emptystate] (aBaa) [right = \dxxxx of aBa]{};
  \node[emptystate] (aBab) [above = \dyyyy of aBa]{};
  \node[emptystate] (aBaB) [below = \dxxxx of aBa]{};
 
  \path[->] (a) edge[blue] (ab);
  \path[->] (aa) edge[blue] (aab);
  \path[->,densely dotted] (aab) edge[blue] (aabb);
  \path[->,densely dotted] (aaa) edge[blue] (aaab);
   
  \path[->] (aB) edge[blue] (a);
  \path[->] (aaB) edge[blue] (aa);
  \path[->,densely dotted] (aaBB) edge[blue] (aaB);
  \path[->,densely dotted] (aaaB) edge[blue] (aaa);
   
  \path[->] (ab) edge[blue] (abb);
  \path[->] (aBB) edge[blue] (aB);
   
  \path[->] (a) edge[red] (1);
  \path[->] (aa) edge[red] (a);
  \path[->] (aaa) edge[red] (aa);
  \path[->,densely dotted] (aaaa) edge[red] (aaa);
   
  \path[->] (ab) edge[red] (abA);
  \path[->] (aba) edge[red] (ab);
  \path[->] (aba) edge[red] (ab);
   
  \path[->,densely dotted] (aab) edge[red] (aabA);
  \path[->,densely dotted] (aaba) edge[red] (aab);
   
  \path[->] (aB) edge[red] (aBA);
  \path[->] (aBa) edge[red] (aB);
  \path[->] (aBa) edge[red] (aB);
   
  \path[->,densely dotted] (aaB) edge[red] (aaBA);
  \path[->,densely dotted] (aaBa) edge[red] (aaB);
   
  \path[->,densely dotted] (abb) edge[blue] (abbb);
  \path[->,densely dotted] (abb) edge[red] (abbA);
  \path[->,densely dotted] (abba) edge[red] (abb);
   
  \path[->,densely dotted] (aBBB) edge[blue] (aBB);
  \path[->,densely dotted] (aBB) edge[red] (aBBA);
  \path[->,densely dotted] (aBBa) edge[red] (aBB);
   
  \path[->,densely dotted] (abA) edge[blue] (abAb);
  \path[->,densely dotted] (abA) edge[red] (abAA);
  \path[->,densely dotted] (abAB) edge[blue] (abA);
   
  \path[->,densely dotted] (aBA) edge[blue] (aBAb);
  \path[->,densely dotted] (aBA) edge[red] (aBAA);
  \path[->,densely dotted] (aBAB) edge[blue] (aBA);
   
  \path[->,densely dotted] (aba) edge[blue] (abab);
  \path[->,densely dotted] (abaa) edge[red] (aba);
  \path[->,densely dotted] (abaB) edge[blue] (aba);
   
  \path[->,densely dotted] (aBa) edge[blue] (aBab);
  \path[->,densely dotted] (aBaa) edge[red] (aBa);
  \path[->,densely dotted] (aBaB) edge[blue] (aBa);
  \end{scope}
   
\end{tikzpicture}
\caption{El dígraf de Cayley de $\Free_{\set{a,b}}$}
\label{fig: Sch(F_2)}
\end{figure}

Com veurem a la secció següent, el dígraf de Cayley dels grups lliures respecte una base, tindrà protagonisme en la teoria d'autòmats de Stallings. Concretament el subgraf definit a continuació.

\begin{defn}
Sigui $A$ un alfabet, i $a\in A^{\pm}$. Anomenem \defin{$a$-branca de Cayley} de $\Free[A]$ a 
la component connexa de $\cayley(\Free[A],A) \setmin \Edgii$ contenint $\bp$, on $\Edgii$ és el conjunt  d'arcs incidents a $\bp$ excepte \smash{$\bp\! \xarc{\,a\ }\! \bullet$} (i el seu invers); vegeu la \Cref{fig: branca}.
\end{defn}

\begin{figure}[h] 
\centering
  \begin{tikzpicture}[shorten >=1pt, node distance=1.2 and 1.2, on grid,auto,>=stealth',transform shape]
  
  \node[state,accepting] (bp) {};
  \node[blue] (l) [above left = 0.13 and 0.8 of 1] {\scriptsize{$a$}};

  \begin{scope}[rotate=180]
  \newcommand{\dx}{1.6}
  \newcommand{\dy}{1.6}
  \newcommand{\dxx}{\dx*0.6}
  \newcommand{\dxxx}{\dx*0.3}
  \newcommand{\dxxxx}{\dx*0.18}
  \newcommand{\dyy}{\dy*0.5}
  \newcommand{\dyyy}{\dy*0.3}
  \newcommand{\dyyyy}{\dy*0.18}

  \node[smallstate] (1) {};
  \node[smallstate] (a) [right = \dx of 1]{};
  \node[smallstate] (aa) [right = \dxx of a]{};
  \node[tinystate] (aaa) [right = \dxxx of aa]{};
  \node[emptystate] (aaaa) [right = \dxxxx of aaa]{};
   
  \node[tinystate] (ab) [above = \dyy of a]{};
  \node[tinystate] (aab) [above = \dyyy of aa]{};
  \node[emptystate] (aabb) [above = \dyyyy of aab]{};
  \node[emptystate] (aaab) [above = \dyyyy of aaa]{};
  \node[emptystate] (aaaab) [above = \dyyyy of aaaa]{};
   
  \node[tinystate] (aB) [below = \dyy of a]{};
  \node[tinystate] (aaB) [below = \dyyy of aa]{};
  \node[emptystate] (aaBB) [below = \dyyyy of aaB]{};
  \node[emptystate] (aaaB) [below = \dyyyy of aaa]{};
  \node[emptystate] (aaaaB) [below = \dyyyy of aaaa]{};
   
  \node[tinystate] (aba) [right = \dxxx of ab]{};
  \node[tinystate] (abA) [left = \dxxx of ab]{};
  \node[emptystate] (aaba) [right = \dxxxx of aab]{};
  \node[emptystate] (aabA) [left = \dxxxx of aab]{};

  \node[tinystate] (aBa) [right = \dxxx of aB]{};
  \node[tinystate] (aBA) [left = \dxxx of aB]{};
  \node[emptystate] (aaBa) [right = \dxxxx of aaB]{};
  \node[emptystate] (aaBA) [left = \dxxxx of aaB]{};

  \node[tinystate] (abb) [above = \dyyy of ab]{};
  \node[tinystate] (aBB) [below = \dyyy of aB]{};
   
  \node[emptystate] (abbA) [left = \dxxxx of abb]{};
  \node[emptystate] (abbb) [above = \dyyyy of abb]{};
  \node[emptystate] (abba) [right = \dxxxx of abb]{};
   
  \node[emptystate] (aBBA) [left = \dxxxx of aBB]{};
  \node[emptystate] (aBBB) [below = \dyyyy of aBB]{};
  \node[emptystate] (aBBa) [right = \dxxxx of aBB]{};
   
  \node[emptystate] (abAA) [left = \dxxxx of abA]{};
  \node[emptystate] (abAb) [above = \dyyyy of abA]{};
  \node[emptystate] (abAB) [below = \dxxxx of abA]{};
   
  \node[emptystate] (aBAA) [left = \dxxxx of aBA]{};
  \node[emptystate] (aBAb) [above = \dyyyy of aBA]{};
  \node[emptystate] (aBAB) [below = \dxxxx of aBA]{};
   
  \node[emptystate] (abaa) [right = \dxxxx of aba]{};
  \node[emptystate] (abab) [above = \dyyyy of aba]{};
  \node[emptystate] (abaB) [below = \dxxxx of aba]{};
   
  \node[emptystate] (aBaa) [right = \dxxxx of aBa]{};
  \node[emptystate] (aBab) [above = \dyyyy of aBa]{};
  \node[emptystate] (aBaB) [below = \dxxxx of aBa]{};
 
  \path[->] (ab) edge[red] (a);
  \path[->] (aab) edge[red] (aa);
  \path[->,densely dotted] (aabb) edge[red] (aab);
  \path[->,densely dotted] (aaab) edge[red] (aaa);
   
  \path[->] (a) edge[red] (aB);
  \path[->] (aa) edge[red] (aaB);
  \path[->,densely dotted] (aaB) edge[red] (aaBB);
  \path[->,densely dotted] (aaa) edge[red] (aaaB);
   
  \path[->] (abb) edge[red] (ab);
  \path[->] (aB) edge[red] (aBB);
   
  \path[->] (a) edge[blue] (1);
  \path[->] (aa) edge[blue] (a);
  \path[->] (aaa) edge[blue] (aa);
  \path[->,densely dotted] (aaaa) edge[blue] (aaa);
   
  \path[->] (ab) edge[blue] (abA);
  \path[->] (aba) edge[blue] (ab);
  \path[->] (aba) edge[blue] (ab);
   
  \path[->,densely dotted] (aab) edge[blue] (aabA);
  \path[->,densely dotted] (aaba) edge[blue] (aab);
   
  \path[->] (aB) edge[blue] (aBA);
  \path[->] (aBa) edge[blue] (aB);
  \path[->] (aBa) edge[blue] (aB);
   
  \path[->,densely dotted] (aaB) edge[blue] (aaBA);
  \path[->,densely dotted] (aaBa) edge[blue] (aaB);
   
  \path[->,densely dotted] (abbb) edge[red] (abb);
  \path[->,densely dotted] (abb) edge[blue] (abbA);
  \path[->,densely dotted] (abba) edge[blue] (abb);
   
  \path[->,densely dotted] (aBB) edge[red] (aBBB);
  \path[->,densely dotted] (aBB) edge[blue] (aBBA);
  \path[->,densely dotted] (aBBa) edge[blue] (aBB);
   
  \path[->,densely dotted] (abAb) edge[red] (abA);
  \path[->,densely dotted] (abA) edge[blue] (abAA);
  \path[->,densely dotted] (abA) edge[red] (abAB);
   
  \path[->,densely dotted] (aBAb) edge[red] (aBA);
  \path[->,densely dotted] (aBA) edge[blue] (aBAA);
  \path[->,densely dotted] (aBA) edge[red] (aBAB);
   
  \path[->,densely dotted] (abab) edge[red] (aba);
  \path[->,densely dotted] (abaa) edge[blue] (aba);
  \path[->,densely dotted] (aba) edge[red] (abaB);
   
  \path[->,densely dotted] (aBab) edge[red] (aBa);
  \path[->,densely dotted] (aBaa) edge[blue] (aBa);
  \path[->,densely dotted] (aBa) edge[red] (aBaB);
  \end{scope}
   
\end{tikzpicture}
\caption{La $(a^{-1})$-branca de Cayley de $\Free_{\set{a,b}}$}
\label{fig: branca}
\end{figure}

Dedicarem la resta d'aquest article a fer una primera aproximació a l'estudi del reticle de subgrups de $\Free[n]$, amb $n\geqslant 2$ finit. I, tal com intuíem més amunt, ens adonarem de seguida que, efectivament, és extremadament complicat. Per una banda, presenta certs comportaments que recorden patrons de l'àlgebra lineal i suggereixen que, potser, alguns aspectes no seran tan intricats; però gairebé sempre acaba apareixent algun comportament distintiu, sovint de tipus fractal, que el distancia radicalment d'ambients algebraics més benignes. 

Acabem aquesta secció proposant un parell de problemes sobre el reticle de subgrups de $\Free[2]=\Free[\{a,b\}]$ indicatius de la complexitat a la que ens referim. Es tracta de problemes molt naturals, anàlegs a problemes típics de l'àlgebra lineal, que en aquell context es resolen fàcilment plantejant i resolent sistemes d'equacions. Només cal endinsar-se una mica en ambdós problemes, com farem a continuació, per copsar la dificultat extra que suposa el fet d'estar tractant amb grups fortament no commutatius. Convidem el lector a intentar resoldre'ls abans de continuar llegint, per tal de percebre la complexitat a la que ens referim. Tot i que en un primer apropament poden semblar difícils de resoldre, veurem a la \Cref{sec: aplicacions} que hi ha una manera elegant, eficient, i sorprenent senzilla de tractar-los: usar la teoria dels autòmats de Stallings, que expliquem en detall a la \Cref{sec: Stallings bijection}.

\begin{exm}\label{ex: membership}
Sigui $\Free[2]$ el grup lliure sobre $A=\{a,b\}$ i considerem el subgrup $H=\langle v_1, v_2, v_3 \rangle\leqslant \Free[2]$ generat pels elements $v_1=baba^{-1}$, $v_2=aba^{-1}$, i $v_3=aba^2$. És cert que l'element $u=b^2aba^{-1}b^7a^{-2}b^{-1}a^2\in \Free[2]$ pertany a $H$? I, en cas afirmatiu, sabem expressar-lo com a producte de $v_1, v_2, v_3$ i els seus inversos?

Per certes paraules, decidir la pertinença a $H$ és fàcil. Per exemple, per veure que $a^3$ pertany a $H$, és suficient observar que $v_2^{-1}v_3= (ab^{-1}a^{-1})(aba^2)=a^3$. Un altre argument senzill ens permet deduir que l'element $a$ \emph{no} pertany a $H$: tant se val com multipliquem $v_1, v_2, v_3$ o els seus inversos, l'$a$-exponent (és a dir, el nombre total de $a$'s obtingut sumant les d'exponent positiu i restant les d'exponent negatiu) al resultat sempre serà múltiple de $3$. En efecte, els $a$-exponents dels generadors són $|v_1|_a=1-1=0$, $|v_2|_a=1-1=0$, i $|v_3|_a=1+2=3$; com que $|w_1 w_2|_a =|w_1|_a +|w_2|_a$ (ja que la possible cance\l.lació entre $w_1$ i $w_2$ elimina la mateixa quantitat de $a$'s positives que de $a$'s negatives), tots els elements del subgrup $H=\langle v_1, v_2, v_3\rangle$ han de tenir forçosament $a$-exponent múltiple de 3. I, com que $|a|_a =1$ no és múltiple de 3, deduïm que $a\not\in H$.

Però en el cas de l'element $u=b^2aba^{-1}b^7a^{-2}b^{-1}a^2$, l'$a$-exponent total és $|u|_a =0$ per tant, no podem descartar-lo usant l'argument anterior. Òbviament, això tampoc ens assegura que hi hagi de pertànyer: no tenim, a priori, cap indicació de com podem multiplicar entre sí les paraules $v_1, v_2, v_3$ i els seus inversos, per tal d'aconseguir $u$; ni de si això és realment possible, o no. Si els elements $v_1, v_2, v_3$ commutessin, el problema es reduiria a decidir si $u$ és de la forma $v_1^{\alpha}v_2^{\beta}v_3^{\gamma}$ amb $\alpha, \beta,\gamma\in \ZZ$, és a dir, si $u$ és una \emph{combinació lineal} de $v_1, v_2, v_3$ amb coeficients enters. Com es pot percebre, però, la no commutativitat dels elements $v_1, v_2, v_3$ complica molt el planteig i la resolució d'aquest problema.   

La resposta en aquest cas és afirmativa, però trobar un producte adequat (que doni $u$) ja no és tan fàcil a simple vista:
 \begin{align*}
&\phantom{\ =\ }v_1v_2^{-1}v_1 (v_1v_2^{-1})^7 v_3^{-1}v_2^{-1}v_3 =\\
&\,=\,baba^{-1}  (aba^{-1})^{-1} baba^{-1}  (baba^{-1} (aba^{-1})^{-1})^7  (aba^2)^{-1} (aba^{-1})^{-1} aba^2 \\
&\,=\,baba^{-1}  ab^{-1}a^{-1} baba^{-1}  (baba^{-1}  ab^{-1}a^{-1})^7  a^{-2}b^{-1}a^{-1}  ab^{-1}a^{-1}  aba^2 \\
&\,=\,b baba^{-1}  b^7  a^{-2}b^{-1} a^2 \,=\,b^2aba^{-1}b^7a^{-2}b^{-1}a^2\,=\,u\,.
 \end{align*}
Evidentment, fer els càlculs per comprovar si un determinat producte dels generadors dóna $u$ és fàcil. Però trobar, o més encara, decidir l'existència d'un producte de $v_1^{\pm 1}, v_2^{\pm 1}, v_3^{\pm 1}$ que doni la paraula candidata $u$, no ho sembla tant.\footnote{Almenys a (la majoria de) els humans (vegeu \url{https://ca.wikipedia.org/wiki/P_versus_NP}).} És fàcil adonar-se que la complicació deriva de com de recargolada sigui l'eventual expressió \emph{(de llargada finita però arbitràriament gran)} del candidat $u$ en termes dels generadors. A la secció \Cref{ssec: MP} explicarem una manera sistemàtica de resoldre aquest problema.

Preguntes afins amb evident importància algebraica són: és aquesta l'única manera d'obtenir $u$ a partir de $v_1, v_2, v_3$? Si n'hi ha d'altres, podem descriure-les totes? Podem calcular la paraula més curta (potser no única) que doni $u$?\ldots
\end{exm}

Observeu que el problema anàleg en àlgebra lineal seria del tipus següent: ``pertany el vector $(1,2,3)$ al subespai vectorial $\big\langle (-1,2,1),\, (0,1,-3)\big\rangle$ de $\mathbb{R}^3$?''. Per respondre, només hem de mirar si el sistema d'equacions $(1,2,3)=\alpha (-1,2,1)+\beta (0,1,-3)$ té o no solucions a $\mathbb{R}$. En l'ambient no commutatiu en què ens trobem no coneixem, de moment, cap eina que pugui fer un paper semblant al que fan els sistemes d'equacions en l'àlgebra lineal. Com veurem a continuació, en el cas dels grups lliures, els autòmats de Stallings jugaran aquest paper. 

Per acabar la secció, us proposem un altre problema involucrant interseccions de subgrups del grup lliure.

\begin{exm}\label{exe: interseccio}
Sigui $\Free[2]$ el grup lliure sobre $A=\{a,b\}$, i considerem els subgrups $H=\langle u_1, u_2, u_3\rangle \leqslant \Free[2]$ i $K=\langle v_1, v_2, v_3\rangle \leqslant \Free[2]$ donats per les paraules $u_1=b$, $u_2=a^3$, $u_3=a^{-1}bab^{-1}a$, i $v_1=ab$, $v_2=a^3$, i $v_3=a^{-1}ba$. Com veurem al \Cref{cor: ranks of subgroups}, els subgrups del grup lliure no tenen perquè ser finitament generats. Ho és la intersecció  $H\cap K$? i, en tal cas, podeu calcular-ne un conjunt de generadors?

A simple vista ja es veu que l'element $a^3$ és present a ambdós subgrups: ${a^3=u_2\in H}$, i ${a^3=v_2\in K}$. També se'n veu fàcilment un altre:
 \begin{equation*}
H\ni u_1^{-1} u_2 u_1
\,=\, b^{-1}a^3 b
\,=\, v_1^{-1}v_2 v_1 \in K. 
 \end{equation*}
I amb una mica més d'imaginació (o de treball de càlcul) aquí tenim un tercer element de la intersecció: 
 \begin{equation*}
H\ni u_3^3 =a^{-1}ba^3b^{-1}a = v_3 v_2 v_3^{-1} \in K. 
 \end{equation*}
Com que $H\cap K$ és un subgrup de $\Free[2]$, deduïm que $\langle a^3,\, b^{-1}a^3b,\, a^{-1}ba^3b^{-1}a\rangle$ està contingut a ${H\cap K}$. Seguir trobant nous elements a simple vista no sembla fàcil; però la pregunta contrària encara sembla més complicada: és cert que $H\cap K=\langle a^3,\, b^{-1}a^3b,\, a^{-1}ba^3b^{-1}a\rangle$? o cal trobar i afegir-hi més generadors?
\end{exm}

El lector interessat pot provar de respondre totes aquestes preguntes per sí sol. Nosaltres, de moment, ho deixarem aquí, i tornarem a tractar aquests exemples a la \Cref{sec: aplicacions}, on ja disposarem d'eines adequades per respondre totes aquestes preguntes\ldots\ i més!

\section{Autòmats de Stallings} \label{sec: Stallings bijection}

\begin{defn} \label{def: automat}
Sigui $A$ un alfabet. Un \defin{$A$-autòmat} $\Ati$ (o simplement \emph{autòmat}, si l'alfabet és clar pel context) és un graf dirigit amb els arcs etiquetats en $A$ i un vèrtex distingit; més formalment, és una tupla $\Ati =(\Verts, \Edgs, \init, \term, \ell, \bp)$, on $\Verts$ i $\Edgs$ són conjunts, $\init,\term \colon \Edgs \to \Verts$  i $\ell\colon \Edgs \to A$ són funcions, i $\bp \in \Verts$. Els conjunts $\Verts =\Verts \Ati$ i $\Edgs =\Edgs \Ati$ s'anomenen conjunt de \defin{vèrtexs} i d'\defin{arcs} (o \defin{arestes dirigides}) de~$\Ati$, respectivament, i les funcions $\init,\term$ i $\ell$ assignen a cada arc $\edgi \in \Edgs$ el seu \defin{origen} $\init(\edgi) \in \Verts$, el seu \defin{final} $\term(\edgi)\in \Verts$, i la seva \defin{etiqueta} $\ell(\edgi) \in A$, respectivament. Si $\init \edgi =\verti$, $\term\edgi =\vertii$ i $\ell(\edgi) =a$, aleshores diem que $\edgi$ és un \defin{$a$-arc de $\verti$ a $\vertii$}, i escrivim $\smash{\edgi \equiv \verti \xarc{\,a\ }\vertii}$. El vèrtex distingit~$\bp$ s'anomena \defin{vèrtex base} o \defin{punt base} del $A$-autòmat. El \defin{dígraf subjacent} d'un $A$-autòmat $\Ati$ és el dígraf resultat d'ignorar-ne la funció $\ell$ i el punt base $\bp$.  
\end{defn}

Observeu que, segons la nostra definició, un autòmat  no és necessàriament connex, i admet la possibilitat tant d'arcs amb el mateix origen i final (\defin{llaços}), com també de múltiples arcs amb el mateix origen i el mateix final (\defin{arcs para\l.lels}).

\begin{defn}
Sigui $\Ati$ un $A$-autòmat. Un vèrtex $\verti$ de $\Ati$ s'anomena \defin{saturat} si per a cada lletra $a \in A$ hi ha (almenys) un $a$-arc amb origen $\verti$. En cas contrari, diem que $\verti$ és insaturat (o \defin{$a$-deficient}, si volem referir-nos a l'etiqueta que falta). El \defin{$a$-dèficit} de $\Ati$, designat per $\defc[a]{\Ati}$, és el nombre (cardinal) de vèrtexs $a$-deficients a $\Ati$. Un $A$-autòmat $\Ati$ s'anomena \defin{saturat} si tots els seus vèrtexs ho són, \ie si $\defc[a]{\Ati}=0$ per a tot $a\in A$. És clar que tot $A$-autòmat saturat és $\card A$-regular. 
\end{defn}

\begin{defn}
Un \defin{camí (dirigit)} en un $A$-autòmat $\Ati$ és una seqüència alternada finita $\walki =\verti_0 \edgi_1 \verti_1 \cdots \edgi_{l} \verti_{l}$, on $\verti_i \in \Verts \Ati $, $\edgi_i \in \Edgs \Ati $, $\init{\edgi_i} = \verti_{i-1}$, i $\term{\edgi_i}=\verti_{i}$ per a $i=1,\dots,l$. Aleshores, $\verti_0$ i $\verti_l$ s'anomenen \defin{origen} i \defin{final} de $\walki$ respectivament, i escrivim $\verti_0 = \init(\walki)$ i $\verti_l = \term(\walki)$. Diem que $\walki$ és un camí de  $\verti_0$ a $\verti_l$, i ho designem per $\walki \colon \smash{\verti_0 \xwalk{\ } \verti_l}$. Si l'origen i el final de $\walki$ coincideixen, diem que $\walki$ és un \defin{camí tancat}.
Un camí tancat de $\verti$ a $\verti$ s'anomena \defin{$\verti$-camí}. La \defin{longitud} d'un camí $\walki$, designada per $|\walki|$, és el nombre d'arcs a $\walki$ (comptant possibles repeticions). Els camins de longitud $0$ s'anomenen \defin{camins trivials} i corresponen als vèrtexs de $\Ati$. 
\end{defn}

Recordem que $A^{\pm} = A \sqcup A^{-1}$, on $A^{-1}$ és el conjunt d'inversos formals de $A$.

\begin{defn} \label{def: automat involutiu}
Un \defin{$A$-autòmat involutiu} és un $A^{\pm}$-autòmat amb una involució ${\edgi \mapsto \edgi^{-1}}$ als seus arcs tal que si $\smash{\edgi \equiv \verti \xarc{\,a\ }\vertii}$ aleshores $\smash{\edgi^{-1} \equiv \verti \xcra{\,a^{\text{-}1}\!} \vertii}$. Aleshores, diem que (l'arc etiquetat) $\edgi^{-1}$ és l'\defin{invers} de $\edgi$. Observeu que $(\edgi^{-1})^{-1} = \edgi$; és a dir, en un autòmat involutiu, els arcs etiquetats apareixen aparellats amb els seus respectius inversos (vegeu la \Cref{fig: involutive arc}).
\begin{figure}[H]
\centering
\begin{tikzpicture}[shorten >=3pt, node distance=.3cm and 2cm, on grid,auto,>=stealth']
   \node[state] (0) {};
   \node[state] (1) [right = of 0] {};
   \path[->]
        ([yshift=0.3ex]0.east) edge[]
            node[pos=0.5,above=-.2mm] {$a$}
            ([yshift=0.3ex]1)
   ([yshift=-0.3ex]1.west) edge[dashed]
            node[pos=0.5,below=-.2mm] {$a^{-1}$}
            ([yshift=-0.3ex]0);
\end{tikzpicture}
\vspace{-5pt}
\caption{Un arc en un autòmat involutiu amb el seu invers puntejat}
\label{fig: involutive arc}
\end{figure}
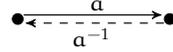
\end{defn}

\begin{defn} \label{def: positive part}
Un arc en un $A$-autòmat involutiu $\Ati$ es diu \defin{positiu} (\resp \defin{negatiu}) si la seva etiqueta pertany a $A$ (\resp $A^{-1}$). S'anomena \defin{part positiva} de $\Ati$, designada per $\Ati^{+}$, l'autòmat obtingut després d'eliminar de $\Ati$ tots els arcs negatius. És clar que tot autòmat involutiu queda determinat per la seva part positiva, que habitualment farem servir per descriure l'autòmat, amb el conveni tàcit que els arcs positius $\smash{\verti \xarc{\,a\ }\vertii}$ es poden travessar cap enrere (\ie de $\vertii$ a $\verti$) llegint la lletra inversa $a^{-1}$. Observeu que, identificant els parells d'arcs mútuament inversos (en una \defin{aresta}, no dirigida), recuperem la noció estàndard de graf \emph{no dirigit}, que anomenarem \defin{graf (no dirigit) subjacent} de $\Ati$. Farem servir aquesta correspondència per traslladar a l'àmbit dels autòmats algunes nocions bàsiques de teoria de grafs. Per exemple, direm que un autòmat $\Ati$ és \defin{connex} si ho és el seu graf subjacent,
direm que $T$ és un \defin{arbre d'expansió}\footnote{\emph{spanning tree}, en anglès.} de $\Ati$ si ho és del seu graf subjacent, i definirem el \defin{grau} d'un vèrtex $\verti$, designat per $\deg(\verti)$, com el nombre d'arestes incidents a $\verti$ en el graf subjacent, i el \defin{rang} de $\Ati$, designat per $\rk(\Gamma)$, com el rang del seu graf subjacent (vegeu \Cref{def: fundamental group}). En particular, si $\Ati$ és un autòmat involutiu, connex i finit, aleshores $\rk(\Ati) = 1 - \card \Verts\Ati + \card \Edgs\Ati^{+}$.

Si $\Ati$ és un $A$-autòmat involutiu, aleshores el $A^{\pm}$-etiquetatge de $\Ati$ s'estén de forma natural a un $\IM{A}$-etiquetatge dels camins, tot concatenant les corresponents etiquetes: si $\walki =\verti_0 \edgi_1 \verti_1 \cdots \edgi_{l} \verti_{l}$ és un camí, definim la seva \defin{etiqueta}, i la seva \defin{etiqueta reduïda} com $\ell(\walki) =\ell(\edgi_1) \cdots \ell(\edgi_{l})\in \IM{A}$, i  $\rlab(\walki) =\red{\lab(\walki)}\in \Free[A]$, respectivament. En ambdós casos, s'entén que l'etiqueta d'un camí trivial és l'element trivial $1$. Si $\lab(\walki) =w\in \IM{A}$ diem que el camí $\walki$ \defin{llegeix} $w$, o que la paraula $w$ \defin{etiqueta} el camí $\walki$, i escrivim $\walki\colon \verti_0 \xwalk{_{\scriptstyle{w}}} \verti_l$.

Es diu que un camí presenta \defin{retrocés}\footnote{\emph{backtracking}, en anglès.} si té dos arcs successius inversos l'un de l'altre. Un camí sense retrocés es diu \defin{camí reduït}. És obvi que si $\lab(\walki)$ és una paraula reduïda, aleshores el camí $\walki$ ha de ser reduït, però el recíproc no és cert en general, ja que es poden donar les situacions (no deterministes) de la \Cref{fig: no determinista}. Observeu també que si $T$ és un arbre d'expansió d'un $A$-autòmat connex $\Ati$, aleshores per a cada parell de vèrtexs $\verti,\vertii$ a $\Ati$ existeix un únic camí \emph{reduït} de $\verti$ a $\vertii$ usant només arcs de $T$; l'anomenem  $T$-geodèsica de $\verti$ a $\vertii$, i el designem per $\walki_{_{T}}[\verti,\vertii]$, o simplement per \smash{$\verti \xwalk{\scriptscriptstyle{T}} \vertii$}.
\begin{figure}[H]
\centering
  \begin{tikzpicture}[shorten >=1pt, node distance=.5cm and 1.5cm, on grid,auto,>=stealth']
   \node[state,semithick, fill=gray!20, inner sep=2pt, minimum size = 10pt] (1) {$\scriptstyle{\verti}$};
   
   \begin{scope}
   \node[state] (2) [above right = of 1] {};
   \node[state] (3) [below right = of 1] {};

   \path[->]
        (1) edge[]
            node[pos=0.45,above] {$a$}
            (2)
            edge[]
            node[pos=0.45,below] {$a$}
            (3);
\end{scope}

\begin{scope}[xshift=3cm] 
\node[state,semithick, fill=gray!20, inner sep=2pt, minimum size = 10pt] (1) {$\scriptstyle{\verti}$};
   \node[state] (2) [right = 1.5 of 1] {};

   \path[->]
        (1) edge[loop above,min distance=9mm,in= 90-35,out=90+35]
            node[] {\scriptsize{$a$}}
            (1)
            edge
            node[pos=0.5,below] {$a$}
            (2);
\end{scope}

\begin{scope}[xshift=5.75cm] 
\node[state,semithick, fill=gray!20, inner sep=2pt, minimum size = 10pt] (1) {$\scriptstyle{\verti}$};
   \node[state] (2) [right = of 1] {};

   \path[->]
        (1) edge[bend right]
            node[pos=0.5,below] {$a$}
            (2)
        (1) edge[bend left]
            node[pos=0.5,above] {$a$}
            (2);
\end{scope}

\begin{scope}[xshift=9cm] 
\node[state,semithick, fill=gray!20, inner sep=2pt, minimum size = 10pt] (1) {$\scriptstyle{\verti}$};

   \path[->]
        (1) edge[loop right,min distance=9mm,in=35,out=-35]
            node[] {\scriptsize{$a$}}
            (1)
             edge[loop left,min distance=9mm,in=180-35,out=180+35]
            node[] {\scriptsize{$a$}}
            (1);
\end{scope}
\end{tikzpicture}
\caption{Situacions no deterministes al vèrtex $\verti$}
\label{fig: no determinista}
\end{figure}

Si $\walki =\verti_0 \edgi_1 \verti_1 \cdots \edgi_{l} \verti_{l}$ és un camí en un $A$-autòmat involutiu $\Ati$ (llegint $\lab(\walki)= \lab(\edgi_1) \cdots \lab(\edgi_l)$), aleshores s'anomena \defin{camí invers} de $\walki$ a $\walki^{-1} =\verti_l \edgi_l^{-1} \verti_{l-1} \cdots \edgi_{1}^{-1} \verti_{0}$ (llegint $\lab(\walki)^{-1}=\lab(\edgi_l)^{-1}\cdots\lab(\edgi_0)^{-1}$). És clar que, també, $\rlab(\walki^{-1})=\rlab(\walki)^{-1}$. 
\end{defn}

És el moment de començar a relacionar els $A$-autòmats amb els grups. Potser la manera més natural de fer-ho és través de la idea topològica de grup fonamental d'un graf. Com veurem, si admetem etiquetes a les arestes, aquesta idea desemboca en la de subgrup reconegut, que, com es mostra a continuació, és un subgrup concret de $\Free[A]$, enlloc d'un grup abstracte. 

\begin{defn} \label{def: fundamental group}
Siguin $\Ati$ un graf no dirigit i connex, i $\verti\in \Verts\Ati$. El \defin{grup fonamental de $\Ati$ al vèrtex $\verti$}, designat per $\pi_{\verti}(\Ati)$, és el conjunt de classes de $\verti$-camins a $\Ati$ mòdul retrocés, o mòdul reducció de camins (que, en el cas dels grafs vistos com a espais topològics, equival a dir mòdul homotopia), juntament amb l'operació de concatenació de $\verti$-camins. És ben sabut que, independentment del vèrtex $\verti$ triat, $\pi_{\verti}(\Ati)$ és un grup lliure de rang igual al nombre (cardinal) d'arestes fora d'un arbre d'expansió $T$ de $\Ati$; a aquest valor, que no depèn de l'arbre $T$, l'anomenem el \defin{rang (gràfic)} de $\Ati$, designat per $\rk (\Ati)$; és fàcil veure que si $\Ati$ és finit, aleshores 
 \begin{equation}\label{eq: rang-degree}
\rk (\Ati)-1
\,=\,
\card\Edgs \Ati -\card\Verts \Ati 
\,=\,
\frac{1}{2}\sum\nolimits_{\vertii\in \Verts\Ati} \left( \deg(\vertii)-2\right),
  \end{equation}
on $\card \Verts \Ati$ és el nombre de vèrtexs i $\card \Edgs\Ati$ el nombre d'arestes (no dirigides) de $\Ati$. 
\end{defn}

\begin{defn} \label{def: subgrup reconegut}
Sigui $\Ati$ un $A$-autòmat involutiu i connex, i $\verti \in \Verts\Ati$.
És fàcil veure que 
el conjunt d'etiquetes reduïdes de $\verti$-camins a $\Ati$,
 \begin{equation} \label{eq: rk graph}
\gen{\Ati}_{\verti}=\{\rlab(\walki )\st \walki\colon \verti \xwalk{\,\,\,} \verti \}\,,
 \end{equation}
és un subgrup de $\Free[A]$; l'anomenem \defin{subgrup reconegut per $\Ati$ al vèrtex $\verti$}. Al subgrup~${\gen{\Ati}_{\bp}}$ el designem simplement per $\gen{\Ati}$.
\end{defn}

\begin{rem}
En general, el subgrup reconegut per un $A$-autòmat $\Ati$ i el grup fonamental de (el subgraf subjacent de) $\Ati$ en un vèrtex $\verti$ són conceptes diferents, però  estretament lligats per l'aplicació \emph{llegir etiquetes reduïdes}: 
 \begin{equation} \label{eq: pi -> gen}
\begin{array}{rcl}
\widetilde{\ell}_{\Ati} \colon \pi_{\verti} (\Ati) & \onto & \gen{\Ati}_{\verti}\leqslant\Free[A]. \\ \textnormal{$[\walki]$} & \mapsto & \rlab(\walki) 
\end{array}
 \end{equation}
És clar que $\widetilde{\ell}_{\Ati}$ està ben definit i és un homomorfisme de grups exhaustiu \emph{i no injectiu, en general}. Més endavant veurem una condició sobre $\Ati$, sota la qual $\widetilde{\ell}_{\Ati}$ serà injectiu.
\end{rem}

La \Cref{def: subgrup reconegut} (prenent $\verti = \bp$) determina l'aplicació següent, que relaciona els $A$-autòmats amb el reticle de subgrups de $\Free[A]$: 
 \begin{equation}\label{eq: Ati ->> <Ati>}
\begin{array}{rcl}
\Set{\text{(classes d'isomorfia de) $A$-autòmats}} & \to & \Set{\text{subgrups de $\Free[A]$}}. \\ \Ati & \mapsto & \gen{\Ati}
\end{array}
 \end{equation}
És fàcil veure que aquesta aplicació és exhaustiva; és a dir, tot subgrup $H\leqslant \Free[A]$ és el subgrup reconegut per algun $A$-autòmat. En efecte, si $S=\set{w_i}_{i\in I} \subseteq \Free[A]$ és un conjunt de paraules reduïdes generant~$H$, aleshores considerem, per a cada $w_i =a_{i_1}a_{i_2}\cdots a_{i_l}\in S$ ($a_j \in A^{\pm}$), el cicle orientat llegint~$w_i$ (ó $w_i^{-1}$ en direcció contrària); anomenat \defin{pètal} associat a $w_i$, i designat per $\flower(w_i)$ (vegeu la~\Cref{fig: petal}).
\begin{figure}[H]
\centering
\begin{tikzpicture}[shorten >=1pt, node distance=0.2 and 1.5, on grid,auto,>=stealth']
    \node[state, accepting] (0) {};
    \node[state] (a) [above right =  of 0] {};
    \node[state] (1) [above right =  of a] {};
    \node[state] (b) [right =  of 1] {};
    \node[state] (d) [below right =  of 0] {};
    \node[state] (4) [below right =  of d] {};
    \node[state] (c) [right =  of 4] {};

    \path[->]
        (0) edge[]
            node[above] {$a_{i_1}$}
            (a);

     \path[->]
        (a) edge[]
            node[above left] {\scriptsize{$a_{i_2}$}}
            (1);

     \path[->]
        (1) edge[]
            node[above] {$a_{i_3}$}
            (b);

     \path[->,dashed]
        (b) edge[bend left,out=90,in=90,min distance=10mm]
            (c);

    \path[->]
        (c) edge[]
            node[midway,below] {$a_{i_{l-2}}$}
            (4);

    \path[->]
        (4) edge[]
            node[below] {$a_{i_{l-1}}$}
            (d);

    \path[->]
        (d) edge[]
            node[below] {$a_{i_l}$}
            (0);
\end{tikzpicture}
\vspace{-5pt}
\caption{El pètal $\flower(w_i)$}
\label{fig: petal}
 \end{figure}
\noindent A continuació definim l'\defin{autòmat flor} $\flower(S)$, com l'autòmat obtingut després d'identificar els punts base dels pètals associats als elements de $S$ (vegeu la~\Cref{fig: flower}).
 \vspace{-20pt}
\begin{figure}[H]
\centering
\begin{tikzpicture}[shorten >=1pt, node distance=2cm and 2cm, on grid,auto,>=stealth',
decoration={snake, segment length=2mm, amplitude=0.5mm,post length=1.5mm}]
  \node[state,accepting] (1) {};
  \path[->,thick]
        (1) edge[loop,out=160,in=200,looseness=8,min distance=25mm,snake it]
            node[left=0.2] {$w_1$}
            (1);
            (1);
  \path[->,thick]
        (1) edge[loop,out=140,in=100,looseness=8,min distance=25mm,snake it]
            node[left=0.15] {$w_2$}
            (1);
  \path[->,thick]
        (1) edge[loop,out=20,in=-20,looseness=8,min distance=25mm,snake it]
            node[right=0.2] {$w_p$}
            (1);
\foreach \n [count=\count from 0] in {1,...,3}{
      \node[dot] (1\n) at ($(1)+(45+\count*15:0.75cm)$) {};}
\end{tikzpicture}
\vspace{-15pt}
\caption{L'autòmat flor $\flower(w_1,w_2,\ldots,w_p)$}
\label{fig: flower}
\end{figure}
És clar que les etiquetes reduïdes dels $\bp$-camins de $\flower(S)$ descriuen precisament el subgrup $H$; \ie $\gen{\flower(S)}=\gen{S}=H\leqslant \Free[A]$. Per tant, l'aplicació~\eqref{eq: Ati ->> <Ati>} és exhaustiva. Observeu també que \eqref{eq: Ati ->> <Ati>} està lluny de ser injectiva ja que, per exemple, diferents famílies de generadors per a $H$ proporcionen diferents autòmats flor, tots preimatge de $H$ per~\eqref{eq: Ati ->> <Ati>}. 

El resultat clau de Stallings a \cite{stallings_topology_1983} ens diu que es pot eliminar tota la redundància de la representació donada per~\eqref{eq: Ati ->> <Ati>} (assolint un únic representant $\Ati$ per a cada subgrup $H\leqslant \Free[A]$) simplement imposant un parell de condicions força naturals sobre els autòmats involucrats: ésser determinista, i ésser cor.

\begin{defn}
Un $A$-autòmat $\Ati$ es diu \defin{determinista} si de cap dels seus vèrtexs surten dos arcs diferents amb la mateixa etiqueta; és a dir, si per a cada vèrtex $\verti \in \Verts \Ati$, i per a cada parell d'arcs $\edgi,\edgi'$ sortint de $\verti$, $\lab(\edgi) = \lab(\edgi')$ implica $\edgi =\edgi'$. (Noteu que un autòmat involutiu és determinista si i només si de cap vèrtex surten i a cap vèrtex arriben arcs diferents amb la mateixa etiqueta positiva.)
\end{defn}

\begin{defn}
Un autòmat involutiu es diu \defin{cor}\footnote{\defin{core}, en anglès.} si tot vèrtex (i, per tant, tot arc) apareix en algun $\bp$-camí reduït. El \defin{cor d'un autòmat involutiu} $\Ati$, designat per $\core(\Ati)$, és el màxim subautòmat cor de $\Ati$ contenint el punt base. És clar que $\core (\Ati)$ és un subautòmat induït\footnote{Un subautòmat
és \defin{induit} si conté tots els arcs (de l'autòmat original) incidents als seus vèrtexs.} de $\Ati$ (de fet, de la component connexa de $\Ati$ contenint~$\bp$), connex, i que reconeix el mateix subgrup, \ie $\gen{\core (\Ati)}=\gen{\Ati}$.

Observeu que $\core (\Ati)$ pot tenir encara un vèrtex de grau $1$ que, en cas d'existir, ha de ser necessàriament el punt base a l'\emph{extrem d'un `pèl' penjant de $\core(\Ati)$}. Eliminant aquest eventual pèl obtenim el que anomenem \defin{cor restringit} de $\Ati$, designat per $\core^*(\Ati)$ (formalment, el digraf etiquetat obtingut després d'eliminar successivament tots el vèrtexs de grau $1$ de $\core(\Ati)$ i ignorar el vèrtex base).
\end{defn}

\begin{rem} \label{rem: <Ati> = <core> = <core*>}
És fàcil veure que $\rk (\core^* (\Ati)) = \rk (\core (\Ati)) = \rk (\Ati) $ (vegeu la~\Cref{ssec: SIP}), i que  $\gen{\core^*(\Ati)}_{\verti}$ és conjugat a $\gen{\core (\Ati)} = \gen{\Ati}$, per a tot $\verti \in \Verts (\core^*(\Ati))$ (vegeu la~\Cref{ssec: conj and normal}).
\end{rem}

\begin{defn}
Un $A$-autòmat involutiu es diu \defin{reduït} si és determinista i cor. 
\end{defn}

\begin{exm}
    Principals variants d'autòmat involutiu i determinista:

    \begin{figure}[H]
        \centering
    \begin{tikzpicture}[baseline=(0.base),shorten >=1pt, node distance= 0.5 cm and 1cm, on grid,auto,>=stealth']
        
        \newcommand{\dx}{1.2}
        \newcommand{\dy}{1.2}
        \node[state, accepting] (1) {};
    
        \node[state, right = \dx of 1] (2) {};    
        \node[state, above right = \dy/2 and \dx of 2] (21) {};
        \node[state, above right = \dy/3 and \dx/1.5 of 21] (211) {};
        \node[state, below right = \dy/2 and \dx of 2] (22) {};
        \node[below right = \dy and \dx/3 of 2] (st) {$\Ati$};
        
                
        \path[->]
            (1) edge[red]
                    node[pos=0.5,above=-.3mm] {$b$}
                (2);
                
        \path[->]
            (2) edge[blue]
                    node[pos=0.5,above left=-0.05] {$a$}
                (21);
                
        \path[->]
            (21) edge[red]
                    node[pos=0.5,above=-.3mm] {}
                (211);
                
        \path[->]
            (21) edge[blue, bend right]
                (22);
                
        \path[->]
            (22) edge[red, bend right]
                (21);
                
        \path[->]
            (22) edge[blue]
                (2);
                    
    \end{tikzpicture}
\qquad
\begin{tikzpicture}[baseline=(0.base),shorten >=1pt, node distance= 0.5 cm and 1cm, on grid,auto,>=stealth']
        
        \newcommand{\dx}{1.2}
        \newcommand{\dy}{1.2}
        \node[state, accepting] (1) {};
    
        \node[state, right = \dx of 1] (2) {};    
        \node[state, above right = \dy/2 and \dx of 2] (21) {};
        \node[state, below right = \dy/2 and \dx of 2] (22) {};
        \node[below = 1*\dy of 2] (st) {$\core(\Ati)$};
        
                
        \path[->]
            (1) edge[red]
                (2);
                
        \path[->]
            (2) edge[blue]
                (21);
                
        \path[->]
            (21) edge[blue, bend right]
                (22);
                
        \path[->]
            (22) edge[red, bend right]
                (21);
                
        \path[->]
            (22) edge[blue]
                (2);
                    
        \end{tikzpicture}
            \qquad \quad
        \begin{tikzpicture}[baseline=(0.base),shorten >=1pt, node distance= 0.5 cm and 1cm, on grid,auto,>=stealth']
        
        \newcommand{\dx}{1.2}
        \newcommand{\dy}{1.2}
    
        \node[state] (2) {};
        \node[below right = 1*\dy and 0.5*\dx of 2] (st) {$\core^*(\Ati)$};
        \node[state, above right = \dy/2 and \dx of 2] (21) {};
        \node[state, below right = \dy/2 and \dx of 2] (22) {};
                
        \path[->]
            (2) edge[blue]
                (21);
                
        \path[->]
            (21) edge[blue, bend right]
                (22);
                
        \path[->]
            (22) edge[red, bend right]
                (21);
                
        \path[->]
            (22) edge[blue]
                (2);
                
        \end{tikzpicture}
        \hspace{20pt}
        
        \label{fig: Stallings variants}
        \caption{D'esquerra a dreta, un autòmat determinista no reduït $\Ati$ (amb dos pèls), el seu cor $\core(\Ati)$, i el seu cor restringit $\core^*(\Ati)$.}
    \end{figure}
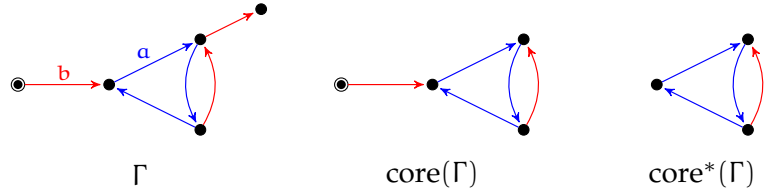
    \end{exm}

Per veure que cada fibra de \eqref{eq: Ati ->> <Ati>} admet un únic representant reduït adaptarem al nostre context un resultat clàssic de teoria de llenguatges.

\begin{defn}
Siguin $\Ati =(\Verts, \Edgs, \init,\term,\lab,\bp)$ i $\Ati'=(\Verts', \Edgs', \init',\term',\lab',\bp')$ $A$-autòmats. Un \defin{homomorfisme (de $A$-autòmats)} de $\Ati$ a $\Ati'$ és una funció $\theta \colon \Verts \to \Verts'$ que envia el punt base al punt base ($\bp \theta =\bp'$) i, per a cada $\verti,\vertii \in \Verts$ i cada $a\in A$, si $\verti \xarc{a\,} \vertii$ aleshores $\verti \theta \xarc{a\,} \vertii \theta$. Habitualment, abusem lleugerament del llenguatge i el designem per $\theta \colon \Ati \to \Ati'$. 
\end{defn}

\begin{prop}
Siguin $\Ati$ i $\Ati'$ $A$-autòmats reduïts. Aleshores, $\gen{\Ati} \leqslant \gen{\Ati'}$ si i només si existeix un homomorfisme d'autòmats $\Ati \to \Ati'$ i, en aquest cas, l'homomorfisme és únic.
\end{prop}

\begin{proof}
La unicitat és conseqüència del determinisme de $\Ati'$. En efecte, suposem que $\theta_1,\theta_2 \colon \Ati \to \Ati'$ són homomorfismes, i $\verti \in \Verts \Ati$ un vèrtex tal que $\verti \theta_1 \neq \verti \theta_2$. Aleshores les imatges per $\theta_1$ i $\theta_2$ de qualsevol camí $\walki \colon \bp \xwalk{\ } \verti$ a $\Ati$ seran camins a $\Ati'$ amb el mateix origen, $\bp'$, i la mateixa etiqueta, $\lab(\walki)$, però diferent final, $\verti \theta_1 \neq \verti \theta_2$, en contradicció amb el determinisme de $\Ati'$.

$[\Leftarrow]$ la implicació cap a l'esquerra és certa sense cap hipòtesis addicional sobre $\Ati, \Ati'$: si $\theta \colon \Ati \to \Ati$ és un homomorfisme d'autòmats, aleshores la imatge de qualsevol $\bp$-camí a $\Ati$ és un $\bp'$-camí a $\Ati'$ amb la mateixa etiqueta, i per tant, la mateixa etiqueta reduïda; per tant, $\gen{\Ati} \leqslant \gen{\Ati '}$.

$[\Rightarrow]$ Suposem ara que $\Ati$ i $\Ati'$ són $A$-autòmats reduïts amb $\gen{\Ati} \leqslant \gen{\Ati'}$ i anem a construir un homomorfisme $\theta \colon \Ati \to \Ati'$. Òbviament, prenem $\bp \theta =\bp'$. Per a tot vèrtex $\verti \neq \bp$, considerem un $\bp$-camí reduït $\walki$ a $\Gamma$ que passi per $\verti$:
 \begin{equation*}
\walki \colon \bp \xwalk{u} \verti \xwalk{v} \bp
 \end{equation*}
(sabem que n'existeix almenys un perquè $\Ati$ és cor). A més, com que $\Ati$ és també determinista, tenim que tant $u$ com $v$ són paraules reduïdes, i sense cance\l.lació entre elles. En particular, $uv \in \gen{\Ati}$ i per tant $uv \in \gen{\Ati'}$. Això vol dir que $uv$ és l'etiqueta reduïda d'un $\bp'$-camí a $\Ati'$ i, com que $\Ati'$ és determinista, $uv$ és també l'etiqueta d'un $\bp'$-camí reduït $\walki'$ a $\Ati'$, que podem descompondre com
 \begin{equation*}
\walki' \colon \bp' \xwalk{u} \verti' \xwalk{v} \bp' \,.
 \end{equation*}
Definim $\verti\theta =\verti'$. Per veure que està ben definit, sigui $\bp \xwalk{w} \verti \xwalk{z} \bp$ un altre $\bp$-camí reduït de $\Ati$ passant per $\verti$, considerem $\bp' \xwalk{w} \verti'' \xwalk{z} \bp'$ el corresponent $\bp'$-camí reduït a $\Gamma'$, i vegem que $\verti'=\verti''$. Observeu que $\bp \xwalk{_{\scriptstyle{u}}} \verti \xwalk{_{\scriptstyle{z}}} \bp$ és un \mbox{$\bp$-camí} a $\Gamma$ \emph{amb possible retrocés a $\verti$}; si anomenem $\vertii$ el vèrtex on acaba aquest retrocés, aleshores tenim
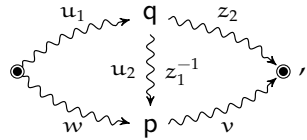
\begin{equation*}
  \begin{tikzpicture}[shorten >=1pt, node distance=1cm and 2cm, on grid, decoration={snake, segment length=2mm, amplitude=0.5mm,post length=1.5mm},>=stealth']

  \node[state,accepting] (0) {};
  \node[] (q) [above right = 0.75 and 1.75 of 0] {$\vertii$};
  \node[] (p) [below right = 0.75 and 1.75 of 0] {$\verti$};
  \node[state,accepting] (0') [below right = 0.75 and 1.75 of q] {};
  \node[] (c) [right = 0.25 of 0'] {,};
   
    \path[->]
        (0) edge[snake it,bend left, out = 20,in= 170]
            node[pos=0.5,above = 0.1] {$u_1$}
            (q);
            
    \path[->]
        (q) edge[snake it,bend left, out = 10,in= 160]
            node[pos=0.5,above = 0.1] {$z_2$}
            (0');
            
    \path[->]
        (0) edge[snake it,bend left, out = -20,in= 190]
            node[pos=0.5,below = 0.1] {$w$}
            (p);
            
    \path[->]
        (p) edge[snake it,bend left, out = -10,in= 200]
            node[pos=0.5,below = 0.1] {$v$}
            (0');
            
    \path[->]
        (q) edge[snake it]
            node[pos=0.5,left = 0.1] {$u_2$}
            node[pos=0.5,right = 0.1] {$z_1^{-1}$}
            (p);
\end{tikzpicture}
\vspace{-5pt}
 \end{equation*}
on $u = u_1 u_2$, $z = z_1 z_2$, $u_2 = z_1^{-1}$, i el camí $\bp \xwalk{u_1\,} \vertii \xwalk{z_2\,} \bp$ és reduït. Per tant, $uv = u_1 u_2 v$, $wz=w z_1 z_2$ i  $u_1 z_2$ han de ser llegibles (sense cance\l.lació entre sí\l.labes) com a etiquetes de $\bp'$-camins reduïts a $\Ati'$:
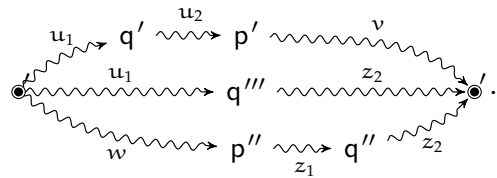
\begin{equation*}
  \begin{tikzpicture}[shorten >=1pt, node distance=1cm and 2cm, on grid, decoration={snake, segment length=2mm, amplitude=0.5mm,post length=1.5mm},>=stealth']

  \node[state,accepting] (0) {};
  \node[] () [above right = 0.1 and 0.1 of 0] {$'$};
  \node[] (q') [above right = 0.75 and 1.5 of 0] {$\vertii'$};
  \node[] (p') [right = 1.5 of q'] {$\verti'$};
  \node[] (q''') [right = 3 of 0] {$\vertii'''$};
  \node[] (p'') [below right = 0.75 and 3 of 0] {$\verti''$};
  \node[] (q'') [right = 1.5 of p''] {$\vertii''$};
  \node[state,accepting] (0') [right =  6 of 0]{};
  \node[] () [above right = 0.1 and 0.1 of 0'] {$'$};
  \node[] (c) [right = 0.25 of 0'] {.};

  \path[->]
        (0) edge[snake it,bend left, out = 20,in= 170]
            node[pos=0.5,above=0.1] {$u_1$}
            (q');
                   
    \path[->]
        (0) edge[snake it, bend left, out = -20,in= 190]
            node[pos=0.5,below = 0.1] {$w$}
            (p'');
            
    \path[->]
        (p') edge[snake it, bend left, out = 10,in= 150]
            node[pos=0.5,above = 0.1] {$v$}
            (0');
            
    \path[->]
        (q') edge[snake it]
            node[pos=0.52,above = 0.1] {$u_2$}
            (p');
            
    \path[->]
        (p'') edge[snake it]
            node[pos=0.52,below = 0.1] {$z_1$}
            (q'');
            
    \path[->]
        (0) edge[snake it]
            node[pos=0.5,above = 0.05] {$u_1$}
            (q''');
            
    \path[->]
        (q''') edge[snake it]
            node[pos=0.5,above = 0.05] {$z_2$}
            (0');
            
    \path[->]
        (q'') edge[snake it,bend left, out = -10,in= 200]
            node[pos=0.5,below=0.1] {$z_2$}
            (0');
\end{tikzpicture}
\vspace{-5pt}
 \end{equation*}
Finalment, pel determinisme de $\Ati'$, és clar, primer que $\vertii'= \vertii'''=\vertii''$, i després (tenint en compte que $u_2 =z_1 ^{-1}$) que $\verti'= \verti''$, tal i com reclamàvem. Per tant, $\theta\colon \Verts\Ati \to \Verts\Ati'$ està ben definit.

Per acabar, vegem que $\theta$ és un homomorfisme d'autòmats entre $\Ati$ i $\Ati'$. En efecte, com que $\Ati$ és cor, donat un arc $\edgi \equiv \verti \xarc{a\,} \vertii$ a $\Ati$, sabem que existeix un $\bp$-camí reduït a $\Ati$ que usa l'arc $\edgi$:
 \begin{equation*}
\bp \xwalk{u} \verti \xarc{\,a\ } \vertii \xwalk{v} \bp \,.
 \end{equation*}
Per tant, $uav \in \gen{\Ati} \leqslant \gen{\Ati'}$ (sense cance\l.lació entre sí\l.labes). I, per la definició de $\theta$, a $\Ati'$ hi tenim el $\bp'$-camí reduït
 \begin{equation*}
\bp' \xwalk{u} \verti \theta \xarc{\,a\ } \vertii\theta \xwalk{v} \bp' \,.
 \end{equation*}
En particular, tenim un arc $\verti \theta \xarc{a\, } \vertii \theta$ a $\Ati'$, i per tant $\theta$ és un homomorfisme d'autòmats de $\Ati$ a $\Ati'$, tal i com volíem demostrar.
\end{proof}

Els cor\l.laris següents són immediats, i capturen exactament la propietat que perseguim.

\begin{cor}
Si $\Ati$ és un autòmat reduït, aleshores l'únic homomorfisme $\Ati \to \Ati$ és la identitat. \qed
\end{cor}

\begin{cor} \label{cor: gen iff isom}
Si dos $A$-autòmats reduïts reconeixen el mateix subgrup, aleshores són isomorfs. És a dir, si $\Ati$ i $\Ati'$ són $A$-autòmats reduïts, aleshores
 \begin{equation*}
\gen{\Ati} = \gen{\Ati'} \,\Leftrightarrow\, \Ati \isom \Ati' \,.
 \end{equation*}
\end{cor}

Per tant, d'entre tots els $A$-autòmats reconeixent un cert subgrup $H \leqslant \Free[A]$, serà suficient triar-ne un de reduït per tal d'obtenir un representant únic, i restringir~\eqref{eq: Ati ->> <Ati>} a una bijecció. Vegem a continuació que aquest representant reduït sempre existeix, i com es relaciona amb un altre autòmat naturalment associat al subgrup $H$ corresponent.

\begin{defn}\label{def: Sch}
Siguin $G$ un grup, $H$ un subgrup de $G$, i $S\subseteq G$ un conjunt de generadors per a $G$. Aleshores, l'\defin{autòmat de Schreier (per la dreta) de $H$ respecte a~$S$}, designat per $\schreier_G(H,S)$, és el $S$-autòmat amb conjunt de vèrtexs $H\backslash G$ (el conjunt de classes laterals per la dreta de $G$ mòdul $H$), un arc $Hg \xarc{s\,} Hgs$ per a cada classe lateral $Hg\in H\backslash G$ i cada element $s\in S^{\pm}$, i la classe lateral $H$ com a vèrtex base.
\end{defn}

\begin{rem}
Noteu que, si $H\normaleq G$, llavors $\schreier_G(H,S)=\cayley(G/H,S)$. En particular, $\schreier_G(\Trivial,S)=\cayley(G,S)$. 
\end{rem}

Donat que d'ara endavant sempre considerarem $G=\Free[A]$, prendrem habitualment la base $A$ com a conjunt de generadors i escriurem simplement $\schreier_{\Free[A]}(H,A)=\schreier(H,A)$. El $A$-autòmat $\schreier(H,A)$ té les propietats següents. 

\begin{prop}\label{prop: propietats Sch}
Sigui $H$ un subgrup de $\Free[A]$. Aleshores, 
\begin{enumerate} [ind]
\item \label{item: Sch inv,det,sat} $\schreier(H,A)$ és un $A$-autòmat involutiu, determinista, i saturat;
\item \label{item: Sch connected} $\schreier(H,A)$ és connex;
\item \label{item: <Sch>=H} $\gen{\schreier(H,A)}=H$;
\item \label{item: Sch hanging trees} cada arc $\edgi$ no pertanyent a $\core (\schreier(H,A))$ és un pont de $\schreier(H,A)$; és a dir $\schreier(H,A)\setminus \{\edgi\}$ té dues components connexes i, a més, la que no conté $\bp$ és un arbre infinit i saturat \emph{excepte al vèrtex incident a $\edgi$};
\item \label{item: Sch no core} $\schreier(H,A)$ no és necessàriament cor.
\end{enumerate}
\end{prop}

\begin{proof}
L'apartat \ref{item: Sch inv,det,sat} és immediat de la definició d'autòmat de Schreier (\Cref{def: Sch}). La connexió de $\schreier(H,A)$ és conseqüència directa del fet que tot vèrtex $Hw$ està connectat amb el punt base: en efecte, posant $w=a_{i_1}^{\epsilon_1} a_{i_2}^{\epsilon_2} \cdots a_{i_l}^{\epsilon_l}$ com a paraula en $A^{\pm}$, tenim el camí
 \begin{equation} \label{eq: Sch connex}
\walki \colon \bp \! \xarc{\raisebox{0.6ex}{$\scriptstyle{a_{i_1}^{\epsilon_1}}$}} Ha_{i_1}^{\epsilon_1} \xarc{\raisebox{0.6ex}{$\scriptstyle{a_{i_2}^{\epsilon_2}}$}} Ha_{i_1}^{\epsilon_1}a_{i_2}^{\epsilon_2}  \xarc{\raisebox{0.6ex}{$\scriptstyle{a_{i_3}^{\epsilon_3}}$}} \ \cdots \ \xarc{\raisebox{0.6ex}{$\scriptstyle{a_{i_l}^{\epsilon_l}}$}} Hw \,.
 \end{equation}
Això demostra~\ref{item: Sch connected}. L'apartat~\ref{item: <Sch>=H} se segueix del fet que el camí~\eqref{eq: Sch connex} és tancat si i només si $Hw=\bp=H$, és a dir, si i només si $\lab(\walki)=w\in H$.

Per veure~\ref{item: Sch hanging trees} considerem un arc $\edgi$ de $\schreier(H,A)$ no pertanyent a $\core (\schreier(H,A))$. Si hi haguessin camins, i per tant camins reduïts, $\walki$ i $\walkii$ a $\schreier(H,A)\setminus \{\edgi\}$, de $\bp$ a $\iota \edgi$, i de $\bp$ a $\tau \edgi$, respectivament, llavors $\walki \edgi \walkii^{-1}$ seria un $\bp$-camí reduït contradient que $\edgi$ no forma part del cor. Per tant, $\edgi$ és un pont a $\schreier(H,A)$. Canviant $\edgi$ per $\edgi^{-1}$ si cal, podem suposar que hi ha un camí $\walki$ de $\bp$ a $\iota \edgi$, que no passa per l'arc $\edgi$. Llavors, si la component connexa de $\schreier(H,A)\setminus \{\edgi\}$ contenint $\tau \edgi$ no fos arbre, hi hauria un $\tau \edgi$-camí reduït no trivial $\walkiii$, i $\walki \edgi\walkiii \edgi^{-1}\walki^{-1}$ seria un $\bp$-camí reduït passant per $\edgi$ i contradient altre cop el fet que $\edgi$ no pertany al cor. Per tant, la component connexa de $\schreier(H,A)\setminus \{\edgi\}$ contenint $\tau \edgi$ és un arbre; i haurà de ser forçosament infinit ja que sino, tindria vèrtexs de grau 1, contradient el fet que $\schreier(H,A)$ és saturat. Com que cada vèrtex és saturat, aquest arbre és necessàriament una branca de Cayley (vegeu la~\Cref{fig: Sch(<b>)}).

Finalment, per veure \ref{item: Sch no core}
només cal observar l'autòmat $\schreier(\gen{b}, \{a,b\})$, dibuixat a la~\Cref{fig: Sch(<b>)}: un $b$-llaç al punt base $\bp$, i dues branques de Cayley  naixent dels dos arcs disponibles al punt base. 
\end{proof}

\begin{figure}[H]\label{fig: Stallings <<b>>}
\centering
  \begin{tikzpicture}[shorten >=1pt, node distance=1.2 and 1.2, on grid,auto,>=stealth',transform shape]
  
  \node[state,accepting] (bp) {};
  
  \path[->] (bp)
            edge[red,loop right,min distance=10mm,in=90+35,out=90-35]
            node[above] {\scriptsize{$b$}}
            (bp);
  
  \begin{scope}
  
  \newcommand{\dx}{1.6}
  \newcommand{\dy}{1.6}
  \newcommand{\dxx}{\dx*0.55}
  \newcommand{\dxxx}{\dx*0.3}
  \newcommand{\dxxxx}{\dx*0.18}
  \newcommand{\dyy}{\dy*0.55}
  \newcommand{\dyyy}{\dy*0.3}
  \newcommand{\dyyyy}{\dy*0.18}

   \node[smallstate] (1) {};
   \node[smallstate] (a) [right = \dx of 1]{};
   \node[smallstate] (aa) [right = \dxx of a]{};
   \node[tinystate] (aaa) [right = \dxxx of aa]{};
   \node[emptystate] (aaaa) [right = \dxxxx of aaa]{};
   
   \node[tinystate] (ab) [above = \dyy of a]{};
   \node[tinystate] (aab) [above = \dyyy of aa]{};
   \node[emptystate] (aabb) [above = \dyyyy of aab]{};
   \node[emptystate] (aaab) [above = \dyyyy of aaa]{};
   \node[emptystate] (aaaab) [above = \dyyyy of aaaa]{};
   
   \node[tinystate] (aB) [below = \dyy of a]{};
   \node[tinystate] (aaB) [below = \dyyy of aa]{};
   \node[emptystate] (aaBB) [below = \dyyyy of aaB]{};
   \node[emptystate] (aaaB) [below = \dyyyy of aaa]{};
   \node[emptystate] (aaaaB) [below = \dyyyy of aaaa]{};
   
   \node[tinystate] (aba) [right = \dxxx of ab]{};
   \node[tinystate] (abA) [left = \dxxx of ab]{};
   \node[emptystate] (aaba) [right = \dxxxx of aab]{};
   \node[emptystate] (aabA) [left = \dxxxx of aab]{};

   \node[tinystate] (aBa) [right = \dxxx of aB]{};
   \node[tinystate] (aBA) [left = \dxxx of aB]{};
   \node[emptystate] (aaBa) [right = \dxxxx of aaB]{};
   \node[emptystate] (aaBA) [left = \dxxxx of aaB]{};

   \node[tinystate] (abb) [above = \dyyy of ab]{};
   \node[tinystate] (aBB) [below = \dyyy of aB]{};
   
   \node[emptystate] (abbA) [left = \dxxxx of abb]{};
   \node[emptystate] (abbb) [above = \dyyyy of abb]{};
   \node[emptystate] (abba) [right = \dxxxx of abb]{};
   
   \node[emptystate] (aBBA) [left = \dxxxx of aBB]{};
   \node[emptystate] (aBBB) [below = \dyyyy of aBB]{};
   \node[emptystate] (aBBa) [right = \dxxxx of aBB]{};
   
   \node[emptystate] (abAA) [left = \dxxxx of abA]{};
   \node[emptystate] (abAb) [above = \dyyyy of abA]{};
   \node[emptystate] (abAB) [below = \dxxxx of abA]{};
   
   \node[emptystate] (aBAA) [left = \dxxxx of aBA]{};
   \node[emptystate] (aBAb) [above = \dyyyy of aBA]{};
   \node[emptystate] (aBAB) [below = \dxxxx of aBA]{};
   
   \node[emptystate] (abaa) [right = \dxxxx of aba]{};
   \node[emptystate] (abab) [above = \dyyyy of aba]{};
   \node[emptystate] (abaB) [below = \dxxxx of aba]{};
   
   \node[emptystate] (aBaa) [right = \dxxxx of aBa]{};
   \node[emptystate] (aBab) [above = \dyyyy of aBa]{};
   \node[emptystate] (aBaB) [below = \dxxxx of aBa]{};
     
   \path[->] (a) edge[red] node[pos = 0.4,left] {\scriptsize{$b$}} (ab);
   \path[->] (aa) edge[red] (aab);
   \path[->,densely dotted] (aab) edge[red] (aabb);
   \path[->,densely dotted] (aaa) edge[red] (aaab);
   
   \path[->] (aB) edge[red] (a);
   \path[->] (aaB) edge[red] (aa);
   \path[->,densely dotted] (aaBB) edge[red] (aaB);
   \path[->,densely dotted] (aaaB) edge[red] (aaa);
   
   \path[->] (ab) edge[red] (abb);
   \path[->] (aBB) edge[red] (aB);
   
   \path[->] (1) edge[blue] node[pos = 0.4,above] {\scriptsize{$a$}}(a);
   \path[->] (a) edge[blue] (aa);
   \path[->] (aa) edge[blue] (aaa);
   \path[->,densely dotted] (aaa) edge[blue] (aaaa);
   
   \path[->] (abA) edge[blue] (ab);
   \path[->] (ab) edge[blue] (aba);
   \path[->] (ab) edge[blue] (aba);
   
   \path[->,densely dotted] (aabA) edge[blue] (aab);
   \path[->,densely dotted] (aab) edge[blue] (aaba);
   
   \path[->] (aBA) edge[blue] (aB);
   \path[->] (aB) edge[blue] (aBa);
   \path[->] (aB) edge[blue] (aBa);
   
   \path[->,densely dotted] (aaBA) edge[blue] (aaB);
   \path[->,densely dotted] (aaB) edge[blue] (aaBa);
   
   \path[->,densely dotted] (abb) edge[red] (abbb);
   \path[->,densely dotted] (abbA) edge[blue] (abb);
   \path[->,densely dotted] (abb) edge[blue] (abba);
   
   \path[->,densely dotted] (aBBB) edge[red] (aBB);
   \path[->,densely dotted] (aBBA) edge[blue] (aBB);
   \path[->,densely dotted] (aBB) edge[blue] (aBBa);
   
   \path[->,densely dotted] (abA) edge[red] (abAb);
   \path[->,densely dotted] (abAA) edge[blue] (abA);
   \path[->,densely dotted] (abAB) edge[red] (abA);
   
   \path[->,densely dotted] (aBA) edge[red] (aBAb);
   \path[->,densely dotted] (aBAA) edge[blue] (aBA);
   \path[->,densely dotted] (aBAB) edge[red] (aBA);
   
   \path[->,densely dotted] (aba) edge[red] (abab);
   \path[->,densely dotted] (aba) edge[blue] (abaa);
   \path[->,densely dotted] (abaB) edge[red] (aba);
   
   \path[->,densely dotted] (aBa) edge[red] (aBab);
   \path[->,densely dotted] (aBa) edge[blue] (aBaa);
   \path[->,densely dotted] (aBaB) edge[red] (aBa);
  \end{scope}

  \begin{scope}[rotate=180]
  \newcommand{\dx}{1.6}
  \newcommand{\dy}{1.6}
  \newcommand{\dxx}{\dx*0.55}
  \newcommand{\dxxx}{\dx*0.3}
  \newcommand{\dxxxx}{\dx*0.18}
  \newcommand{\dyy}{\dy*0.55}
  \newcommand{\dyyy}{\dy*0.3}
  \newcommand{\dyyyy}{\dy*0.18}

  \node[smallstate] (1) {};
  \node[smallstate] (a) [right = \dx of 1]{};
  \node[smallstate] (aa) [right = \dxx of a]{};
  \node[tinystate] (aaa) [right = \dxxx of aa]{};
  \node[emptystate] (aaaa) [right = \dxxxx of aaa]{};
   
  \node[tinystate] (ab) [above = \dyy of a]{};
  \node[tinystate] (aab) [above = \dyyy of aa]{};
  \node[emptystate] (aabb) [above = \dyyyy of aab]{};
  \node[emptystate] (aaab) [above = \dyyyy of aaa]{};
  \node[emptystate] (aaaab) [above = \dyyyy of aaaa]{};
   
  \node[tinystate] (aB) [below = \dyy of a]{};
  \node[tinystate] (aaB) [below = \dyyy of aa]{};
  \node[emptystate] (aaBB) [below = \dyyyy of aaB]{};
  \node[emptystate] (aaaB) [below = \dyyyy of aaa]{};
  \node[emptystate] (aaaaB) [below = \dyyyy of aaaa]{};
   
  \node[tinystate] (aba) [right = \dxxx of ab]{};
  \node[tinystate] (abA) [left = \dxxx of ab]{};
  \node[emptystate] (aaba) [right = \dxxxx of aab]{};
  \node[emptystate] (aabA) [left = \dxxxx of aab]{};

  \node[tinystate] (aBa) [right = \dxxx of aB]{};
  \node[tinystate] (aBA) [left = \dxxx of aB]{};
  \node[emptystate] (aaBa) [right = \dxxxx of aaB]{};
  \node[emptystate] (aaBA) [left = \dxxxx of aaB]{};

  \node[tinystate] (abb) [above = \dyyy of ab]{};
  \node[tinystate] (aBB) [below = \dyyy of aB]{};
   
  \node[emptystate] (abbA) [left = \dxxxx of abb]{};
  \node[emptystate] (abbb) [above = \dyyyy of abb]{};
  \node[emptystate] (abba) [right = \dxxxx of abb]{};
   
  \node[emptystate] (aBBA) [left = \dxxxx of aBB]{};
  \node[emptystate] (aBBB) [below = \dyyyy of aBB]{};
  \node[emptystate] (aBBa) [right = \dxxxx of aBB]{};
   
  \node[emptystate] (abAA) [left = \dxxxx of abA]{};
  \node[emptystate] (abAb) [above = \dyyyy of abA]{};
  \node[emptystate] (abAB) [below = \dxxxx of abA]{};
   
  \node[emptystate] (aBAA) [left = \dxxxx of aBA]{};
  \node[emptystate] (aBAb) [above = \dyyyy of aBA]{};
  \node[emptystate] (aBAB) [below = \dxxxx of aBA]{};
   
  \node[emptystate] (abaa) [right = \dxxxx of aba]{};
  \node[emptystate] (abab) [above = \dyyyy of aba]{};
  \node[emptystate] (abaB) [below = \dxxxx of aba]{};
   
  \node[emptystate] (aBaa) [right = \dxxxx of aBa]{};
  \node[emptystate] (aBab) [above = \dyyyy of aBa]{};
  \node[emptystate] (aBaB) [below = \dxxxx of aBa]{};
 
  \path[->] (ab) edge[red] (a);
  \path[->] (aab) edge[red] (aa);
  \path[->,densely dotted] (aabb) edge[red] (aab);
  \path[->,densely dotted] (aaab) edge[red] (aaa);
   
  \path[->] (a) edge[red] (aB);
  \path[->] (aa) edge[red] (aaB);
  \path[->,densely dotted] (aaB) edge[red] (aaBB);
  \path[->,densely dotted] (aaa) edge[red] (aaaB);
   
  \path[->] (abb) edge[red] (ab);
  \path[->] (aB) edge[red] (aBB);
   
  \path[->] (a) edge[blue] (1);
  \path[->] (aa) edge[blue] (a);
  \path[->] (aaa) edge[blue] (aa);
  \path[->,densely dotted] (aaaa) edge[blue] (aaa);
   
  \path[->] (ab) edge[blue] (abA);
  \path[->] (aba) edge[blue] (ab);
  \path[->] (aba) edge[blue] (ab);
   
  \path[->,densely dotted] (aab) edge[blue] (aabA);
  \path[->,densely dotted] (aaba) edge[blue] (aab);
   
  \path[->] (aB) edge[blue] (aBA);
  \path[->] (aBa) edge[blue] (aB);
  \path[->] (aBa) edge[blue] (aB);
   
  \path[->,densely dotted] (aaB) edge[blue] (aaBA);
  \path[->,densely dotted] (aaBa) edge[blue] (aaB);
   
  \path[->,densely dotted] (abbb) edge[red] (abb);
  \path[->,densely dotted] (abb) edge[blue] (abbA);
  \path[->,densely dotted] (abba) edge[blue] (abb);
   
  \path[->,densely dotted] (aBB) edge[red] (aBBB);
  \path[->,densely dotted] (aBB) edge[blue] (aBBA);
  \path[->,densely dotted] (aBBa) edge[blue] (aBB);
   
  \path[->,densely dotted] (abAb) edge[red] (abA);
  \path[->,densely dotted] (abA) edge[blue] (abAA);
  \path[->,densely dotted] (abA) edge[red] (abAB);
   
  \path[->,densely dotted] (aBAb) edge[red] (aBA);
  \path[->,densely dotted] (aBA) edge[blue] (aBAA);
  \path[->,densely dotted] (aBA) edge[red] (aBAB);
   
  \path[->,densely dotted] (abab) edge[red] (aba);
  \path[->,densely dotted] (abaa) edge[blue] (aba);
  \path[->,densely dotted] (aba) edge[red] (abaB);
   
  \path[->,densely dotted] (aBab) edge[red] (aBa);
  \path[->,densely dotted] (aBaa) edge[blue] (aBa);
  \path[->,densely dotted] (aBa) edge[red] (aBaB);
  \end{scope}
   
\end{tikzpicture}
\caption{L'autòmat de Schreier de $\gen{b} \leqslant \Free_{\set{a,b}}$}
\label{fig: Sch(<b>)}
\end{figure}
De la \Cref{prop: propietats Sch} és clar que la propietat d'esser cor és l'única que li falta a $\schreier(H,A)$ per ser un autòmat reduït que reconeix $H$.

\begin{defn} \label{def: St}
Sigui $H$ un subgrup de $\Free[A]$. Anomenem \defin{autòmat de Stallings} de $H$ respecte a $A$, designat per $\stallings(H,A)$, al cor de $\schreier(H,A)$; és a dir $\stallings(H,A)=\core (\schreier(H,A))$. També definim l'\defin{autòmat restringit de Stallings} de $H$ respecte a $A$ com el dígraf $\stallings^*(H,A)=\core^*(\stallings(H,A))$. Si la base de referència $A$ és clara pel context, habitualment l'ometrem i escriurem simplement $\schreier(H)$, $\stallings(H)$, i $\stallings^*(H)$.
\end{defn}

\begin{exm}
(La part positiva de) l'autòmat de Stallings $\stallings(\Free[A],A)$ té un únic vèrtex i un llaç etiquetat per cada $a \in A$. Aquest autòmat s'anomena \defin{rosa} (de $\card A$ pètals), i es designa per $\bouquet_A$ ó $\bouquet_n$, on $n=\card A$:
\begin{figure}[H]
    \centering
    \begin{tikzpicture}[shorten >=1pt, node distance=.5cm and 1.5cm, on grid,auto,>=stealth']
   \node[state,accepting] (1) {};

   \path[->]
        (1) edge[blue,loop right,min distance=12mm,in=330+35,out=330-35]
            node[right] {\scriptsize{$a$}}
            (1)
            edge[red,loop right,min distance=12mm,in=90+35,out=90-35]
            node[above] {\scriptsize{$b$}}
            (1)
            edge[black,loop right,min distance=12mm,in=210+35,out=210-35]
            node[left] {\scriptsize{$c$}}
            (1);
\end{tikzpicture} 
\vspace{-15pt}
\caption{$\bouquet_{3}$, una rosa de tres pètals}
\label{fig: bouquet}
\end{figure}
\end{exm}
\begin{rem}\label{rem: rose free}
    És clar que, per a qualsevol cardinal $n$, el grup fonamental de la rosa de $n$ pètals és el grup lliure de rang $n$; és a dir $\pi_{\bp}(\bouquet_{n}) \isom \Fn$.
\end{rem}

A continuació resumim les principals propietats de l'autòmat de Stallings i la seva relació amb el de Schreier.

\begin{prop}\label{prop: propietats St}
Sigui $H$ un subgrup de $\Free[A]$. Aleshores, 
\begin{enumerate}[ind]
\item \label{item: St inv,det,cor} $\stallings(H,A)$ és un $A$-autòmat involutiu, determinista, i cor (per tant reduït);
\item \label{item: <St>=H} $\gen{\stallings(H,A)}=H$;
\item \label{item: St sat iff Sch cor}
$\stallings(H,A) \text{ és saturat}
\ \Leftrightarrow\  \stallings(H,A)=\schreier(H,A) 
\ \Leftrightarrow\  \schreier(H,A) \text{ és cor;}$
\item \label{item: St no sat} $\stallings(H,A)$ no és necessàriament saturat;
\item \label{item: Sch = St + Cay branches} $\schreier(H,A)$ consisteix en $\stallings(H,A)$ amb una $a$-branca de Stallings adjacent a cada vèrtex $a$-deficient de $\stallings(H,A)$, on $a \in A^\pm$.
\item \label{item: fi iff St} L'índex $\ind{H}{\Free[A]}$ és finit si i només si $\stallings(H,A)$ és saturat i $\card\Verts \,\stallings(H,A) < \infty$.
\end{enumerate}
\end{prop}

\begin{proof}
L'apartat \ref{item: St inv,det,cor} és immediat de la definició d'autòmat de Stallings (\Cref{def: St}), mentre que \ref{item: <St>=H} se segueix de la mateixa definició i l'\Cref{rem: <Ati> = <core> = <core*>}. 

Per veure \ref{item: St sat iff Sch cor} només cal recordar que $\schreier(H)$ és sempre saturat, $\stallings(H)$ és sempre cor, i $\stallings(H)=\core(\schreier(H))$. Per tant,  \ref{item: St no sat} és conseqüència de la \Cref{prop: propietats Sch}\ref{item: Sch no core}.

\ref{item: Sch = St + Cay branches} Si $\verti$ és un vèrtex $a$-deficient de $\stallings(H)$, l'arc de $\schreier(H)$ llegint $a$ des de $\verti$ no pertany a $\stallings(H)$ i aleshores el resultat reclamat no és més que una reinterpretació de la \Cref{prop: propietats Sch}\ref{item: Sch hanging trees}.

\ref{item: fi iff St} Recordem que $\ind{H}{\Free[A]}=\card\Verts\, \schreier(H,A)$. Si $\stallings(H)$ és saturat i té un nombre finit de vèrtexs, aleshores $\ind{H}{\Free[A]}=\card\Verts\, \schreier(H,A)=\card\Verts\, \stallings(H,A) <\infty$, on hem usat l'apartat \ref{item: St sat iff Sch cor} a la segona igualtat. Recíprocament, suposem $\ind{H}{\Free[A]}<\infty$: per una banda $\card\Verts \, \stallings(H,A) \leq \card\Verts \, \schreier(H,A)=\ind{H}{\Free[A]}<\infty$; i per altra, si $\stallings(H,A)$ fos insaturat, \ref{item: Sch = St + Cay branches} implicaria que $\schreier(H,A)$ té arbres infinits penjant de tots els vèrtexs insaturats, contradient que $\card\Verts \,\schreier(H,A)= \ind{H}{\Free[A]} < \infty$. 
\end{proof}

És convenient tenir present que, com a conseqüència de la teoria desenvolupada, les propietats (i) i (ii) de la \Cref{prop: propietats Sch}, i la propietat \ref{item: St inv,det,cor} de la \Cref{prop: propietats St}, caracteritzen, de fet, els  respectius tipus d'autòmats.

\begin{cor}\label{cor: SchSt}
Sigui $\Ati$ un $A$-autòmat involutiu i determinista. Aleshores:
 \begin{enumerate}[ind]
\item si $\Ati$ és saturat i connex, aleshores $\Ati \isom \schreier(\gen{\Ati},A)$;
\item si $\Ati$ és cor, aleshores $\Ati \isom \stallings(\gen{\Ati},A)$.\qed
 \end{enumerate}
\end{cor}

\medskip

Disposem ja de tots els ingredients necessaris per a establir la bijecció que buscàvem entre $A$-autòmats i subgrups de $\Free[A]$, tot restringint~\eqref{eq: Ati ->> <Ati>} al domini adequat.

\begin{thm}[\citenr{stallings_topology_1983}]\label{thm: Stallings bijection}
La funció
 \begin{equation} \label{eq: Stallings bijection1}
\begin{array}{rcl}
\Set{\text{(classes d'isomorfia de) $A$-autòmats reduïts}} & \to & \Set{\text{subgrups de $\Free[A]$}} \\ \Ati & \mapsto & \gen{\Ati}
\end{array}
 \end{equation}
és una bijecció, amb inversa $\stallings(H,A) \mapsfrom H$. 
\end{thm}

\begin{proof}
Les dues aplicacions estan ben definides: en efecte, per a tot $\Ati$, $\gen{\Ati}$ és un subgrup de $\Free[A]$; i, gràcies a la \Cref{prop: propietats St}\ref{item: St inv,det,cor}, per tot subgrup $H\leqslant \Free[A]$, $\stallings(H,A)$ és un $A$-autòmat reduït. I són una inversa de l'altra: per \ref{prop: propietats St}\ref{item: <St>=H}, $\gen{\stallings(H,A)}=H$ i, com que $\Ati$ és reduït, $\stallings(\gen{\Ati},A)=\Ati$, a conseqüència del \Cref{cor: gen iff isom}.
\end{proof}

\medskip

El pas següent és garantir el caràcter computacional de la bijecció~\eqref{eq: Stallings bijection1}, que serà fonamental per als nostres interessos. Veurem que els subgrups \emph{finitament generats} corresponen precisament als $A$-autòmats reduïts \emph{finits} i que la conversió entre ambdós és algorísmica. Trobarem mètodes eficients per a computar les dues direccions de la bijecció \eqref{eq: Stallings bijection1}:
\begin{enumerate*}[dep]
\item 
donat un $A$-autòmat reduït i finit $\Ati$, calcularem explícitament un sistema de generadors finit de $\gen{\Ati}$ (millor encara, en calcularem una base, demostrant de pas que $\gen{\Ati}$ sempre torna a ser un grup lliure); i
\item donat un sistema de generadors finit de $\gen{\Ati}$, construirem el seu autòmat (finit) de Stallings $\stallings(H,A)$. 
\end{enumerate*}

\begin{prop} \label{prop: S_T}
Sigui $\Ati$ un $A$-autòmat involutiu i connex, sigui $T$ és un arbre d'expansió de $\Ati$, i considerem 
 \begin{equation}
S_T=\{w_{\edgi} \st \edgi\in \Edgs^+\Ati \setmin \Edgs T\}\subseteq \gen{\Ati}\leqslant \Free[A],
  \end{equation}
on $w_{\edgi}=\red{\lab}(\bp \xwalk{\scriptscriptstyle{T}} \! \bullet \! \xarc{\,\edgi\ } \!\bullet\! \xwalk{\scriptscriptstyle{T}} \bp)$. Llavors,
\begin{enumerate}[ind]
\item $S_T$ és un sistema de generadors per $\gen{\Ati}$;
\item \label{item: S_T lliure} si $\Ati$ és determinista, aleshores $\gen{\Ati}$ és un grup lliure i $S_T$ n'és una base; en particular, el rang algebraic $\rk(\gen{\Ati})$ i el rang gràfic $\rk(\Ati)$ coincideixen;
\item si $\Ati$ és reduït, aleshores $\gen{\Ati}$ és finitament generat si i només si $\Ati$ és finit; i en aquest cas, $\rk(\gen{\Ati})=1- \card\Verts\Ati+\card\Edgs^{+}\Ati$.
\end{enumerate}
\end{prop}

\begin{proof}

(i). Sigui $w\in \gen{\Ati}$, i expressem-lo $w=\rlab(\walki)$, per a cert $\bp$-camí reduït $\walki$ a $\Ati$. Si distingim les visites de $\walki$ a arcs fora de $T$, podem escriure
 \begin{equation*}
\walki \colon \bp \xwalk{T} \!\!\bullet\!\!  \xarc{\raisebox{0.6ex}{$\scriptstyle{\edgi_1^{\epsilon_1}}$}} \!\!\bullet\!\! \xwalk{T}  \!\!\bullet\!\!  \xarc{\raisebox{0.6ex}{$\scriptstyle{\edgi_2^{\epsilon_2}}$}} \!\!\bullet\!\! \xwalk{T} \!\!\bullet \ \cdots \ \bullet\!\! \xwalk{T} \!\!\bullet\!\! \xarc{\raisebox{0.6ex}{$\scriptstyle{\edgi_l^{\epsilon_l}}$}}\!\!\bullet\!\! \xwalk{T} \bp,
 \end{equation*}
on $\edgi_1,\ldots ,\edgi_l \in \Edgs^+\Ati \setmin \Edgs T$ (amb possibles repeticions) i $\epsilon_j =\pm 1$. Observem ara que $\walki$ té la mateixa etiqueta reduïda que el camí ampliat
 \begin{equation} \label{eq: cami ampliat}
\walki' \colon \bp \! \xwalk{T} \!\! \bullet \!\! \xarc{\raisebox{0.6ex}{$\scriptstyle{\edgi_1^{\epsilon_1}}$}} \!\!\bullet \!\! \xwalk{T} \bp \xwalk{T}\!\! \bullet \!\! \xarc{\raisebox{0.6ex}{$\scriptstyle{\edgi_2^{\epsilon_2}}$}} \!\!\bullet\!\! \xwalk{T} \bp \ \cdots \ \bp \xwalk{T} \!\!\bullet\!\! \xarc{\raisebox{0.6ex}{$\scriptstyle{\edgi_l^{\epsilon_l}}$}}\!\!\bullet\!\! \xwalk{T} \bp \,.
 \end{equation}
Per tant, $w=\rlab(\walki)=\rlab(\walki')=w_{\edgi_1}^{\epsilon_1} w_{\edgi_2}^{\epsilon_2} \cdots w_{\edgi_l}^{\epsilon_l}\in \gen{S_T}$. 

(ii). Suposem ara, a més, que $\Ati$ és determinista. Demostrarem que $S_T$ és, de fet, una base de $\gen{\Ati}$ i que, per tant, $\gen{\Ati}$ és un grup lliure. És suficient veure que qualsevol paraula reduïda no buida en $S_T$ representa un element no trivial de $\Free[A]$. En efecte, sigui 
$w= w_{\edgi_1}^{\epsilon_1} w_{\edgi_2}^{\epsilon_2} \! \cdots w_{\edgi_l}^{\epsilon_l}$, amb $l\geq 1$, $\edgi_1,\ldots ,\edgi_l \in \Edgs^+\Ati \setmin \Edgs T$ (amb possibles repeticions), i tal que no hi ha dos factors successius inversos l'un de l'altre, \ie $\edgi_i=\edgi_{i+1}$ implica $\epsilon_i=\epsilon_{i+1}$. Aleshores,
 \begin{align*} 
\red{w} & =\red{w_{\edgi_1}^{\epsilon_1} w_{\edgi_2}^{\epsilon_2} \! \cdots w_{\edgi_l} ^{\epsilon_l}} \notag \\[3pt]
& =\rlab\,\big{(} \bp \! \xwalk{T} \!\! \bullet \!\! \xarc{\raisebox{0.5ex}{$\scriptstyle{\edgi_1^{\epsilon_1}}$}} \!\!\bullet \!\! \xwalk{T} \bp \xwalk{T}\!\! \bullet \!\! \xarc{\raisebox{0.5ex}{$\scriptstyle{\edgi_2^{\epsilon_2}}$}} \!\!\bullet\!\! \xwalk{T} \bp \ \cdots \ \bp \xwalk{T} \!\!\bullet\!\! \xarc{\raisebox{0.5ex}{$\scriptstyle{\edgi_l^{\epsilon_l}}$}}\!\!\bullet\!\! \xwalk{T} \bp \big{)} \notag \\[3pt] & =\rlab\,\big{(} \bp \xwalk{T} \!\!\bullet\!\!  \xarc{\raisebox{0.5ex}{$\scriptstyle{\edgi_1^{\epsilon_1}}$}} \!\!\bullet\!\! \xwalk{T}  \!\!\bullet\!\!  \xarc{\raisebox{0.5ex}{$\scriptstyle{\edgi_2^{\epsilon_2}}$}} \!\!\bullet\!\! \xwalk{T} \!\!\bullet \ \cdots \ \bullet\!\! \xwalk{T} \!\!\bullet\!\! \xarc{\raisebox{0.5ex}{$\scriptstyle{\edgi_l^{\epsilon_l}}$}}\!\!\bullet\!\! \xwalk{T} \bp \big{)} \,.
  \end{align*}
Però aquest últim $\bp$-camí és no buit ($l\geq 1$), i reduït (donat que els arcs $\edgi_i^{\epsilon_i}$ estan tots fora de~$T$, i dos de consecutius mai presenten retrocès). Com que $\Ati$ és determinista, això significa que $\red{w} \neq \trivial$, com volíem demostrar. Per tant, $S_T$ és una base de $\gen{\Ati}$ i el rang algebraic i el rang gràfic coincideixen, $\rk(\gen{\Ati}) = \card S_T =\rk(\Ati)$. 

(iii). Finalment, suposem que $\Ati$ és reduït. Si $\Ati$ és finit llavors $\card (\Edgs^+\Ati \setmin \Edgs T)<\infty$ i $\gen{\Ati}$ és finitament generat. Recíprocament, si $\Ati=\core(\Ati)$ té rang finit aleshores és finit (ja que afegint un nombre finit d'arcs a un arbre infinit no resulta mai un cor). I en aquest cas, $\rk(\gen{\Ati})=\rk(\Ati)=1- \card\Verts\Ati+\card\Edgs^{+}\Ati$. La demostració està completa.
\end{proof}

\begin{exm}
    Considereu l'autòmat determinista següent, i l'arbre d'expansió~$T$ format pels arcs horitzontals (representats amb traç més gruixut).
    
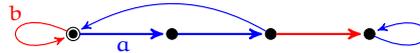
\begin{figure}[H]
        \centering
        \begin{tikzpicture}[shorten >=1pt, node distance=1.2cm and 2cm, on grid,auto,auto,>=stealth']
        
        \newcommand{\dx}{1.3}
        \newcommand{\dy}{1.2}

        \node[state,accepting] (a1) [right = \dy-1/3 of 0] {};
        \node[state] (a2) [right = \dx of a1] {};
        \node[state] (a3) [right = \dx of a2] {};
        \node[state] (a4) [right = \dx of a3] {};

        \path[->]
     (a1) edge[red,loop above,min distance=10mm,in=205,out=155]
     node[above = 0.1] {$b$}
     (a1)
     (a1) edge[blue,thick] node[below] {$a$} (a2)
     (a2) edge[blue,thick] (a3)
     (a3) edge[blue,bend right=30] (a1)
     (a3) edge[red,thick] (a4)
     (a4) edge[blue,loop above,min distance=10mm,in=25,out=-25] (a4);
    \end{tikzpicture}
\end{figure}

\vspace{-15pt}
D'acord amb la \Cref{prop: S_T}.\ref{item: S_T lliure}, el subgrup $\gen{\Ati}$ reconegut per $\Ati$ és lliure amb base el conjunt $S_T = \set{\,b, a^3,a^2 b a b^{-1} a^{-2}\,}$ i, per tant, de rang $3$.
\end{exm}

\medskip
El següent resultat fonamental, demostrat primer per \citeauthor{nielsen_om_1921}, per a subgrups finitament generats, i uns anys més tard per \citeauthor{schreier_untergruppen_1927}, amb total generalitat, se segueix ara de manera transparent de la teoria de Stallings.

\begin{thm}[Teorema de Nielsen--Schreier, \cite{nielsen_om_1921,schreier_untergruppen_1927}] \label{thm: Nielsen-Schreier}
Tot subgrup d'un grup lliure és lliure. 
\end{thm}

\begin{proof}
Pel \Cref{thm: Stallings bijection}, tot subgrup $H\leqslant \Free[A]$ és de la forma $H=\gen{\stallings(H,A)}$; i, per la \Cref{prop: S_T}, és lliure. 
\end{proof}

\medskip

Finalment, situem-nos en el cas finit, i busquem un algorisme eficient per computar els dos sentits de la bijecció~\eqref{eq: Stallings bijection1}. El càlcul d'una base del subgrup reconegut per un $A$-autòmat reduït i finit $\Ati$ és directe de la proposició anterior: calculem un arbre d'expansió $T$ de $\Ati$ i, per a cada arc $\edgi\in \Edgs^+\Ati \setmin \Edgs T$ (n'hi ha un nombre finit), calculem la paraula reduïda $w_{\edgi}=\red{\lab}(\bp \xwalk{\scriptscriptstyle{T}} \! \bullet \! \xarc{\,\edgi\ } \!\bullet\! \xwalk{\scriptscriptstyle{T}} \bp)$. Per la \Cref{prop: S_T}(ii), $S_T=\{w_{\edgi} \st \edgi\in \Edgs^+\Ati \setmin \Edgs T\}$ és una base de $\gen{\Ati}\leqslant \Free[A]$. 

Per computar la direcció contrària, donat un subgrup $H\leqslant \Free[A]$ (mitjançant una família finita $S$ de generadors, que podem suposar paraules reduïdes) hem de trobar una manera efectiva de calcular $\stallings(H,A)$; o, el que és el mateix, una manera de calcular un $A$-autòmat $\Ati$ que sigui involutiu, reduït, finit, i que reconegui $\gen{\Ati}=H$. Ja coneixem un $A$-autòmat, molt fàcil de construir, i que gairebé compleix totes aquestes condicions: l'autòmat flor $\flower(S)$ és involutiu, cor, finit, reconeix $\gen{\flower(S)}=\gen{S}=H$, i és determinista \emph{excepte potser al vèrtex base $\bp$} (ja que els diferents elements de $S$ poden començar o acabar amb la mateixa lletra de $A^{\pm}$). Per a `reparar' l'eventual no determinisme de $\flower(S)$, necessitem un procediment que converteixi un $A$-autòmat finit en determinista (sense perdre les altres propietats); el desenvolupem a continuació. 

\begin{defn}\label{def: foldings}
Sigui $\Ati$ un $A$-autòmat involutiu; i siguin $\edgi_1$ i $\edgi_2$ dos arcs de $\Ati$ violant el determinisme, \ie complint $\iota\edgi_1 =\iota \edgi_2$, $\lab(\edgi_1 )=\lab(\edgi_2 )$, però $\edgi_1\neq \edgi_2$. 
Anomenem \defin{plegament (elemental) de Stallings}\footnote{\emph{Stallings folding}, en anglès.}, i el designem per \smash{$\Ati\! \xtr{\!}\! \Ati'$}, a la transformació consistent en identificar els arcs $\edgi_1$ i $\edgi_2$ (i els seus corresponents inversos) a $\Ati$.
Si els arcs $\edgi_1$ i $\edgi_2$ no són para\l.lels (\ie si $\tau \edgi_1\neq \tau \edgi_2$), direm que és un \defin{plegament obert}; en cas contrari, direm que és un \defin{plegament tancat} (vegeu la \Cref{fig: foldings}).
\end{defn}

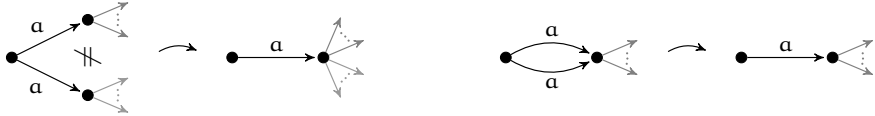
\begin{figure}[H] 
\centering
\begin{tikzpicture}[shorten >=1pt, node distance=1cm and 1.5cm, on grid,>=stealth']
\begin{scope}
   \node[state] (1) {};
   \node[state] (2) [above right = 0.5 and 1 of 1] {};
   \node[] (21) [above right = 0.3 and 0.7 of 2] {};
   \node[] (22) [below right = 0.3 and 0.7 of 2] {};
   \node[state] (3) [below right = 0.5 and 1 of 1] {};
   \node[] (31) [above right = 0.3 and 0.7 of 3] {};
   \node[] (32) [below right = 0.3 and 0.7 of 3] {};
   \node[] (neq) [right = 1 of 1]{\rotatebox[origin=c]{90}{$\neq$}};

   \path[->]
        (1) edge[]
            node[pos=0.52,above left] {$a$}
            (2)
            edge[]
            node[pos=0.5,below left] {$a$}
            (3)
        (2) edge[gray]
            (21)
            edge[gray]
            (22)
        (3) edge[gray!80]
            (31)
            edge[gray!80]
            (32);

    \node[] (i) [right = 1.8 of 1]{};
    \node[] (f) [right = 0.8 of i] {};
    \path[->]
        (i) edge[bend left] (f);

   \node[state] (1') [right= 0.3 of f] {};
   \node[state] (N) [right =  1.2 of 1'] {};
   \node[] (21') [above right = 0.7 and 0.3 of N] {};
   \node[] (22') [above right = 0.3 and 0.7 of N] {};
   \node[] (32') [below right = 0.3 and 0.7 of N] {};
   \node[] (33') [below right = 0.7 and 0.3 of N] {};
   \path[->]
        (1') edge
            node[pos=0.48,above] {$a$}
            (N)
        (N) edge[gray]
            (21')
            edge[gray]
            (22')
            edge[gray!80]
            (32')
            edge[gray!80]
            (33');

\foreach \n [count=\count from 0] in {1,...,3}{
       \node[dot,gray] (2d\n) at ($(2)+(-10+\count*10:0.4cm)$) {};}

\foreach \n [count=\count from 0] in {1,...,3}{
       \node[dot,gray!80] (2d\n) at ($(3)+(-10+\count*10:0.4cm)$) {};}

\foreach \n [count=\count from 0] in {1,...,3}{
       \node[dot,gray] (2d\n) at ($(N)+(34+\count*10:0.4cm)$) {};}

\foreach \n [count=\count from 0] in {1,...,3}{
       \node[dot,gray!80] (2d\n) at ($(N)+(-53+\count*10:0.4cm)$) {};}
  \end{scope}
  
 \begin{scope}[xshift=6.5cm]
  \node[state] (1)  {};
   \node[state] (2) [right = 1.2 of 1] {};
  \node[] (21) [above right = 0.7 and 0.3 of 2] {};
  \node[] (22) [above right = 0.3 and 0.7 of 2] {};
  \node[] (32) [below right = 0.3 and 0.7 of 2] {};
  \node[] (33) [below right = 0.7 and 0.3 of 2] {};
  \path[->]
        (2) edge[gray]
            (22)
            edge[gray]
            (32);

  \path[->]
        (1) edge[bend left]
            node[pos=0.5,above] {$a$}
            (2);
  \path[->]
        (1) edge[bend right]
            node[pos=0.5,below] {$a$}
            (2);

    \node[] (i) [right = 2 of 1] {};
    \node[] (f) [right = 0.8 of i] {};
    \path[->]
        (i) edge[bend left] (f);

  \node[state] (1') [right= 0.3 of f] {};
  \node[state] (N) [right =  1.2 of 1'] {};
  \node[] (21') [above right = 0.7 and 0.3 of N] {};
  \node[] (22') [above right = 0.3 and 0.7 of N] {};
  \node[] (32') [below right = 0.3 and 0.7 of N] {};
  \node[] (33') [below right = 0.7 and 0.3 of N] {};
  \path[->]
        (1') edge
            node[pos=0.48,above] {$a$}
            (N)
        (N) edge[gray]
            (22')
            edge[gray]
            (32');

\foreach \n [count=\count from 0] in {1,...,3}{
      \node[dot,gray] (2d\n) at ($(2)+(-12+\count*10:0.4cm)$) {};}

\foreach \n [count=\count from 0] in {1,...,3}{
      \node[dot,gray] (2d\n) at ($(N)+(-12+\count*10:0.4cm)$) {};} 
 \end{scope}
\end{tikzpicture}
 \vspace{-5pt}
\caption{Un plegament obert (esquerra) i un de tancat (dreta)}
\label{fig: foldings}
\end{figure}
Un senzill recompte de vèrtexs i arcs proporciona el resultat següent, que serà rellevant més endavant.

\begin{rem} \label{rem: rank folding}
Si $\Ati$ és finit i $\Ati \xtr{} \Ati'$ és un plegament elemental de Stallings, aleshores:
 \begin{equation}
  \rk(\Ati')
  \,=\,
  \bigg\{ \begin{array}{lll}
  \rk(\Ati) & & \text{si $\Ati \xtr{} \Ati'$ és obert,} \\
  \rk(\Ati)-1 & & \text{si $\Ati \xtr{} \Ati'$ és tancat.}
  \end{array}     
 \end{equation}
\end{rem}

És clar que els plegaments de Stallings no modifiquen el subgrup reconegut.

\begin{lem}\label{lem: grrec}
Si $\Ati \!\xtr{}\! \Ati'$ és un plegament de Stallings, aleshores $\gen{\Ati'}=\gen{\Ati}$.\qed
\end{lem}

Llavors, si $\Ati$ és un $A$-autòmat involutiu, finit reconeixent $H$, el \defin{procés de Stallings} per convertir $\Ati$ en determinista és com segueix: si $\Ati$ ja és determinista, no cal fer res; en cas contrari, detectem successivament possibles plegaments (oberts o tancats) i els realitzem fins que no en quedi cap de disponible, és a dir, fins a obtenir un autòmat determinista. Observeu que realitzant un plegament podem provocar-ne de nous i fins i tot \emph{augmentar} el nombre de plegaments visibles a fer. Malgrat això, està garantit que arribarem a una situació determinista amb un nombre finit de passes ja que l'autòmat inicial $\Ati$ és finit, i en cada plegament el nombre d'arcs decreix exactament en una unitat. Un cop arribem a un $A$-autòmat determinista, prenem el cor i designem per $\Ati'$ el $A$-autòmat reduït resultant. Observeu que, pel \Cref{lem: grrec} i l'\Cref{rem: <Ati> = <core> = <core*>}, $\gen{\Ati'}=\gen{\Ati}$. I pel \Cref{cor: gen iff isom}, $\Ati'$ ha de ser (isomorf a) l'autòmat de Stallings $\stallings(H,A)$. El procés descrit s'anomena (una) \defin{seqüència de plegaments de Stallings} per a $\Ati$:
 \begin{equation*}
\Ati = \Ati_{\hspace{-1pt}0} \xtr{\!\!\! \phi_1 \!\!}\Ati_{\hspace{-1pt}1} \xtr{\!\!\! \phi_2 \!\!} \cdots \xtr{\!\!\! \phi_p \!\!} \Ati_{\hspace{-1pt}p} \xtr{\!\! \core} \!\core(\Ati_{\hspace{-1pt}p}) \isom \stallings(H,A).
 \end{equation*}
Observeu que hi pot haver, en principi, diferents plegaments --- i per tant diferents seqüències de plegaments--- disponibles a $\Ati$. Ara bé, com que l'autòmat final és reduït, el \Cref{cor: gen iff isom} ens garanteix que l'objecte final obtingut serà sempre el mateix (mòdul isomorfisme), \emph{independentment de la seqüència de plegaments seguida}. 

Construint una seqüència de plegaments de Stallings per a l'autòmat flor $\flower(S)$, obtenim una manera efectiva de calcular $\stallings(H,A)$ a partir de qualsevol conjunt de generadors finit $S$ per a $H$. No és difícil demostrar que en aplicar una seqüència de plegaments de Stallings a l'autòmat flor $\flower(S)$ d'un conjunt $S$ de paraules \emph{reduïdes}, cap dels plegaments de la seqüència trenca el caràcter cor de l'autòmat inicial i, per tant, no és necessari realitzar el pas final de prendre el cor: 
 \begin{equation*}
\flower(S) = \Ati_{\hspace{-1pt}0} \xtr{\!\!\! \phi_1 \!\!}\Ati_{\hspace{-1pt}1} \xtr{\!\!\! \phi_2 \!\!} \cdots \xtr{\!\!\! \phi_p \!\!} \Ati_{\hspace{-1pt}p} =\stallings(H,A).
 \end{equation*}
Observeu finalment, que el resultat de la seqüència de plegaments, $\stallings(H,A)$, depèn canònicament del subgrup $H$ i, per tant, \emph{també és independent del sistema de generadors finit $S$} de $H$ amb què haguem començat el procés (mentre que $\flower(S)$ depèn fortament de $S$). Això conclou la computació de la bijecció~\eqref{eq: Stallings bijection1}.

Incorporant a l'enunciat del \Cref{thm: Stallings bijection} els resultats que acabem de demostrar, obtenim el teorema principal d'aquesta secció.

\begin{thm}[\citenr{stallings_topology_1983}]\label{thm: Stallings bijection2}
La funció
 \begin{equation}\label{eq: Stallings bijection}
\begin{array}{rcl} \operatorname{St}\colon \set{\text{subgrups de } \Free[A] } & \to & \set{\text{(classes d'isomorfia de) $A$-autòmats reduïts}} \\ H & \mapsto & \stallings(H,A)
\end{array}
 \end{equation}
és una bijecció, amb inversa $\gen{\Ati} \mapsfrom \Ati$. A més, els subgrups finitament generats corresponen precisament als autòmats reduïts finits i, en aquest cas, la bijecció és computable en ambdós sentits. \qed
\end{thm}

\begin{exm}\label{ex: Stallings}
Siguin $\Free[2]$ el grup lliure sobre $A=\{a,b\}$, $S=\set{u_1,u_2,u_3} \subseteq \Free[2]$ i ${H=\gen{S}\leqslant \Free[2]}$ el subgrup generat pels elements $u_1=a^3$, $u_2=abab^{-1}$, i $u_3=a^{-1}bab^{-1}$. Per obtenir una base (i, per tant, el rang) de $H$, calcularem primer $\stallings(H,A)$. Comencem construint l'autòmat flor $\Ati_0=\flower(S)$, i anem fent plegaments fins a obtenir l'autòmat de Stallings $\stallings(H,A)$: a la \Cref{fig: Stallings sequence} teniu una possible seqüència de Stallings $\Ati_{\hspace{-1pt} 0} \xtr{\,} \Ati_{\hspace{-1pt} 1} \xtr{\,} \Ati_{\hspace{-1pt} 2} \xtr{\,} \Ati_{\hspace{-1pt} 3} \xtr{\,} \Ati_{\hspace{-1pt} 4} \xtr{\,} \Ati_{\hspace{-1pt} 5} \xtr{\,} \Ati_{\hspace{-1pt} 6}=\stallings(H)$, tot indicant, en cada pas, amb traç gruixut, els arcs que es pleguen al pas següent (observeu que el primer pas són en realitat tres plegaments fets simultàniament):
\begin{figure}[H]
\centering
\begin{tikzpicture}[shorten >=1pt, node distance=1cm and 1cm, on grid,auto,auto,>=stealth']
\begin{scope}[rotate=90]
\newcommand{\rad}{0.85}
\newcommand{\Rad}{1.45}
\node[state,accepting] (bp)  {};

\foreach \x in {0,...,5}
{
\node[state] (\x) at ($(90+\x*360/6:\rad cm)$) {};
}

\node[state] (34) at ($(180+120:\Rad cm)$) {};
\node[state] (50) at ($(180+2400:\Rad cm)$) {};
\node[] (l0) [above = 1.2 of bp] {$\Ati_{\hspace{-1pt}0} \!=\! \flower(S)$};

\draw[->,blue] (bp) edge (1);
\draw[->,blue] (1) edge (2);
\draw[->,blue] (2) edge (bp);

\draw[->,blue] (bp) edge node[above,pos=0.45] {\scriptsize{$a$}}  (0);
\draw[->,red] (0) edge[thick] (50);
\draw[->,blue] (50) edge[thick] (5);
\draw[->,red] (bp) edge[thick] node[right,pos=0.5] {\scriptsize{$b$}} (5);

\draw[->,red] (bp) edge[thick] (4);
\draw[->,blue] (34) edge[thick] (4);
\draw[->,red] (3) edge[thick] (34);
\draw[->,blue] (3) edge (bp);

\node[] (i) [right = 1.4 of bp]{};
    \node[] (f) [right = 0.9 of i] {};
    \path[->]
        (i) edge[bend left] (f);
\end{scope}

\begin{scope}[xshift=3.5cm]         

\newcommand{\rad}{0.85}
\newcommand{\Rad}{1.45}
\node[state,accepting] (0)  {};
\node[state] (1) [right = \rad of 0] {};
\node[state] (2) [right = \rad of 1] {};
\node[state] (3) [right = \rad of 2] {};
\node[state] (5) at ($(180+30:\rad cm)$) {};
\node[state] (6) at ($(180-30:\rad cm)$) {};
\node[] (l2) [above right = 0.85 and \rad/2 of 1] {$\Ati_{\hspace{-1pt}1}$};

\draw[->,blue] (0) edge[thick] (6);
\draw[->,blue] (6) edge (5);
\draw[->,blue] (5) edge (0);
\draw[->,red] (0) edge (1);
\draw[->,blue] (2) edge (1);
\draw[->,red] (3) edge (2);
\draw[->,blue,bend left] (0) edge[thick] (3);
\draw[->,blue,bend left] (3) edge (0);

\node[] (i) [right = 3 of 0]{};
    \node[] (f) [right = 0.9 of i] {};
    \path[->]
        (i) edge[bend left] (f);
\end{scope}

\begin{scope}[xshift=8cm]         

\newcommand{\rad}{0.85}
\newcommand{\Rad}{1.45}
\node[state,accepting] (0)  {};
\node[state] (1) [right = \rad of 0] {};
\node[state] (2) [right = \rad of 1] {};
\node[state] (3) [right = \rad of 2] {};
\node[state] (4) [below right = \rad and \rad/2 of 1] {};
\node[] (l2) [above right = 0.85 and \rad/2 of 1] {$\Ati_{\hspace{-1pt}2}$};

\draw[->,red] (0) edge (1);
\draw[->,blue] (2) edge (1);
\draw[->,red] (3) edge (2);
\draw[->,blue,bend left] (0) edge (3);
\draw[->,blue,bend left] (3) edge[thick] (0);
\draw[->,blue,bend left] (3) edge (4);
\draw[->,blue,bend left] (4) edge[thick] (0);

\node[] (i) [below right = \rad and \rad/2 of 2]{};
    \node[] (f) [below = 0.9 of i] {};
    \path[->]
        (i) edge[bend left] (f);
\end{scope}

\begin{scope}[xshift=8cm,yshift=-2.3cm]         

\newcommand{\rad}{0.85}
\newcommand{\Rad}{1.45}
\node[state,accepting] (0)  {};
\node[state] (1) [right = \rad of 0] {};
\node[state] (2) [right = \rad of 1] {};
\node[state] (3) [right = \rad of 2] {};
\node[] (l2) [above right = 0.75 and \rad/2 of 1] {$\Ati_{\hspace{-1pt}3}$};

\draw[->,red] (0) edge (1);
\draw[->,blue] (2) edge (1);
\draw[->,red] (3) edge (2);
\draw[->,blue,bend left] (0) edge (3);
\draw[->,blue,bend left] (3) edge[thick] (0);
\draw[->,blue,loop right,min distance=8 * \rad mm,in=30,out=-30] (3) edge[thick] (3);

\node[] (i) [left = 0.7 of 0]{};
    \node[] (f) [left = 0.9 of i] {};
    \path[->]
        (i) edge[bend right] (f);
\end{scope}

\begin{scope}[xshift=4cm,yshift=-2.3cm]         

\newcommand{\rad}{0.85}
\newcommand{\Rad}{1.45}
\node[state,accepting] (0)  {};
\node[state] (1) [right = \rad of 0] {};
\node[state] (2) [right = \rad of 1] {};
\node[] (l2) [above = 0.55 of 1] {$\Ati_{\hspace{-1pt}4}$};
\draw[->,blue,loop right,min distance=8 * \rad mm,in=160,out=100] (0) edge[thick] (0);
\draw[->,blue,loop right,min distance=8 * \rad mm,in=-160,out=-100] (0) edge[thick] (0);

\draw[->,red] (0) edge (1);
\draw[->,blue] (2) edge (1);
\draw[->,red,bend left] (0) edge (2);

\node[] (i) [left = 0.8 of 0]{};
    \node[] (f) [left = 0.9 of i] {};
    \path[->]
        (i) edge[bend right] (f);
\end{scope}

\begin{scope}[xshift=-0.3cm,yshift=-2.3cm]         

\newcommand{\rad}{0.85}
\newcommand{\Rad}{1.45}
\node[state,accepting] (0)  {};
\node[state] (1) [right = \rad of 0] {};
\node[state] (2) [right = \rad of 1] {};
\node[] (l2) [above = 0.55 of 1] {$\Ati_{\hspace{-1pt}5}$};
\draw[->,blue,loop right,min distance=8 * \rad mm,in=150,out=210] (0) edge (0);

\draw[->,red] (0) edge[thick] (1);
\draw[->,blue] (2) edge (1);
\draw[->,red,bend left] (0) edge[thick] (2);

\node[] (i) [below = 0.4 of 0]{};
    \node[] (f) [below = 0.9 of i] {};
    \path[->]
        (i) edge[bend right] (f);
\end{scope}

\begin{scope}[xshift=0.3cm,yshift=-4cm]         

\newcommand{\rad}{0.85}
\newcommand{\Rad}{1.45}
\node[state,accepting] (0)  {};
\node[state] (1) [right = \rad of 0] {};
\node[] (l6) [above left= 0.5 and 0.3 of 1] {$\Ati_{\hspace{-1pt}6} \!=\! \stallings(H)$};

\draw[->,blue,loop right,min distance=8 * \rad mm,in=150,out=210] (0) edge node[left,pos=0.5] {\scriptsize{$a$}} (0);
\draw[->,blue,loop right,min distance=8 * \rad mm,in=-30,out=30] (1) edge (1);

\draw[->,red] (0) edge node[above,pos=0.45] {\scriptsize{$b$}} (1);
\end{scope}
\end{tikzpicture}
\vspace{-5pt}
\caption{Una seqüència de plegaments de Stallings pel subgrup $H=\gen{a^3,\, abab^{-1},\, a^{-1}bab^{-1}}\leqslant \Free_{\set{a,b}}$} 
\label{fig: Stallings sequence}
\end{figure}
Finalment, prenent ${\bp \R{\xarc{ \R{b}\, }} \bullet}$ com a arbre d'expansió de $\stallings(H)$, la \Cref{prop: S_T}(ii) ens diu que $\set{a, b a b^{-1}}$ és una base del subgrup $H$; en particular, $\rk (H)=2$. 
\end{exm}

Tornem per un moment al morfisme $\widetilde{\ell}_{\Ati}$ definit a \eqref{eq: pi -> gen}. Dèiem més amunt que no era un morfisme injectiu, en general. I, efectivament, pel $A$-autòmat $\Ati_0=\flower(S)$ de l'exemple anterior (al vèrtex base $\bp$) no ho és ja que $\pi_{\bp}(\Ati_0)$ és un grup lliure de rang $\rk(\Ati_0)=1-\card\Verts\Ati_0+\card\Edgs^{+}\Ati_0=1-9+11=3=\card S$, mentre que $\gen{\Ati_0}=\gen{S}=H
$ és un (sub)grup (de $\Free[A]$) lliure de rang 2. Això vol dir que els tres generadors inicials $u_1, u_2, u_3$ de $H$ \emph{no} són base de $H$, és a dir, no són independents; per tant hi ha d'haver un producte reduït i no buit d'ells i els seus inversos que doni l'element trivial, $w(u_1, u_2, u_3)=1$. Això és, exactament, la contrapartida algebraica del fet gràfic que $\widetilde{\ell}_{\Ati}$ no sigui injectiu: hi ha un element $1\neq [\walki]\in \pi_{\bp} (\Ati_0)$ (\ie un $\bp$-camí reduït i no buit al $A$-autòmat flor $\Ati_0$) tal que $\rlab(\walki)=1$ (\ie que llegeixi l'etiqueta reduïda trivial). Fent càlculs veiem, per exemple, $u_2 u_3^{-1}u_1^{-1}u_2 u_3^{-1}u_1^{-1}u_2 u_3^{-1}=1$. Podem anar encara una mica més enllà en aquest para\lgem elisme entre àlgebra i geometria, i quantificar la no injectivitat del morfisme $\widetilde{\ell}_{\Ati}\colon \pi_{\bp} (\Ati) \onto \gen{\Ati} \leqslant\Free[A]$, $[\walki] \mapsto \rlab(\walki)$, des dels dos punts de vista: una manera algebraica de fer-ho és mirant el nombre mínim de generadors de $\ker (\widetilde{\ell}_{\Ati})$ com a subgrup normal de $\pi_{\bp} (\Ati)$; com veurem a continuació, la contrapartida gràfica és el nombre de plegaments tancats que apareixen en una (i per tant en qualsevol) seqüència de plegaments de Stallings de~$\Ati$ a $\stallings(H,A)$. 

\begin{rem}
Si $\Ati$ és determinista aleshores $\widetilde{\ell}_{\Ati}$ és injectiu.
\end{rem}

És clar que la pèrdua de rang gràfic en qualsevol seqüència de Stallings per un $A$-autòmat finit $\Ati$ és precisament el nombre de plegaments tancats, ja que els plegaments oberts no alteren el rang, i els tancats el redueixen en una unitat (vegeu l'\Cref{rem: rank folding}). Es pot usar un argument geomètric-algebraic per demostrar que aquest nombre també coincideix amb la quantitat mínima de generadors de $\ker (\widetilde{\ell}_{\Ati})$ com a subgrup normal de $\pi_{\bp} (\Ati)$ (per raons d'espai, ometem els detalls tècnics d'aquesta demostració).

\begin{prop}\label{prop: loss}
Sigui $\Ati$ un $A$-autòmat involutiu, connex, i finit, reconeixent el subgrup $H\leqslant \Free[A]$, i sigui  
  \begin{equation} \label{eq: seq loss}
  \Ati  \xtr{\!\!\! \phi_1 \!\!}\Ati_{\hspace{-1pt}1} \xtr{\!\!\! \phi_2 \!\!} \cdots \xtr{\!\!\! \phi_p \!\!} \Ati_{\hspace{-1pt}p}
  \xtr{ \!\core } \core(\Ati_{\hspace{-1pt}p}) =\stallings(H,A)  
  \end{equation}
una seqüència de plegaments de Stallings per a $\Ati$. Aleshores, els tres nombres següents coincideixen:
 \begin{enumerate}[dep]
\item \label{item: card closed} el nombre de plegaments tancats a \eqref{eq: seq loss};
\item \label{item: rk diff} $\rk(\Ati) - \rk(\stallings(H,A))$;
\item \label{item: min gen ker} el nombre mínim de generadors de $\ker (\widetilde{\ell}_{\Ati})$ \emph{com a subgrup normal} de $\pi_{\bp} (\Ati)$. \qed
 \end{enumerate}
\end{prop}

A aquest valor --- independent de la seqüència concreta de plegaments i només depenent de $\Ati$ --- se l'anomena la \defin{pèrdua} de $\Ati$, designada per $\loss(\Ati)$. Podem ampliar aquesta definició per a cobrir $A$-autòmats infinits de la manera següent:

\begin{defn}
Sigui $\Ati$ un $A$-autòmat involutiu i connex. Definim la \defin{pèrdua} de~$\Ati$, designada per $\loss(\Ati)$, com el suprem (que pot ser $\infty$, si $\Ati$ és infinit) del nombre de plegaments tancats, pres sobre el conjunt de totes les seqüències (finites) de plegaments de Stallings començant a $\Ati$. 
\end{defn}

Observeu que si $\Ati$ és finit les tals seqüències de llargada maximal són, precisament, les que acaben en un $\Ati_{\hspace{-1pt}p}$ determinista; i com que $\core(\Ati_p)=\stallings(\gen{\Ati})$, totes elles tenen el mateix nombre de plegaments tancats, per la \Cref{prop: loss}. Per tant, en el cas finit, el suprem a la definició de $\loss(\Ati)$ és un màxim, i es pren sobre totes les seqüències de plegaments de Stallings començant a $\Ati$. 

\medskip

Acabem la secció amb un exemple per i\lgem ustrar que el canvi de la base de referència pot canviar dràsticament l'aspecte de l'autòmat de Stallings. De fet entendre com es comporta la funció $\stallings(H,A) \mapsto \stallings(H,B)$ (on $A,B$ són bases de $\Free$) és un dels grans reptes encara per aclarir de la Teoria de Stallings.

\begin{exm}[Example 2.2 a \cite{miasnikov_algebraic_2007}]
Sigui $\Free[3]$ un grup lliure amb base $A=\set{a,b,c}$ i sigui $H=\gen{ab,acba}$. És fàcil veure que $B=\set{a',b',c'}$ és també una base de $\Free[3]$, on $a'=a$, $b'=ab$, i $c'=acba$. Els autòmats de Stallings de~$H$ respecte de $A$ i $B$ estan representats a la figura següent.
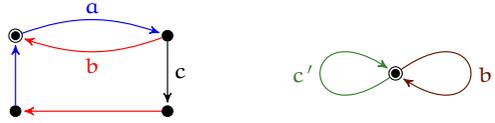
\begin{figure}[H] 
\centering
  \begin{tikzpicture}[shorten >=1pt, node distance=1 and 2, on grid,auto,>=stealth']
  \begin{scope}
   \node[state,accepting] (0) {};
   \node[state] (1) [right = of 0]{};
   \node[state] (2) [below = of 1]{};
   \node[state] (3) [below = of 0]{};

    \path[->]
        (0) edge[blue, bend left = 20]
            node[above] {\scriptsize{$a$}}
            (1);
            
    \path[->]
        (1) edge[red, bend left = 20]
            node[below] {\scriptsize{$b$}}
            (0);
            
    \path[->]
        (1) edge[black]
            node[right] {\scriptsize{$c$}}
            (2);
            
    \path[->] (2) edge[red] (3);
    \path[->] (3) edge[blue] (0);
    \end{scope}
    
    \begin{scope}[xshift=5 cm,yshift=-0.5 cm]
    \node[state,accepting] (0) {};
    
    \path[->]
        (0) edge[Sepia,loop left,min distance=15mm,in=-35,out=35]
            node[right] {\scriptsize{$b'$}}
            (0);
            
    \path[->]
        (0) edge[OliveGreen,loop left,min distance=15mm,in=180-35,out=180+35]
            node[left] {\scriptsize{$c'$}}
            (0);            
    \end{scope}
\end{tikzpicture}
\vspace{-10pt}
\end{figure}
\end{exm}




\section{Aplicacions} \label{sec: aplicacions}

En aquesta secció presentem algunes de les aplicacions més rellevants de la teoria dels autòmats de Stallings, desenvolupada a la \Cref{sec: Stallings bijection}. Començarem amb alguns teoremes clàssics sobre els subgrups del grup lliure que emergeixen de forma natural --- de vegades com a simples reinterpretacions --- dels resultats de la Teoria de Stallings, i anirem progressant cap a aplicacions més elaborades, com ara el problema de la intersecció, font d'una de les qüestions més famoses del darrer segle en teoria geomètrica de grups.

\subsection{Estructura dels subgrups} \label{ssec: Nielsen-Schreier}
Tal i com hem vist a la \Cref{sec: Stallings bijection}, una de les primeres aplicacions
del Teorema de Stallings (\Cref{thm: Stallings bijection2}) és el Teorema de Nielsen--Schreier (\Cref{thm: Nielsen-Schreier}).
Fixeu-vos que la bijecció~\eqref{eq: Stallings bijection} ens permet no només garantir que tots els subgrups d'un grup lliure són, de nou, lliures  sinó també descriure exactament (les classes d'isomorfia de) els subgrups d'un grup lliure donat: com que d'autòmats de Stallings sobre un alfabet de dues lletres n'hi ha de qualsevol rang fins a numerable (inclòs), la descripció següent se segueix immediatament.

\begin{cor} \label{cor: ranks of subgroups}
Sigui $\Free[\kappa]$ el grup lliure de rang $\kappa \in [2,\aleph_0]$. Aleshores, per a tot cardinal $\mu \in [0,\aleph_0]$ existeix un subgrup $H\leqslant \Free[\kappa]$ tal que $H\simeq \Free[\mu]$. \qed
\end{cor}

En particular, subgrups de qualsevol rang finit i, fins i tot, de rang infinit numerable, poden aparèixer com a subgrups de $\Free[2]$; a continuació presentem dos exemples clàssics d'aquest fet. Aquest és un comportament radicalment diferent del de l'àlgebra lineal (un $K$-espai vectorial de dimensió $n$ \emph{només} conté subespais de dimensions  $0,1,\ldots,n$), que confereix al reticle de subgrups d'un grup lliure un cert caràcter fractal.

\begin{exm}
Sigui $\Free_{\set{a,b}} = \pres{a,b}{-}$ un grup lliure de rang $2$. Aleshores, la clausura normal de $b$, designada per $\normalcl{b}$, és un subgrup de $\Free_{\set{a,b}}$ de rang infinit amb autòmat de Stallings (de rang infinit) $\stallings(\normalcl{b})$ representat a la \Cref{fig: Stallings <<b>>}. Prenent com a arbre d'expansió el marcat amb traç gruixut (de fet, l'únic possible) obtenim la base $\set{a^kba^{-k}\st k\in \ZZ}$.
\begin{figure}[H] 
\centering
  \begin{tikzpicture}[shorten >=1pt, node distance=1.2 and 1.2, on grid,auto,>=stealth']
   \node[state,accepting] (0) {};
   \node[state] (1) [right = of 0]{};
   \node[state] (2) [right = of 1]{};
   \node[state] (3) [right = of 2]{};
   \node[] (4) [right = 1.5 of 3]{$\cdots$};
   \node[] (c) [right = .5 of 4]{};

   \node[state] (-1) [left = of 0]{};
   \node[state] (-2) [left = of -1]{};
   \node[state] (-3) [left = of -2]{};
   \node[] (-4) [left = 1.5 of -3]{$\cdots$};

   \path[->]
        (0) edge[loop above,red,min distance=10mm,in=55,out=125]
            node[] {\scriptsize{$b$}}
            (0)
            edge[blue,thick]
            node[below] {\scriptsize{$a$}}
            (1);

    \path[->]
        (1) edge[loop above,red,min distance=10mm,in=55,out=125]
            (1)
            edge[blue,thick]
            (2);

    \path[->]
        (2) edge[loop above,red,min distance=10mm,in=55,out=125]
            (2)
            edge[blue,thick]
            (3);

    \path[->]
        (3) edge[loop above,red,min distance=10mm,in=55,out=125]
            (3)
            edge[blue,thick]
            (4);

    \path[->]
        (-1) edge[loop above,red,min distance=10mm,in=55,out=125]
            (-1)
            edge[blue,thick]
            (0);

    \path[->]
        (-2) edge[loop above,red,min distance=10mm,in=55,out=125]
            (-2)
            edge[blue,thick]
            (-1);

    \path[->]
        (-3) edge[loop above,red,min distance=10mm,in=55,out=125]
            (-3)
            edge[blue,thick]
            (-2);

    \path[->]
        (-4) edge[blue,thick]
            (-3);
\end{tikzpicture}
\caption{L'autòmat de Stallings de $\normalcl{b} \leqslant \Free_{\set{a,b}}$}
\label{fig: Stallings <<b>>}
\end{figure}
\end{exm}

\begin{exm}
Un altre exemple típic és el commutador $H=\Comm{\Free_{\set{a,b}}}\normaleq  \Free_{\set{a,b}}$, \ie el subgrup normal de $\Free_{\set{a,b}}$ generat per tots els elements de la forma $[v,w]=v^{-1}w^{-1}vw$, amb $v,w\in \Free[\set{a,b}]$. El subgrup $H$ té rang infinit ja que, com tots els subgrups normals, el seu autòmat de Stallings $\stallings(H)=\schreier(H)$ és el dígraf de Cayley del grup quocient $\Free_{\set{a,b}}/H\simeq \mathbb{Z}^2$, respecte dels generadors $\{[a],\, [b]\}$, \ie respecte la base canònica de $\ZZ^2$: la graella bidimensional infinita, amb els arcs horitzontals etiquetats per $a$, i els arcs verticals etiquetats per $b$; vegeu la \Cref{fig: commutator}. Proposem al lector descriure una base del commutador $H$ a partir del vostre arbre d'expansió preferit.
\end{exm}

\vspace{-10pt}
\begin{figure}[H]
\centering
\begin{tikzpicture}[shorten >=1pt, node distance=1.2cm and 2cm, on grid,auto,auto,>=stealth']

\newcommand{\dx}{0.75}
\newcommand{\dy}{0.7}
\node[] (0)  {};

\foreach \x in {-2,...,2}
\foreach \y in {-2,...,2} 
{
\node[state] (\x!\y) [above right = \y*\dy and \x*\dx of 0] {};
}
\node[state,accepting] (0!0){};

 \foreach \x in {-3,...,3} 
{
\node (\x!3) [above right = 3*\dy and \x*\dx of 0] {};
\node (\x!-3) [above right = -3*\dy and \x*\dx of 0] {};
}

\foreach \y in {-2,...,2} 
{
\node (3!\y) [above right = \y*\dy and 3*\dx of 0] {};
\node (-3!\y) [above right = \y* \dy and -3*\dx of 0] {};
}

\draw[draw=none,red] (0!0) edge node[left,pos=0.45] {$b$} (0!1);
\draw[draw=none,blue] (0!0) edge node[above,pos=0.45] {$a$} (1!0);

\foreach \x [evaluate = \x as \xx using int(\x+1)] in {-3,...,2}
  \foreach \y in {-2,...,2}
{
\ifthenelse{\x=-3 \OR \x=2 \OR \y=-3 \OR \y=3}
{\draw[->,blue,dotted] (\x!\y) edge (\xx!\y);}
{\draw[->,blue] (\x!\y) edge (\xx!\y);}
}

\foreach \y [evaluate = \y as \yy using int(\y+1)] in {-3,...,2}
  \foreach \x in {-2,...,2}
{
\ifthenelse{\x=-3 \OR \x=3 \OR \y=-3 \OR \y=2}
{\draw[->,red,dotted] (\x!\y) edge (\x!\yy);}
{\draw[->,red] (\x!\y) edge (\x!\yy);}
}

\end{tikzpicture}
\vspace{-10pt}
\caption{L'autòmat de Stallings del commutador $\Comm{\Free_{\set{a,b}}}$}
\label{fig: commutator}
\end{figure}

\subsection{El problema de la pertinença} \label{ssec: MP}

En aquesta secció estudiarem i donarem una solució general per al problema de la pertinença en grups lliures. A continuació l'enunciem en general per a un grup finitament presentat~$G = \pres{A}{R}$.

\begin{named}[Problema de la pertinença a $G=\pres{A}{R}$, $\MP(G)$]
Donada una família finita $u,v_1,\ldots,v_k$ de paraules en els generadors de $G$, decidir si (l'element $g\in G$ representat per) $u$ pertany al subgrup $\gen{v_1,\ldots,v_k}_G\leqslant G$; i, en cas afirmatiu, trobar una expressió per a $u$ com a producte dels $v_i^{\pm}$'s.
\end{named}

Situem-nos al grup lliure $\Free[A]$; els elements $u,v_1,\ldots,v_k$ són paraules (que podem suposar reduïdes, ja que la reducció és trivialment algorísmica) en l'alfabet $A$, i hem de decidir si $u\in H=\langle v_1,\ldots,v_k \rangle$. 

A tal efecte, primer calculem l'autòmat de Stallings $\stallings(H)$ tal com hem vist a la secció anterior: construïm l'autòmat flor, $\flower(S)$, de pètals $S=\set{v_1,\ldots,v_k}$, i anem fent plegaments successius (en qualsevol ordre) fins a obtenir $\stallings(H)$, 
 \begin{equation}\label{torre}
 \flower(S) = \Ati_{\hspace{-1pt}0} \xtr{} \Ati_{\hspace{-1pt}1} \xtr{} \cdots \xtr{} \Ati_{\hspace{-1pt}p-1} \xtr{} \Ati_{\hspace{-1pt}p} =\stallings(H).
 \end{equation}
Per les propietats (i) i (ii) a la \Cref{prop: propietats St}, sabem que $\stallings(H)$ és un autòmat determinista que reconeix el subgrup $H$. Per tant, $u\in \gen{H}$ si i només si podem llegir la paraula reduïda $u$ com a etiqueta d'algun $\bp$-camí a $\stallings(H)$. Però, gràcies al determinisme de $\stallings(H)$, això és molt fàcil de comprovar: començant a $\bp$, anem resseguint l'únic possible camí etiquetat amb les lletres consecutives de $u=a_{i_1}^{\epsilon_1} a_{i_2}^{\epsilon_2} \cdots a_{i_l}^{\epsilon_l}$; si en un determinat punt no podem llegir la lletra següent, deduïm que $u\not\in H$; si podem completar la lectura de $u$ amb un camí que no acaba a $\bp$ deduïm igualment que $u\not\in H$; i finalment, si completem la lectura de $u$ mitjançant un camí tancat que retorna a $\bp$ (\ie un $\bp$-camí), concloem que $u\in H$. Donat que aquest procediment cobreix tots els desenllaços possibles, això ens permet decidir algorísmicament si~$u\in H$. 

La segona part de l'algorisme consisteix en suposar que la resposta és afirmativa (\ie que $u\in H$ i tenim aquesta paraula realitzada com l'etiqueta d'un $\bp$-camí reduït $\walki$ a l'autòmat $\stallings(H)$), i buscar una manera efectiva d'expressar $u$ en termes dels generadors originals $v_1,\ldots,v_k$. De pas, entendrem quan aquesta expressió és única i quan no. La idea és la següent: anomenem $\walki_{p}=\walki$ i l'anem \emph{elevant enrere}, plegament a plegament, al llarg de la seqüència de plegaments usada per calcular $\stallings(H)$, fins a obtenir un $\bp$-camí reduït $\walki_{ 0}$ a l'autòmat flor $\flower(S) = \Ati_{\hspace{-1pt} 0}$: 
 \begin{equation} \label{eq: elevation}
\begin{array}{p{32pt}ccccccccl}
    $\flower(S) \, =$ & \Ati_{\hspace{-1pt} 0} & \xtr{} & \Ati_{\hspace{-1pt} 1} & \xtr{} & \cdots &  \xtr{} & \Ati_{\hspace{-1pt}p-1} & \xtr{} & \Ati_{\hspace{-1pt} p} =\,\stallings(H) \\
&\walki_{ 0} & \mapsfrom & \walki_{ 1} & \mapsfrom & \cdots & \mapsfrom & \walki_{p-1} &  \mapsfrom& \walki_{ p} =\walki. 
\end{array}
 \end{equation}
Per una banda, gràcies a la forma de l'autòmat $\flower(S)$, aquest últim camí $\walki_{ 0}$ no és res més que una successió de pètals o pètals inversos, és a dir, un camí reduït amb etiqueta no necessàriament reduïda a $(A^{\pm})^*$, \emph{però igual a una paraula reduïda en~$v_1^{\pm 1},\ldots,v_k^{\pm 1}$}. Per altra banda, farem l'elevació $\walki_{ i+1} \mapsto \walki_{ i}$ a través de cada plegament elemental de manera que el resultat sigui un camí reduït amb etiqueta $\lab(\walki_{ i})=u_i$ no necessàriament reduïda, però tal que $\red{u_i} = \red{u}$; és a dir, representant el mateix element que $u$ a $\Free[A]$. Per tant, el procés d'elevació de $\walki$ produirà una seqüència de paraules $u=u_p, u_{p-1},\ldots ,u_1, u_0\in \IM{A}$, totes amb etiqueta reduïda igual a $\red{u}\in \Free[A]$, introduint una o diverses cance\lgem acions a cada pas, de manera que $u_0$ sigui l'etiqueta d'un $\bp$-camí reduït de $\flower(S)$, és a dir, un cert producte dels $v_i^{\pm 1}$'s, que és el que estem buscant.

Fixem-nos doncs en un dels plegaments elementals $\Ati_{\hspace{-1pt} i} \xtr{\, } \Ati_{\hspace{-1pt} i+1}$ de la seqüència~\eqref{torre}, considerem un $\bp$-camí reduït $\walki_{ i+1}$ a $\Ati_{\hspace{-1pt} i+1}$ (amb $u_{i+1}=\ell(\walki_{ i+1})\in \IM{A}$), i estudiem com fer-ne l'elevació a $\Ati_{\hspace{-1pt} i}$. Permutant i invertint lletres de $A$ si cal, podem suposar que el plegament elemental en qüestió consisteix en la identificació de dos arcs diferents, diguem-ne $\edgi_1$ i $\edgi_2$, tals que $\iota \edgi_1 =\iota \edgi_2$ i $\ell(\edgi_1)=\ell(\edgi_2)=a$; vegeu la~\Cref{fig: foldings}. 

Observeu que els autòmats $\Ati_{\hspace{-1pt} i}$ i $\Ati_{\hspace{-1pt} i+1}$ són idèntics excepte en els dos arcs $\edgi_1$ i $\edgi_2$ (que són diferents a $\Ati_{\hspace{-1pt} i}$ i estan identificats a $\Ati_{\hspace{-1pt} i+1}$), i excepte potser en els vèrtexs $\vertii_1=\tau \edgi_1$ i $\vertii_2=\tau \edgi_2$ (que són iguals a $\Ati_{\hspace{-1pt} i+1}$, i \emph{potser} diferents a $\Ati_{\hspace{-1pt} i}$). Si $\vertii_1\neq_{\Ati_{\hspace{-1pt}i}} \vertii_2$ el plegament és obert, i en cas contrari tancat. Això és independent del fet que el vèrtex $\verti=\iota \edgi_1=\iota \edgi_2$ coincideixi o no amb $\vertii_1$ i/ó amb $\vertii_2$ (és a dir, que $\edgi_1$ i/ó $\edgi_2$ siguin o no llaços); aquesta possibilitat complica la visualització del procediment que farem, però no afecta els arguments. Distingim dos casos. 
\begin{itemize}
\item \emph{Cas 1: $\Ati_{\hspace{-1pt} i} \xtr{\,} \Ati_{\hspace{-1pt} i+1}$ és obert (\ie $\vertii_1\neq_{\Ati_{\hspace{-1pt}i}} \vertii_2$)}. Sigui $k\geq 0$ el nombre de visites del camí reduït $\walki_{ i+1}$ al vèrtex identificat $\vertii_1=\vertii_2$; cadascuna d'aquestes visites s'anomena \emph{crítica} si $\walki_{ i+1}$ arriba a $\vertii_1=\vertii_2$ per un arc incident a $\vertii_1$ en $\Ati_{\hspace{-1pt} i}$ i en surt per un d'incident a $\vertii_2$ (o viceversa). És evident que cada segment $\walki$ de $\walki_{ i+1}$ entre visites crítiques consecutives a $\vertii_1=\vertii_2$ s'eleva al camí idèntic $\walki$ de $\Ati_{\hspace{-1pt} i}$ (entenent que cada visita a l'arc plegat $\edgi_1=\edgi_2$ cal substituir-la per $\edgi_1$ ó $\edgi_2$ segons quins siguin els arcs adjacents a $\walki_{i+1}$; observeu que només hi ha una possibilitat). Això ens dóna una elevació de $\walki_{i+1}$ a un ``$\bp$-camí'' de $\Ati_{\hspace{-1pt} i}$ amb la mateixa etiqueta $u_{i+1}$, però amb una discontinuïtat (saltant de $\vertii_1$ a $\vertii_2$, ó viceversa) per cada visita crítica de $\walki_{i+1}$ al vèrtex $\vertii_1=\vertii_2$. Aquest problema el podem solventar fàcilment intercalant el segment $\edgi_1^{-1}\edgi_2$ (ó $\edgi_2^{-1}\edgi_1$) per cadascuna de les visites crítiques, obtenint així el $\bp$-camí reduït $\walki_{ i}$ a $\Ati_{\hspace{-1pt} i}$. A més, com que l'etiqueta d'aquests segments afegits és sempre $a^{-1}a$, l'etiqueta $u_i=\ell(\walki_{ i})$ serà $\ell(\walki_{ i+1})=u_{i+1}$ amb $k'$ segments $a^{-1}a$ intercalats en les posicions de les $k'\leq k$ visites crítiques a $\vertii_1=\vertii_2$; en particular, $\red{u_i}=\red{u_{i+1}}\in \Free_A$.

Cal esmentar una possibilitat degenerada que pot ocórrer quan $\bp=\vertii_1=\vertii_2$ a $\Ati_{\hspace{-1pt} i+1}$ (i, posem per cas, $\bp=\vertii_1$ a $\Ati_{\hspace{-1pt} i}$). En aquesta situació, el $\bp$-camí $\walki_{ i+1}$ pot començar (resp., acabar) amb un arc $\edgi$ incident a $\vertii_2$: en aquests casos, cal afegir el segment $\edgi_1^{-1}\edgi_2$ (resp., $\edgi_2^{-1}\edgi_1$) al començament (resp., final) de $\walki_{ i}$ per tal que l'elevació realment comenci (resp., acabi) a $\bp\in \Ati_{\hspace{-1pt} i}$. 

\item \emph{Cas 2: $\Ati_{\hspace{-1pt} i} \xtr{\,} \Ati_{\hspace{-1pt} i+1}$ és tancat (\ie $\vertii_1 =_{\Ati_{\hspace{-1pt}i}} \vertii_2$)}. La noció de visita crítica al vèrtex $\vertii_1=\vertii_2$ no té sentit ja que ara $\vertii_1=\vertii_2$ també a $\Ati_{\hspace{-1pt} i}$, i l'elevació de $\walki_{ i+1}$ es pot fer d'una tirada (com si tot $\walki_{ i+1}$ fos un sol segment entre visites crítiques): només cal considerar el mateix $\bp$-camí a $\Ati_{\hspace{-1pt} i}$ utilitzant només un qualsevol dels arcs plegats. És més, hi ha infinites possibles elevacions de $\walki_{ i+1}$ (a $\bp$-camins reduïts de $\Ati_{\hspace{-1pt} i}$ llegint $\rlab(\walki_{ i+1})$): per cada factorització $\walki_{ i+1}=\walki'\vertiii\walki''$, i per cada camí reduït $\walki'''$ de $\vertiii$ a $\verti=\iota \edgi_1=\iota \edgi_2$ (resp., a $\vertii=\tau \edgi_1=\tau \edgi_2$), podem elevar $\walki_{ i+1}$ com a $\walki'\vertiii \walki''' \verti (\edgi_1 \vertii \edgi_2^{-1})^k \verti (\walki''')^{-1} \vertiii \walki''$ (resp., $\walki'\vertiii \walki''' \vertii (\edgi_1^{-1} \verti \edgi_2)^k \vertii (\walki''')^{-1} \vertiii \walki''$), on $k\in \ZZ\setmin \{0\}$. Es pot veure que aquesta és la família de totes les elevacions possibles de $\walki_{i+1}$ a $\bp$-camins reduïts de $\Ati_{i}$ llegint $\overline{u_{i+1}}\in \Free[A]$. 
\end{itemize}

Tot aquest procediment demostra el resultat següent.

\begin{thm}
El problema de la pertinença per a grups lliures, $\MP(\Free[A])$, és resoluble. \qed
\end{thm}

Podeu veure i\lgem ustrat aquest mètode gràfic amb la resolució de l'exemple següent. Proposem al lector el seguiment d'aquest mateix mètode per deduir la paraula $v_1v_2^{-1}v_1 (v_1v_2^{-1})^7 v_3^{-1}v_2^{-1}v_3$, que a l'\Cref{ex: membership} ens ha permès d'expressar $u$ en termes de $v_1, v_2, v_3$ (sense explicar-ne l'origen).

\begin{exm}\label{ex: membeship}
Sigui $\Free[2]$ el grup lliure sobre $A=\{a,b\}$ i considerem el subgrup $H=\langle v_1, v_2, v_3 \rangle\leqslant \Free[2]$ generat pels elements $v_1=a^3$, $v_2=abab^{-1}$, i $v_3=a^{-1}bab^{-1}$. Decidiu si l'element $u=a\in \Free[2]$ pertany a $H$ i, en cas afirmatiu, expresseu-lo en termes de $v_1, v_2, v_3$.
\end{exm}

Per començar, construïm l'autòmat $\stallings(H)$, a partir del conjunt de generadors $\{v_1, v_2, v_3\}$; això està fet a l'\Cref{ex: Stallings}; vegeu la \Cref{fig: Stallings sequence}. Observeu que el primer pas ${\mathcal F} \xtr{} \Ati_{\hspace{-1pt} 1}$ són en realitat tres plegaments elementals simultanis (no ens caldrà distinguir-los entre ells). Noteu també que tots els plegaments d'aquesta seqüència són oberts excepte un, $\Ati_{\hspace{-1pt} 4} \xtr{} \Ati_{\hspace{-1pt} 5}$, que és tancat perquè identifica dos $a$-llaços en un de sol. D'aquest procés deduïm que $H=\gen{v_1,v_2,v_3}=\gen{a, bab^{-1}}$; o, millor encara, que $\{a, bab^{-1}\}$ és una base de $H$ (corresponent a l'arbre maximal $\bp \R{\xarc{\ }} \bullet$ de $\Ati_{\hspace{-1pt} 6}$). 

D'aquí veiem clarament que $a\in H$, ja que és l'etiqueta del $\bp$-camí $\walki_6 =a_1$ de $\Ati_{\hspace{-1pt} 6}=\stallings(H)$, assenyalat al dibuix amb una línia discontínua de color gris:
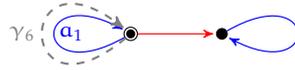
\vspace{-10pt}
\begin{figure}[H]
\centering
\begin{tikzpicture}[shorten >=1pt, node distance=1cm and 1cm, on grid,auto,auto,>=stealth']       
\newcommand{\rad}{1.2}
\node[state,accepting] (0)  {};
\node[state] (1) [right = \rad of 0] {};

\draw[->,blue,loop right,min distance=12 * \rad mm,in=150,out=210] (0) edge 
node[pos=0.5,right=-.3mm] {$a_1$} (0);
\draw[->,gray,loop right,min distance= 16 * \rad mm,in=140,out=220] (0) edge[dashed,thick]
node[left] {\scriptsize{$\walki_6$}} (0);
\draw[->,blue,loop right,min distance= 12 * \rad mm,in=-30,out=30] (1) edge (1);

\draw[->,red] (0) edge (1);
\end{tikzpicture}
\end{figure}
\vspace{-25pt}
A l'$a$-llaç incident a $\bp$ (\ie l'únic arc travessat per $\walki_6$) l'hem anomenat $a_1$; ens referirem als diversos arcs dels autòmats $\Ati_{\hspace{-1pt} i}$ de la seqüència de plegaments amb la lletra de les seves respectives etiquetes ($a$ per les blaves, i $b$ per les vermelles) i uns certs subíndexs que anirem definint dinàmicament; als arcs no usats no els hi assignarem cap  nom. 

Com a primer pas, elevem el camí $\walki_6$ a través de l'últim plegament $\Ati_{\hspace{-1pt} 5} \xtr{\,} \Ati_{\hspace{-1pt} 6}$: com que $\walki_6$ no visita el vèrtex identificat, s'eleva al $\bp$-camí $\walki_5 =a_1$ de $\Ati_{\hspace{-1pt} 5}$ (on, amb un petit abús de llenguatge, li tornem a dir $a_1$ a l'arc corresponent de $\Ati_{\hspace{-1pt} 5}$):
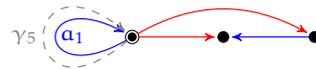
\vspace{-10pt}
\begin{figure}[H]
\centering
\begin{tikzpicture}[shorten >=1pt, node distance=1.2cm and 1.2cm, on grid,auto,auto,>=stealth']       
\newcommand{\rad}{1.2}
\node[state,accepting] (0)  {};
\node[state] (1) [right = \rad of 0] {};
\node[state] (2) [right = \rad of 1] {};

\draw[->,blue,loop right,min distance=12 * \rad mm,in=150,out=210] (0) edge node[pos=0.5,right=-.3mm] {$a_1$} (0);
\draw[->,gray,loop right,min distance= 16 * \rad mm,in=140,out=220] (0) edge[dashed]
node[left] {\scriptsize{$\walki_5$}} (0);

\draw[->,red] (0) edge (1);
\draw[->,blue] (2) edge (1);
\draw[->,red,bend left] (0) edge (2);

\end{tikzpicture}
\end{figure}
\vspace{-25pt}
Observeu ara el plegament $\Ati_{\hspace{-1pt} 4} \xtr{\,} \Ati_{\hspace{-1pt} 5}$, que és tancat; per mantenir una notació coherent, anomenem $a_{11}$ i $a_{12}$ als dos $a$-llaços de $\Ati_{\hspace{-1pt} 4}$ que es pleguen en $a_1$. Per elevar $\walki_5$ a l'autòmat $\Ati_{\hspace{-1pt} 4}$, la manera més senzilla és usant un dels dos $a$-llaços plegats, per exemple, $\walki_4=a_{11}$:
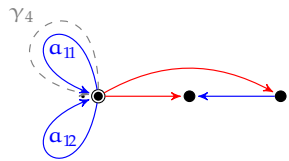
\vspace{-15pt}
\begin{figure}[H]
\centering
\begin{tikzpicture}[shorten >=1pt, node distance=1cm and 1cm, on grid,auto,auto,>=stealth']       
\newcommand{\rad}{1.2}
\node[state,accepting] (0)  {};
\node[state] (1) [right = \rad of 0] {};
\node[state] (2) [right = \rad of 1] {};
\node[] (00) [left = 0.2 of 0] {.};

\draw[->,gray,loop left,min distance= 16 * \rad mm,in=170,out=90] (0) edge[dashed]
node[above left] {\scriptsize{$\walki_4$}} (0);

\draw[->,blue,loop right,min distance=12 * \rad mm,in=160,out=100] (0) edge node[pos=0.5, below right = 0pt and -3pt] {$a_{1\!1}$} (0);
\draw[->,blue,loop right,min distance=12 * \rad mm,in=-160,out=-100] (0) edge node[pos=0.5,above right = 1pt and -3pt] {$a_{1\!2}$} (0);

\draw[->,red] (0) edge (1);
\draw[->,blue] (2) edge (1);
\draw[->,red,bend left] (0) edge (2);

\end{tikzpicture}
\end{figure}
\vspace{-10pt}
(Hi ha, però, altres maneres de fer-ho, $\walki_4=a_{12}$, $\walki_4=a_{11}a_{12}a_{11}^{-1}$, $\walki_4=a_{11}a_{12}^{-1}a_{11}$, etc.) El pas següent és l'elevació a $\Ati_{\hspace{-1pt} 3}$: designem per $a_{121}$ i $a_{122}$ els dos $a$-arcs de $\Ati_{\hspace{-1pt} 3}$ que es pleguen sobre $a_{12}$, i elevem el camí $\walki_4$ a $\walki_3=a_{11}a_{122}^{-1}a_{121}$:
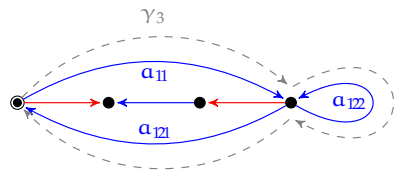
\begin{figure}[H]
\centering
\begin{tikzpicture}[shorten >=1pt, node distance=1cm and 1cm, on grid,auto,auto,>=stealth']       
\newcommand{\rad}{1.2}
\node[state,accepting] (0)  {};
\node[state] (1) [right = \rad of 0] {};
\node[state] (2) [right = \rad of 1] {};
\node[state] (3) [right = \rad of 2] {};
\node[inner sep=0pt] (33) [above = 0.2 of 3]{};
\node[inner sep=0pt] (333) [below = 0.2 of 3]{};

\draw[->,red] (0) edge (1);
\draw[->,blue] (2) edge (1);
\draw[->,red] (3) edge (2);
\draw[->,blue,bend left] (0) edge node[pos=0.5, below] {$a_{1\!1}$} (3);
\draw[->,gray,bend left, in =135,out=45] (0) edge[dashed] node[above = 0.1] {\scriptsize{$\walki_3$}} (33);

\draw[->,blue,bend left] (3) edge node[pos=0.5, above] {$a_{1\!2\!1}$} (0);
\draw[->,gray,bend left, in =135,out=45] (333) edge[dashed] (0);

\draw[->,blue,loop right,min distance=13 * \rad mm,in=30,out=-30] (3) edge node[pos=0.5, left] {$a_{1\!2\!2}$} (3);

\draw[->,gray,loop left,min distance= 17 * \rad mm,in=-30,out=30] (33) edge[dashed] (333);

\end{tikzpicture}
\end{figure}
\vspace{-10pt}
Anomenant $a_{1211}$ i $a_{1212}$ els dos arcs de $\Ati_{\hspace{-1pt} 2}$ que es pleguen sobre $a_{121}$, el camí $\walki_3$ s'eleva a $\Ati_{\hspace{-1pt} 2}$ com a $\walki_2=a_{11}a_{1211}a_{1212}^{-1} a_{122}^{-1}a_{1211}$:
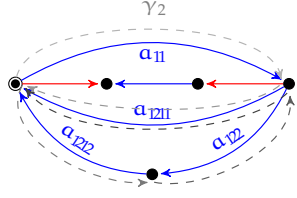
\begin{figure}[H]
\centering
\begin{tikzpicture}[shorten >=1pt, node distance=1cm and 1cm, on grid,auto,auto,>=stealth']       
\newcommand{\rad}{1.2}
\node[state,accepting] (0)  {};
\node[state] (1) [right = \rad of 0] {};
\node[state] (2) [right = \rad of 1] {};
\node[state] (3) [right = \rad of 2] {};
\node[state] (4) [below right = \rad and \rad/2 of 1] {};
\node[inner sep=0pt] (00) [right = 0.1 of 0]{};
\node[inner sep=0pt] (33) [left = 0.1 of 3]{};
\node[inner sep=0pt] (44) [below = 0.1 of 4]{};
\node[inner sep=0pt] (333) [below = 0.1 of 3]{};

\draw[->,red] (0) edge (1);
\draw[->,blue] (2) edge (1);
\draw[->,red] (3) edge (2);
\draw[->,blue,bend left] (0) edge node[pos=0.5, below] {$a_{1\!1}$} (3);
\draw[->,blue,bend left] (3) edge node[pos=0.5, above] {$a_{1\!2\!1\!1}$} (0);
\draw[->,blue,bend left] (3) edge node[sloped, anchor=center, above] {$a_{1\!2\!2}$} (4);
\draw[->,blue,bend left] (4) edge node[sloped, anchor=center, above] {$a_{1\!2\!1\!2}$} (0);

\draw[->,darkgray!40,bend left, in =90,out=80,max distance= 9 mm] (0) edge[dashed] node[above = 0.1] {\color{gray}\scriptsize{$\walki_2$}} (33);
\draw[->,darkgray!55,bend left, in =150,out=40,max distance= 7.5 mm] (33) edge[dashed] (00);
\draw[->,darkgray!70,bend left, in =225,out=-80,max distance= 7.5 mm] (00) edge[dashed] (44);

\draw[->,darkgray!85,bend left, in =250,out=-45,max distance= 7 mm] (44) edge[dashed] (333);
\draw[->,darkgray,bend left, in =120,out=45,max distance= 9 mm] (333) edge[dashed] (00);

\end{tikzpicture}
\end{figure}
El pas següent és elevar $\walki_2$ a $\Ati_{\hspace{-1pt} 1}$ com $\walki_1=a_{111}a_{1211} a_{1212}^{-1} a_{122}^{-1} a_{112}^{-1}a_{111}a_{1211}$:
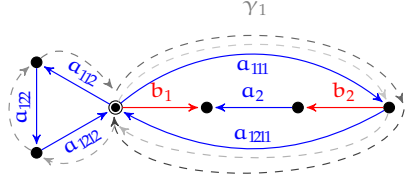
\vspace{-5pt}
\begin{figure}[H]
\centering
\begin{tikzpicture}[shorten >=1pt, node distance=1cm and 1cm, on grid,auto,auto,>=stealth']       
\newcommand{\rad}{1.2}
\node[state,accepting] (0)  {};
\node[state] (1) [right = \rad of 0] {};
\node[state] (2) [right = \rad of 1] {};
\node[state] (3) [right = \rad of 2] {};
\node[state] (5) at ($(180+30:\rad cm)$) {};
\node[state] (6) at ($(180-30:\rad cm)$) {};
\node[inner sep=0pt] (00) [below = 0.2 of 0] {};
\node[inner sep=0pt] (000) [above = 0.2 of 0] {};
\node[inner sep=0pt] (33) [right = 0.2 of 3] {};

\draw[->,blue] (0) edge node[sloped, anchor=center, above] {$a_{1\!1\!2}$} (6);
\draw[->,blue] (6) edge node[sloped, anchor=center, below] {\rotatebox{180}{$a_{1\!2\!2}$}} (5);
\draw[->,blue] (5) edge node[sloped, anchor=center, below] {$a_{1\!2\!1\!2}$} (0);
\draw[->,red] (0) edge node[above] {$b_{1}$}  (1);
\draw[->,blue] (2) edge node[sloped, anchor=center, above] {$a_{2}$} (1);
\draw[->,red] (3) edge node[above] {$b_{2}$} (2);
\draw[->,blue,bend left, out=40,in=140] (0) edge node[sloped, anchor=center, below] {$a_{1\!1\!1}$} (3);
\draw[->,blue,bend left] (3) edge node[sloped, anchor=center, above] {$a_{1\!2\!1\!1}$} (0);

\draw[->,darkgray!30,bend left, in =110,out=65,max distance= 11 mm] (0) edge[dashed] node[above = 0.3] {\color{gray}\scriptsize{$\walki_1$}} (3);

\draw[->,darkgray!40,bend left, in = 130,out=75,max distance= 8 mm] (3) edge[dashed]  (00);

\draw[->,darkgray!50,bend left, in = 120,out=50,max distance= 9 mm] (00) edge[dashed] (5);

\draw[->,darkgray!60,bend left, in = 120,out=65,max distance= 9 mm] (5) edge[dashed] (6);

\draw[->,darkgray!70,bend left, in = 150,out=65,max distance= 9 mm] (6) edge[dashed] (000);

\draw[->,darkgray!80,bend left, in =110,out=65,max distance= 12 mm] (000) edge[dashed] (33);

\draw[->,darkgray,bend left, in = 80,out=80,max distance= 11 mm] (33) edge[dashed]  (0);
\end{tikzpicture}
\end{figure}
\vspace{-15pt}
Finalment, si designem per $a_2$, $b_1$, i $b_2$ els arcs de $\Ati_{\hspace{-1pt} 1}$ encara no usats, i mantenim el conveni de notació per als arcs corresponents de $\Ati_{\hspace{-1pt} 0}={\mathcal F}$, podem elevar $\walki_1$ a
 \begin{equation*}
\walki_0=a_{111}b_{21}a_{21}b_{11}^{-1}b_{12}a_{22}^{-1}b_{22}^{-1}  a_{1211}a_{1212}^{-1}a_{122}^{-1}a_{112}^{-1}a_{111}b_{21}a_{21}b_{11}^{-1}b_{12}a_{22}^{-1}b_{22}^{-1}a_{1211}.
 \end{equation*}
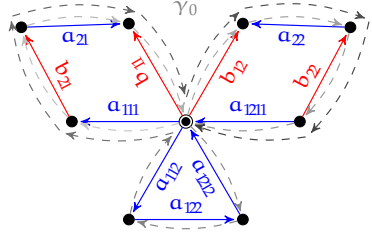
\vspace{-15pt}
\begin{figure}[H]
\centering
\begin{tikzpicture}[shorten >=1pt, node distance=1cm and 1cm, on grid,auto,auto,>=stealth',rotate=90,transform shape]       
\newcommand{\rad}{1.5}
\newcommand{\Rad}{2.5}
\node[state,accepting] (bp)  {};

\foreach \x in {0,...,5}
{
\node[state] (\x) at ($(90+\x*360/6:\rad cm)$) {};
}

\node[state] (34) at ($(180+120:\Rad cm)$) {};
\node[state] (50) at ($(180+2400:\Rad cm)$) {};
\node[inner sep=0pt] (bp0) [above left = 0.2 of bp] {};
\node[inner sep=0pt] (00) [above left = 0.2 of 0] {};
\node[inner sep=0pt] (500) [above right = 0.2 of 50] {};
\node[inner sep=0pt] (55) [below right = 0.2 of 5] {};
\node[inner sep=0pt] (bp1) [right = 0.3 of bp] {};
\node[inner sep=0pt] (44) [right = 0.2 of 4] {};
\node[inner sep=0pt] (344) [below right= 0.25 and 0.1 of 34] {};
\node[inner sep=0pt] (33) [below left = 0.1 and 0.2 of 3] {};
\node[] (g0) [right = 1.5 of bp] {\rotatebox{-90}{\color{gray}{\scriptsize{$\walki_0$}}}};

\draw[->,blue] (bp) edge node[sloped, below, pos=0.5] {$a_{1\!1\!2}$} (1);
\draw[->,blue] (1) edge node[sloped, above, pos=0.5] {$a_{1\!2\!2}$} (2);
\draw[->,blue] (2) edge node[sloped, above, pos=0.5] {\rotatebox{180}{$a_{1\!2\!1\!2}$}} (bp);

\draw[->,blue] (bp) edge node[sloped, anchor=center, below,pos=0.55] {\rotatebox{180}{$a_{1\!1\!1}$}} (0);
\draw[->,red] (0) edge node[sloped, below, pos=0.4] {\rotatebox{180}{$b_{2\!1}$}} (50);
\draw[->,blue] (50) edge[] node[sloped, anchor=center, below] {$a_{2\!1}$} (5);
\draw[->,red] (bp) edge  node[sloped, above, pos=0.6] {$b_{1\!1}$} (5);

\draw[->,red] (bp) edge node[sloped, below, pos=0.6] {$b_{1\!2}$} (4);
\draw[->,blue] (34) edge node[sloped, above, pos=0.55] {\rotatebox{180}{$a_{2\!2}$}} (4);
\draw[->,red] (3) edge node[sloped, above, pos=0.4] {$b_{2\!2}$} (34);
\draw[->,blue] (3) edge node[sloped, below, pos=0.45] {\rotatebox{180}{$a_{1\!2\!1\!1}$}} (bp);

\draw[->,darkgray!30,bend left=15,max distance= 11 mm] (bp) edge[dashed] (0);
\draw[->,darkgray!34,bend left=15,max distance= 11 mm] (0) edge[dashed] (50);
\draw[->,darkgray!38,bend left=15,max distance= 11 mm] (50) edge[dashed] (5);
\draw[->,darkgray!42,bend left=15,max distance= 11 mm] (5) edge[dashed] (bp);
\draw[->,darkgray!46,bend left=15,max distance= 11 mm] (bp) edge[dashed] (4);
\draw[->,darkgray!50,bend left=15,max distance= 11 mm] (4) edge[dashed] (34);
\draw[->,darkgray!54,bend left=15,max distance= 11 mm] (34) edge[dashed] (3);
\draw[->,darkgray!58,bend left=15,max distance= 11 mm] (3) edge[dashed] (bp);
\draw[->,darkgray!62,bend left=15,max distance= 11 mm] (bp) edge[dashed] (2);
\draw[->,darkgray!66,bend left=15,max distance= 11 mm] (2) edge[dashed] (1);
\draw[->,darkgray!70,bend left=15,max distance= 11 mm] (1) edge[dashed] (bp0);
\draw[->,darkgray!74,bend left=15,max distance= 11 mm] (bp0) edge[dashed] (00);
\draw[->,darkgray!78,bend left=15,max distance= 11 mm] (00) edge[dashed] (500);
\draw[->,darkgray!82,bend left=15,max distance= 11 mm] (500) edge[dashed] (55);
\draw[->,darkgray!86,bend left=15,max distance= 11 mm] (55) edge[dashed] (bp1);
\draw[->,darkgray!90,bend left=15,max distance= 11 mm] (bp1) edge[dashed] (44);
\draw[->,darkgray!94,bend left=15,max distance= 11 mm] (44) edge[dashed] (344);
\draw[->,darkgray!98,bend left=15,max distance= 11 mm] (344) edge[dashed] (33);
\draw[->,darkgray,bend left=15,max distance= 11 mm] (33) edge[dashed] (bp);
\end{tikzpicture}
\end{figure}
\vspace{-5pt}
Assenyalant amb parèntesis les visites completes a cada pètal o pètal invers (marcades per les visites de $\walki_0$ a $\bp$) obtenim l'expressió desitjada de $u=a$ en termes de $v_1, v_2, v_3$:
 \begin{equation*}
a=\big( abab^{-1}\big) \big( ba^{-1}b^{-1}a\big) \big( a^{-1}a^{-1}a^{-1} \big) \big( abab^{-1} \big) \big( ba^{-1}b^{-1}a\big)
=v_2 v_3^{-1} v_1^{-1} v_2 v_3^{-1}.
 \end{equation*}
Com hem comentat més amunt, les elevacions a través de plegaments oberts són úniques, però l'elevació de $\walki_5$ a través del plegament tancat $\Ati_{\hspace{-1pt} 4} \xtr{\,} \Ati_{\hspace{-1pt} 5}$ té diverses possibilitats (que acabaran en resultats diferents, \ie en expressions alternatives de $u=a$ en termes de $v_1^{\pm 1}, v_2^{\pm 1}, v_3^{\pm 1}$). Convidem el lector a prendre $\walki_4=a_{12}$ enlloc de $\walki_4=a_{11}$, i a calcular la seva elevació fins a $\Ati_{\hspace{-1pt}0}$ per a obtenir la nova expressió $a=(a^{-1}bab^{-1}) (ba^{-1}b^{-1}a^{-1})(aaa) =v_3v_2^{-1}v_1$. 

El fet de poder obtenir $a$ en termes de $v_1^{\pm 1}, v_2^{\pm 1}, v_3^{\pm 1}$ de \emph{dues maneres diferents}, $v_2 v_3^{-1} v_1^{-1} v_2 v_3^{-1} =a= v_3v_2^{-1}v_1$, ens confirma el fet que $\{ v_1, v_2, v_3\}$ generen $H$, però \emph{no són un conjunt lliure} (vegeu la \Cref{def: base}). Dit d'una altra manera, $v_2 v_3^{-1} v_1^{-1} v_2 v_3^{-1} v_1^{-1}v_2v_3^{-1}=1$ és una relació no-trivial entre els tres generadors $v_1, v_2, v_3$ de $H$.

\subsection{Generadors, bases i rang} \label{ssec: bases and rk}
Comencem recordant que la \Cref{prop: S_T} garanteix la computabilitat d'una base per qualsevol subgrup de $\Free[A]$ a partir d'un conjunt \emph{finit} de generadors. 

\begin{cor}
Si $S$ és un conjunt finit d'elements d'un grup lliure $\Free[A]$, aleshores una base (i, per tant, el rang) de $\gen{S}\leqslant \Free[A]$ és computable. Més concretament, el rang de $\gen{S}$ és ${\rk(\gen{S})= 1-\card\Verts\Ati+\card\Edgs^+\Ati}$, on ${\Ati =\stallings(\gen{S})}$. 
\end{cor}

Molts altres resultats fonamentals sobre bases de (subgrups de) grups lliures es deriven fàcilment de la construcció de Stallings. L'observació següent ens portarà a la important propietat de Hopfianitat dels grups lliures. 


Sigui $\Ati$ un $A$-autòmat involutiu, finit, i connex, i sigui $\Ati \xtr{\phi\,} \Ati'$ un plegament elemental. El corresponent morfisme d'autòmats $\phi\colon \Ati \to \Ati'$ indueix un morfisme exhaustiu de grups  $\phi\colon \pi_{\bp}(\Ati)\onto \pi_{\bp}(\Ati')$, $[\walki]\mapsto [\walki\phi]$ tal que $\phi \widetilde{\ell}_{\Ati'} =\widetilde{\ell}_{\Ati}$.
A més, és fàcil veure que si el plegament és obert llavors $\phi$ és bijectiu, mentre que si és tancat (posem per cas, amb $\edgi_1, \edgi_2 \in \Edgs\Ati$ els dos arcs que s'identifiquen a $\Ati'$, $\verti=\iota\edgi_1=\iota \edgi_2$, $\vertii=\tau \edgi_1 =\tau \edgi_2$, i $\lab(\edgi_1)=\lab(\edgi_2)\in A^{\pm}$) llavors $\ker \phi$ està generat, com a subgrup normal de $\pi_{\bp}(\Ati)$, per l'element $[\eta \verti \edgi_1 \vertii \edgi_2^{-1} \verti \eta^{-1}]$, on $\eta$ és un camí qualsevol a $\Ati$ de $\bp$ a $\verti$. 

L'enunciat següent mostra com la teoria desenvolupada captura geomètricament els dos ingredients (algebraics) de la definició elemental de base (\Cref{def: base}).

\begin{prop}\label{prop: ind gen iff}
Sigui $S\subseteq \Free[A]$. Aleshores,
 \begin{enumerate} [ind]
\item \label{item: gen iff} $S$ genera $\Free[A]$ si i només si $\stallings(\gen{S})=\bouquet_{A}$;
\item \label{item: ind iff} $S$ és una família lliure de $\Free[A]$ si i només si $\loss (\flower(S))=0$; i, en aquest cas, ${\rk (\gen{S}) = \card S}$.
\end{enumerate}
Ambdues condicions són algorísmicament decidibles si $S$ és finit.
\end{prop}

\begin{proof}
L'apartat \ref{item: gen iff} és clar ja que $\stallings(\Free[A]) = \bouquet_{A}$. 

Per veure l'apartat \ref{item: ind iff} distingim dos casos. Si $S$ és finit, considerem l'autòmat flor $\flower(S)$, i una seqüència qualsevol de plegaments elementals fins a obtenir $\stallings(\gen{S})$, 
 \begin{equation*}
\flower(S) = \Ati_{\hspace{-1pt}0} \xtr{\! \phi_{_1} \!} \Ati_{\hspace{-1pt}1} \xtr{\!\phi_{_2} \!} \cdots \xtr{\! \phi_{_{p-1}} \!} \Ati_{\hspace{-1pt} p-1} \xtr{\! \phi_{_p}\!} \Ati_{\hspace{-1pt}p} =\stallings(\gen{S}).
 \end{equation*}
Component els corresponents morfismes obtenim el diagrama commutatiu 

\begin{equation}
    \begin{tikzcd}[
        column sep={6em,between origins},
      ]
        \pi_{\bp}(\Ati_{\hspace{-1pt}0}) \arrow[r,"\phi_1",->>] \arrow[rrd,"\widetilde{\lab}_{0}"',->>] & \pi_{\bp}(\Ati_{\hspace{-1pt}1}) \arrow[r,"\phi_2",->>] \arrow[rd,"\widetilde{\lab}_{1}",->>] & {\cdots} \arrow[r,"\phi_{p-1}",->>] & \pi_{\bp}(\Ati_{\hspace{-1pt}p-1}) \arrow[r,"\phi_p",->>] \arrow[ld,"\widetilde{\lab}_{p-1}"',->>] & \pi_{\bp}(\Ati_{\hspace{-1pt}p}) \arrow[lld,"\widetilde{\lab}_{p}",->>] \\
                                  &                          &               \gen{S}          
        \end{tikzcd}
    \end{equation}
on $\rk(\pi_{\bp}(\Ati_0))=\card S$, amb (les classes de) els pètals com a base; i on $\widetilde{\ell}_p$ és bijectiu ja que $\Ati_{\hspace{-1pt}p}$ és determinista. Si $\loss (\flower(S))=0$, tots els plegaments són oberts, cadascun dels morfismes $\phi_1, \ldots \phi_p$ és bijectiu i, concloem que $\widetilde{\ell}_0=\phi_1\cdots \phi_p \widetilde{\ell}_p$ és també bijectiu; per tant, $\gen{S}$ està lliurement generat per les etiquetes dels pètals de $\flower(S)$, és a dir, per $S$. Recíprocament, si $S$ és una família lliure d'elements de $\Free[A]$, llavors $S$ és base de $\gen{S}\leqslant \Free[A]$, i el morfisme $\widetilde{\ell}_0$ és bijectiu; la commutativitat del diagrama anterior implica aleshores que $\phi_1, \ldots ,\phi_p$ són tots bijectius i, per tant, tots els plegaments de la seqüència són oberts; en conseqüència, $\loss(\flower(S))=0$.

Finalment, suposem que $S$ és infinit. Si $S$ no és una família lliure, llavors existeix un subconjunt finit $S_0\subset S$ que tampoc és lliure i, aplicant el cas finit, deduïm que $\loss(\flower(S))\geq \loss(\flower(S_0))\geq 1$. Recíprocament, si $S$ és una família lliure, tota subfamília finita $S_0\subseteq S$ també és lliure i, aplicant el cas finit, $\loss(\flower(S_0))=0$ ; per tant, cap seqüència finita de plegaments elementals començant a $\flower(S)$ (que involucrarà forçosament un nombre finit de pètals) pot contenir un plegament tancat; és a dir, ${\loss(\flower(S))=0}$.
\end{proof}

Noteu que d'aquest resultat se'n segueix una prova alternativa per a la \Cref{prop: isom iff =card}: si $B$ és una base de $\Free[A]$, aleshores $\loss (\flower(B))=0$ i $\stallings(\gen{B},A)=\bouquet_{A}$; per tant, $\card B=\rk (\flower(B))=\rk (\stallings(\gen{B},A))=\rk (\bouquet_{A})=\card A$. També en podem deduir la important propietat de Hopfianitat dels grups lliures finitament generats:

\begin{thm} \label{thm: basis iff min gs}
El grup lliure $\Fn$ és \emph{Hopfià}: tot endomorfisme $\varphi\colon \Fn \to \Fn$ exhaustiu és automàticament injectiu. (Equivalentment, $S\subseteq \Fn$ és base de $\Fn$ si i només si $\gen{S}=\Fn$ i $\card S=n$.) 
\end{thm}

\begin{proof}
De nou, la demostració és immediata de la Teoria de Stallings: si $\gen{S}=\Fn$ i $\card S=n$, aleshores $\stallings(\gen{S})=\bouquet_n$, i $\rk (\flower(S))=n$. Per tant, 
 \begin{equation}
\loss(\flower(S))=\rk (\flower(S))-\rank (\stallings(\gen{S})) =n-n=0
 \end{equation}
i, per la \Cref{prop: ind gen iff}\ref{item: ind iff}, tenim que $S$ és una família lliure; per tant, $S$ és base de~$\Fn$.
\end{proof}

D'acord amb l'\Cref{rem: base => gen minimal}, i tal i com succeeix als $K$-espais vectorials, les bases de grups lliures $\Free$ són sempre conjunts minimals de generadors. El recíproc --- també cert als $K$-espais vectorials --- és fals als grups lliures, fins a la màxima degeneració possible.

\begin{lem}
Existeixen conjunts de generadors minimals de $\Free[n]$, $n\geq 1$ de qualsevol cardinalitat finita $m\geq n$.
\end{lem}

\begin{proof}
Aquest comportament ja es pot observar a $\ZZ=\Free[1]$ per raons purament aritmètiques: considerem un conjunt de nombre primers $\{p_1,\ldots ,p_m\}$ diferents dos a dos, i prenem $S=\{ \prod_{i\neq j} p_i \st j=1,\ldots ,m\}$; clarament $\operatorname{mcd}(S)=1$ i, per tant, $\gen{S}=\ZZ$. Però $\operatorname{mcd}(S\setminus \{\prod_{i\neq j} p_i\})=p_j$ i, per tant, $\gen{S\setminus \{\prod_{i\neq j} p_i\}}=p_j\ZZ \neq \ZZ$. 

Usant la Teoria de Stallings podem construir exemples més sofisticats. Per exemple, al grup $\Free_{\{a,b\}}$ considereu la família d'autòmats involutius $\Ati_{\hspace{-1pt} k}$, $k\geq 2$, representada a continuació:
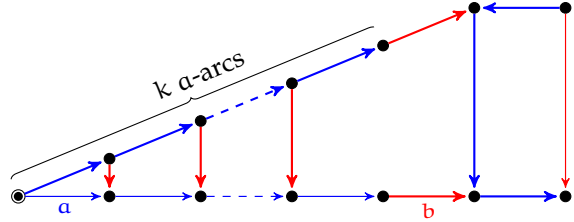
\begin{figure}[H] 
\centering
  \begin{tikzpicture}[shorten >=1pt, node distance=0.5 and 1.2, on grid,auto,>=stealth']
   \node[state,accepting] (0) {};
   \node[state] (1) [right = of 0]{};
   \node[state] (2) [right = of 1]{};
   \node[state] (3) [right = of 2]{};
   \node[state] (4) [right = of 3]{};
   \node[state] (5) [right = of 4]{};
   \node[state] (6) [right = of 5]{};
\node[state] (1') [above right = of 0]{};
   \node[state] (2') [above right = of 1']{};
   \node[state] (3') [above right = of 2']{};
   \node[state] (4') [above right = of 3']{};
   \node[state] (5') [above right = of 4']{};
   \node[state] (6') [right = of 5']{};'

   \path[->]
        (0) edge[blue]
            node[below] {\scriptsize{$a$}}
            (1);

    \path[->] (1) edge[blue] (2);            
    \path[->] (2) edge[blue,dashed] (3);
    \path[->] (3) edge[blue] (4);
    \path[->] (4) edge[red,thick] node[below] {\scriptsize{$b$}} (5);
    \path[->] (5) edge[blue,thick] (6);
    
    \path[->] (0) edge[blue,thick] (1');
    \path[->] (1') edge[blue,thick] (2');
    \path[->] (2') edge[blue,dashed,thick] (3');
    \path[->] (3') edge[blue,thick] (4');
    \path[->] (4') edge[red,thick] (5');
    \path[->] (6') edge[blue,thick] (5');
    
    \path[->] (1') edge[red,thick] (1);
    \path[->] (2') edge[red,thick] (2);
    \path[->] (3') edge[red,thick] (3);
    \path[->] (5') edge[blue,thick] (5);
    \path[->] (6') edge[red] (6);

    \draw [decorate, 
    decoration = {brace,raise=7pt}] (0,0) --  (4.8,2) node[pos=0.75,above left = 13pt,sloped]{\small{$k$ $a$-arcs}};
\end{tikzpicture}
\end{figure}

Prenent com a arbre d'expansió el donat pels arcs representats en traç gruixut obtenim el conjunt de generadors
 \begin{equation*} \label{eq: minimal gen}
\begin{aligned}
S_k &\,=\, \set{a b a^{-1},\, a^k b a b^{-1} a^{-1} b^{-1} a^{k-1},\, a^k b a^{-1} b a^{-2} b^{-1} a^{-k}} \,\cup\\ &\hspace{150pt}\cup \, \set{a^i b a^{-1} b^{-1} a^{-(i-1)} \st i\in[2,k-1]}
\end{aligned}
 \end{equation*}
(de cardinal $k+1$) per a $\gen{\Ati_{\hspace{-1pt} k}}$. Observem però, que $\Ati_{\hspace{-1pt} k }$ no és determinista (ho seria sense un dels dos $a$-arcs sortint de $\bp$) i que, fent-li plegaments de Stallings co\l.lapsa fins a la rosa $\bouquet_{\set{a,b}}$. Per tant, $\gen{S_k}= \gen{\Ati_{\hspace{-1pt}k}}=\Free[\set{a,b}]$ i, per a tot $k\geq 2$, $S_k$ és una família de $k+1$ generadors de $\Free_{\{a,b\}}$. Ara bé, si eliminem un qualsevol dels generadors de $S_k$ (equivalentment, si eliminem de $\Ati_{\hspace{-1pt} k }$ l'arc corresponent fora de l'arbre d'expansió), el procés de plegments queda bloquejat justament a l'arc eliminat i s'obté un $A$-autòmat determinista diferent a $\bouquet_{\set{a,b}}$. Per tant $S_k$ és un sistema de generadors \emph{minimal} de $\Free[\set{a,b}]$, de cardinal $k+1$. 
\end{proof}

Fixeu-vos finalment, que la família $S_k$ també serveix com a contraexemple al grup lliure de la propietat fonamental en àlgebra lineal de que tota família de generadors conté una base.


\subsection{Conjugació i normalitat} \label{ssec: conj and normal}
Un dels conceptes fonamentals en teoria de grups és el de \defin{conjugació}: si $G$ és un grup, es diu que dos elements $g,g' \in G$ (\resp subgrups $H,H'\leqslant G$) són conjugats, i escrivim $g \sim g'$ (\resp $H \sim H'$) si existeix un element $z\in G$ tal que $z^{-1}gz=g'$ (\resp $z^{-1}Hz=H'$). Aleshores diem que $g'$ (\resp $H'$) és $g$ (\resp $H$) \defin{conjugat per}~$z$, i escrivim $g'=g^z =z^{-1}gz$ (\resp $H'=H^z =z^{-1}Hz$).

La teoria de Stallings permet caracteritzar la conjugació de subgrups de forma molt transparent.

\begin{lem} \label{lem: conj iff basepoint}
Siguin $H$ un subgrup de $\Free[A]$, i $w\in \Free[A]$. Aleshores, $\stallings(H^w,A)$ és el cor de l'autòmat $\schreier_{H\!w}(H,A)$ obtingut canviant el punt base de $\schreier(H,A)$ per $Hw$.
\end{lem}

\begin{proof}
És clar de la \Cref{prop: propietats Sch} i la \Cref{def: subgrup reconegut} que $\schreier_{H\!w}(H,A)$ és un $A$-autòmat involutiu, determinista, saturat i connex que reconeix el subgrup $H^w$. El resultat se segueix immediatament del \Cref{cor: SchSt}(ii).
\end{proof}

Com a conseqüència del lema anterior obtenim la següent caracterització gràfica de la conjugació de subgrups al grup lliure, que és òbviament decidible en cas que els subgrups siguin finitament generats.

\begin{prop}\label{prop: conj}
Siguin $H$ i $K$ dos subgrups de $\Free[A]$. Aleshores, $H$ i $K$ són conjugats si i nomes si $\stallings^*(H)=\stallings^*(K)$. 
\end{prop}

\begin{proof}
Observeu que, per construcció, $\stallings(H)$ coincideix amb $\stallings^*(H)$ reintroduint el pèl (eventualment buit) fins a $\bp_H$. 

Suposem que $\stallings^*(H)=\stallings^*(K)$. Reintroduïm a $\stallings^*(H)=\stallings^*(K)$ els dos pèls fins als vèrtexs $\bp_H$ i $\bp_K$ i sigui $w$ l'etiqueta d'un camí reduït qualsevol de $\bp_H$ a $\bp_K$. És fàcil veure que $H^w=K$; per tant, $H$ i $K$ són conjugats.

Recíprocament, suposem que $H^w=K$ per alguna paraula (reduïda) $w\in \Free[A]$. Sigui $\gamma$ l'únic possible camí a $\stallings(H)$ començant a $\bp_H$ i llegint $w$: la unicitat és clara pel determinisme de $\stallings(H)$; i si no n'existeix cap, afegim a $\stallings(H)$ el pèl necessari per poder completar $\gamma$. Declarem $\tau \gamma$ com a nou punt base, i designem per $\Ati$ l'autòmat obtingut. Per construcció, $\core^*(\Ati)=\core^*(\stallings(H))=\stallings^*(H)$ és determinista, i $\gen{\Ati}=w^{-1}\gen{\stallings(H)}w=H^w=K$. Per tant, $\core^*(\Ati)=\core^*(\stallings(K))=\stallings^*(K)$, d'on deduïm $\stallings^*(H)=\stallings^*(K)$.
\end{proof}

Observeu que, en cas que $H$ i $K$ siguin finitament generats i conjugats, la demostració anterior proporciona, a més, un mètode per a calcular efectivament un conjugador $w$. Això demostra que

\begin{cor}
El problema de la conjugació de subgrups $\SCP(\Free)$ és decidible per a grups lliures: existeix un algoritme que, donats dos subconjunts finits $S_1,S_2 \subseteq \Free$, decideix si els corresponents subgrups generats són conjugats i, en cas afirmatiu, retorna un element conjugador. \qed
\end{cor}

Recordem que un subgrup $H$ d'un grup $G$ es diu \defin{normal} (a $G$), designat ${H \normaleq G}$, si i només si $H^g=H$, per a tot $g\in G$. Del \Cref{lem: conj iff basepoint} també se segueix fàcilment una caracterització gràfica de normalitat en el grup lliure (i la seva decidibilitat en el cas finitament generat).

\begin{prop} \label{prop: normal iff}
Sigui $H\neq \trivial$ un subgrup de $\Free[A]$. Aleshores, els enunciats següents són equivalents:
\begin{enumerate}[dep]
\item H és normal a $\Free[A]$;
\item $\schreier(H)$ és vèrtex-transitiu;\footnote{Un $A$-autòmat $\Ati$ és \defin{vèrtex-transitiu} si per a qualsevol parell de vèrtexs, $\verti,\vertii$ de $\Ati$, existeix un automorfisme \emph{de $A$-digrafs etiquetats} (és a dir, \emph{respectant les etiquetes de les arestes} i \emph{ignorant el punt base}) de $\Ati$  que envia $\verti$ a $\vertii$. (Informalment, si `$\Ati$ té el mateix aspecte vist des de qualsevol vèrtex'.)}
\item $\schreier(H)$ és vèrtex-transitiu i cor;
\item $\stallings(H)$ és vèrtex-transitiu i saturat.
\end{enumerate}
\end{prop}

\begin{proof}
L'equivalència entre  (a) i (b) se segueix clarament de les definicions a través del \Cref{lem: conj iff basepoint}. Concretament, per a tot $w \in \Free[A]$:
\begin{equation*}
    H^w = H \ \Leftrightarrow\  \schreier(H^w) = \schreier(H) \ \Leftrightarrow\  \schreier_{Hw}(H) = \schreier(H).
\end{equation*}
La interpretació de l'esquerra correspon a la normalitat de $H$, mentre que la de la dreta correspon a la vèrtex-transitivitat de $\schreier(H)$.

Per veure l'equivalència entre (b) i (c) és suficient veure que (b) implica~(c). En efecte, si~$\schreier(H)$ és vèrtex-transitiu aleshores ha de ser cor ja que, clarament, no existeix cap automorfisme d'$A$-autòmats $\Ati$ enviant un vèrtex del cor de $\Ati$ (i, per tant, pertanyent a un $\bp$-camí reduït de $\Ati$) a un vèrtex fora del cor de $\Ati$ (i, per tant, no pertanyent a cap $\bp$-camí reduit de $\Ati$). 

Finalment, l'equivalència entre (c) i (d) és conseqüència immediata de la \Cref{prop: propietats St}\ref{item: St sat iff Sch cor}.
\end{proof}

\begin{rem}
Tot i que per decidir si un subgrup finitament generat $H\leqslant \Free[A]$ és normal, n'hi ha prou amb comprovar si els conjugats dels generadors per les lletres de $A^{\pm}$ tornen a pertànyer a $H$ (un nombre finit d'instàncies del problema de la pertinença), la proposició anterior ens dóna una caracterització gràfica molt natural per a la normalitat, que també es pot usar per donar un algorisme alternatiu per decidir la normalitat de $H$. 
\end{rem}

\subsection{El problema de l'índex} \label{ssec: FIP}

Recordem que si $H$ és un subgrup d'un grup $G$, aleshores $G$ sempre admet una partició $G = \bigsqcup_{i\in I} H g_i$, on $g_i \in G$. Els conjunts $H g_i$ s'anomenen \defin{classes laterals (per la dreta)} de $H$ a $G$; el cardinal del conjunt de classes laterals $H\backslash G=\set{H g_i\st i\in I}$ s'anomena \defin{índex} de $H$ en $G$, i es designa per $\ind{H}{G}$; es diu també que $\set{g_i}_{i\in I}$ és una \defin{família de representants} (o \defin{transversal}) de les classes laterals (per la dreta) de $H$ en~$G$.

En aquesta secció usarem els autòmats de Stallings per estudiar qüestions relatives a l'índex dels subgrups dels grups lliures, i resoldrem el problema de la finitud de l'índex, que  enunciem en general a continuació.

\begin{named}[Problema de la finitud de l'índex a $G=\pres{A}{R}$, $\FIP(G)$]
Donada una família finita $v_1,\ldots, v_k$ de paraules en els generadors de $G$, decidir si el subgrup ${\gen{v_1,\ldots, v_k}\leqslant G}$ és d'índex finit a $G$ i, en cas afirmatiu, calcular aquest índex i una família de representants de les corresponents classes laterals.
\end{named}

Recordem (\Cref{def: Sch}) que els vèrtexs de l'autòmat de Schreier $\schreier(H)$ són precisament les classes laterals per la dreta de $H$ a $\Free$, és a dir $\ind{H}{\Free} = \card\Verts\schreier(H)$; i observem que
l'apartat \ref{item: fi iff St} de la \Cref{prop: propietats St} pren la forma següent per a subgrups finitament generats.

\begin{cor} \label{cor: fi iff sat}
Sigui $H$ un subgrup finitament generat d'un grup lliure $\Free$. Aleshores, l'índex $\ind{H}{\Free}$ és finit si i només si $\stallings(H)$ és saturat; en aquest cas, $\ind{H}{\Free}=\card\Verts \stallings(H)$. \qed
\end{cor}


Com a primera conseqüència podem deduir que $\Free[A]$ té una quantitat finita de subgrups d'un índex finit $k$ donat (vegeu \cite{hall_jr_subgroups_1949} per una fórmula recurrent que compta aquest nombre de subgrups). D'aquí es dedueix fàcilment el mateix resultat per a qualsevol grup $G$ finitament generat.

\begin{cor} 
Sigui $G$ un grup finitament generat. Per a tot $k\geq 1$, $G$ conté una quantitat finita de subgrups d'índex $k$.\qed
\end{cor}

\begin{exm}
    A continuació representem els autòmats de Stallings dels $13$ subgrups d'índex $3$ de $\Free[2] = \pres{a,b}{-}$ (amb les corresponents classes de conjugació encerclades en gris). Òbviament, els subgrups normals corresponen a les classes amb un únic representant.
\begin{figure}[H]
    \centering
\begin{tikzpicture}[shorten >=1pt, node distance=1.2cm and 1.75cm, on grid,auto,auto,>=stealth']

\draw [ultra thick, rounded corners, draw=black, opacity=0.1]
    (-0.5,0.5) -- (-0.5,-1.5) -- (1.5,-1.5) -- (1.5,0.5) -- cycle;

\draw [ultra thick, rounded corners, draw=black, opacity=0.1]
    (1.68,0.5) -- (1.68,-1.5) -- (3.66,-1.5) -- (3.66,0.5) -- cycle;

\draw [ultra thick, rounded corners, draw=black, opacity=0.1]
    (3.85,0.5) -- (3.85,-1.5) -- (5.5,-1.5) -- (5.5,0.5) -- cycle;

\draw [ultra thick, rounded corners, draw=black, opacity=0.1]
    (5.7,0.5) -- (5.7,-1.5) -- (7.25,-1.5) -- (7.25,0.5) -- cycle;

\draw [ultra thick, rounded corners, draw=black, opacity=0.1]
       (7.45,0.5) -- (7.45,-1.5) -- (12.65,-1.5) -- (12.65,0.5) -- cycle;
       
\draw [ultra thick, rounded corners, draw=black, opacity=0.1]
       (0.5,-2) -- (0.5,-3.65) -- (5.75,-3.65) -- (5.75,-2) -- cycle;
       
\draw [ultra thick, rounded corners, draw=black, opacity=0.1]
       (6,-2) -- (6,-3.65) -- (11.3,-3.65) -- (11.3,-2) -- cycle;

\begin{scope}
    \node[state,accepting] (1) {};
    \node[state] (2) at (1,0) {};
    \node[state] (3) at (.5,{-sin(60)}) {};
  
   \path[->] (1) edge[blue,bend left=15] (2);
   \path[->] (2) edge[blue,bend left=15] (3);
   \path[->] (3) edge[blue,bend left=15] (1);
   \path[->] (1) edge[red,loop above,min distance=6mm,in=110,out=160] (1);
   \path[->] (2) edge[red,loop above,min distance=6mm,in=20,out=70] (2);
   \path[->] (3) edge[red,loop below,min distance=6mm,in=245,out=295] (3);
\end{scope}
   
  \begin{scope}[shift={(2.2,0)}]
   \node[state,accepting] (1) {};
   \node[state] (2) at (1,0) {};
   \node[state] (3) at (.5,{-sin(60)}) {};

   \path[->] (1) edge[red,bend left=15] (2);
   \path[->] (2) edge[red,bend left=15] (3);
   \path[->] (3) edge[red,bend left=15] (1);
   \path[->] (1) edge[blue,loop above,min distance=6mm,in=110,out=160] (1);
   \path[->] (2) edge[blue,loop above,min distance=6mm,in=20,out=70] (2);
   \path[->] (3) edge[blue,loop below,min distance=6mm,in=245,out=295] (3);
  \end{scope}
  
  \begin{scope}[shift={(4.2,0)}]
   \node[state,accepting] (1) {};
    \node[state] (2) at (1,0) {};
    \node[state] (3) at (.5,{-sin(60)}) {};
  
   \path[->] (1) edge[red,bend right=15] (2);
   \path[->] (2) edge[red,bend right=15] (3);
   \path[->] (3) edge[red,bend right=15] (1);
   \path[->] (1) edge[blue,bend left=15] (2);
   \path[->] (2) edge[blue,bend left=15] (3);
   \path[->] (3) edge[blue,bend left=15] (1);
  \end{scope}
  
  \begin{scope}[shift={(6,0)}]
   \node[state,accepting] (1) {};
    \node[state] (2) at (1,0) {};
    \node[state] (3) at (.5,{-sin(60)}) {};
  
   \path[->] (2) edge[red,bend left=15] (1);
   \path[->] (3) edge[red,bend left=15] (2);
   \path[->] (1) edge[red,bend left=15] (3);
   \path[->] (1) edge[blue,bend left=15] (2);
   \path[->] (2) edge[blue,bend left=15] (3);
   \path[->] (3) edge[blue,bend left=15] (1);
  \end{scope}
  
  \begin{scope}[shift={(8,0)}]
    \node[state,accepting] (1) {};
    \node[state] (2) at (1,0) {};
    \node[state] (3) at (.5,{-sin(60)}) {};
  
   \path[->] (1) edge[blue,bend left=15] (2);
   \path[->] (2) edge[blue,bend left=15] (1);
   \path[->] (2) edge[red,bend left=15] (3);
   \path[->] (3) edge[red,bend left=15] (2);
   \path[->] (1) edge[red,loop above,min distance=6mm,in=110,out=160] (1);
   \path[->] (3) edge[blue,loop below,min distance=6mm,in=245,out=295] (3);
  \end{scope}
  
  \begin{scope}[shift={(9.7,0)}]
    \node[state] (1) {};
    \node[state,accepting] (2) at (1,0) {};
    \node[state] (3) at (.5,{-sin(60)}) {};
  
   \path[->] (1) edge[blue,bend left=15] (2);
   \path[->] (2) edge[blue,bend left=15] (1);
   \path[->] (2) edge[red,bend left=15] (3);
   \path[->] (3) edge[red,bend left=15] (2);
   \path[->] (1) edge[red,loop above,min distance=6mm,in=110,out=160] (1);
   \path[->] (3) edge[blue,loop below,min distance=6mm,in=245,out=295] (3);
  \end{scope}
  
  \begin{scope}[shift={(11.4,0)}]
    \node[state] (1) {};
    \node[state] (2) at (1,0) {};
    \node[state,accepting] (3) at (.5,{-sin(60)}) {};
  
   \path[->] (1) edge[blue,bend left=15] (2);
   \path[->] (2) edge[blue,bend left=15] (1);
   \path[->] (2) edge[red,bend left=15] (3);
   \path[->] (3) edge[red,bend left=15] (2);
   \path[->] (1) edge[red,loop above,min distance=6mm,in=110,out=160] (1);
   \path[->] (3) edge[blue,loop below,min distance=6mm,in=245,out=295] (3);
  \end{scope}
  
  \begin{scope}[shift={(1,-2.5)}]
    \node[state,accepting] (1) {};
    \node[state] (2) at (1,0) {};
    \node[state] (3) at (.5,{-sin(60)}) {};
  
   \path[->] (1) edge[blue,bend left=15] (2);
   \path[->] (2) edge[blue,bend right=25] (3);
   \path[->] (2) edge[red] (3);
   \path[->] (3) edge[red,bend right=25] (2);
   \path[->] (3) edge[blue,bend left=15] (1);
   \path[->] (1) edge[red,loop above,min distance=6mm,in=110,out=160] (1);
  \end{scope}
  
  \begin{scope}[shift={(2.75,-2.5)}]
    \node[state] (1) {};
    \node[state,accepting] (2) at (1,0) {};
    \node[state] (3) at (.5,{-sin(60)}) {};
  
   \path[->] (1) edge[blue,bend left=15] (2);
   \path[->] (2) edge[blue,bend right=25] (3);
   \path[->] (2) edge[red] (3);
   \path[->] (3) edge[red,bend right=25] (2);
   \path[->] (3) edge[blue,bend left=15] (1);
   \path[->] (1) edge[red,loop above,min distance=6mm,in=110,out=160] (1);
  \end{scope}
  
  \begin{scope}[shift={(4.5,-2.5)}]
    \node[state] (1) {};
    \node[state] (2) at (1,0) {};
    \node[state,accepting] (3) at (.5,{-sin(60)}) {};
  
   \path[->] (1) edge[blue,bend left=15] (2);
   \path[->] (2) edge[blue,bend right=25] (3);
   \path[->] (2) edge[red] (3);
   \path[->] (3) edge[red,bend right=25] (2);
   \path[->] (3) edge[blue,bend left=15] (1);
   \path[->] (1) edge[red,loop above,min distance=6mm,in=110,out=160] (1);
  \end{scope}
  
  \begin{scope}[shift={(6.5,-2.5)}]
    \node[state,accepting] (1) {};
    \node[state] (2) at (1,0) {};
    \node[state] (3) at (.5,{-sin(60)}) {};
  
   \path[->] (1) edge[red,bend left=15] (2);
   \path[->] (2) edge[red,bend right=25] (3);
   \path[->] (2) edge[blue] (3);
   \path[->] (3) edge[blue,bend right=25] (2);
   \path[->] (3) edge[red,bend left=15] (1);
   \path[->] (1) edge[blue,loop above,min distance=6mm,in=110,out=160] (1);
  \end{scope}
  
  \begin{scope}[shift={(8.25,-2.5)}]
    \node[state] (1) {};
    \node[state,accepting] (2) at (1,0) {};
    \node[state] (3) at (.5,{-sin(60)}) {};
  
   \path[->] (1) edge[red,bend left=15] (2);
   \path[->] (2) edge[red,bend right=25] (3);
   \path[->] (2) edge[blue] (3);
   \path[->] (3) edge[blue,bend right=25] (2);
   \path[->] (3) edge[red,bend left=15] (1);
   \path[->] (1) edge[blue,loop above,min distance=6mm,in=110,out=160] (1);
  \end{scope}
  
  \begin{scope}[shift={(10,-2.5)}]
    \node[state] (1) {};
    \node[state] (2) at (1,0) {};
    \node[state,accepting] (3) at (.5,{-sin(60)}) {};
  
   \path[->] (1) edge[red,bend left=15] (2);
   \path[->] (2) edge[red,bend right=25] (3);
   \path[->] (2) edge[blue] (3);
   \path[->] (3) edge[blue,bend right=25] (2);
   \path[->] (3) edge[red,bend left=15] (1);
   \path[->] (1) edge[blue,loop above,min distance=6mm,in=110,out=160] (1);
  \end{scope}
  
\end{tikzpicture}
\caption{Els $13$ subgrups d'índex $3$ de $\Free[2] = \pres{a,b}{-}$}
\end{figure}
\end{exm}

Del \Cref{cor: fi iff sat} se segueix immediatament el resultat següent.
\begin{thm}
El problema de l'índex finit per a grups lliures, $\FIP(\Free)$, és computable.
\end{thm}

\begin{proof}
La decidibilitat de $\FIP(\Free)$ és conseqüència de \Cref{cor: fi iff sat} i de la computabilitat de $\stallings(H)$ en el cas finitament generat.

Suposem ara que $\stallings(H)$ és saturat i finit. Per a calcular una família de representats de les classes per la dreta mòdul $H$, només hem d'observar que, a l'autòmat $\stallings(H)= \core (\schreier(H))=\schreier(H)$, qualsevol camí $\walki$ amb inici $\bp$ té la forma~\eqref{eq: cami ampliat} i, per tant, compleix $\tau \walki=Hw$, on $w=\lab(\walki)\in \Free[A]$. Per tant, podem formar una família de representants de les classes laterals per la dreta mòdul $H$ simplement prenent un camí qualsevol del punt base $\bp$ a cadascun dels vèrtexs de $\stallings(H)$. Una manera sistemàtica de fer-ho és calcular un arbre d'expansió $T$ de $\stallings(H)$, i prendre les etiquetes de les $T$-geodèsiques de $\bp$ a cada vèrtex: $\{\ell(T[\bp, p]) \mid \verti \in \Verts\stallings(H)\}$ és una família de representants de les classes laterals per la dreta mòdul $H$, és a dir, 
\begin{equation}
H\backslash \Free[A] \,=\, \bigsqcup_{\mathclap{\verti \in \Verts\stallings(H)}} \ H\ell(T[\bp, p]).
\end{equation}
Òbviament, l'index $\ind{H}{\Free}$ és el cardinal (finit) d'aquest conjunt. Això resol completament el problema de l'índex finit per a grups lliures. (Si volem representants de les classes per l'esquerra, només hem d'invertir les paraules obtingudes, ja que~$(Hw)^{-1}=w^{-1}H^{-1}=w^{-1}H$.)
\end{proof}

\begin{exm}
Sigui $\Free[2]$ el grup lliure sobre $A=\{a,b\}$ i considerem el subgrup $H=\langle v_1, v_2, v_3, v_4 \rangle\leqslant \Free[2]$ generat pels elements $v_1=a$, $v_2=b^2$, $v_3=ba^2b^{-1}$, i $v_4=baba^{-1}b^{-1}$. Per decidir si $H$ és d'índex finit a $\Free[2]$ calculem $\stallings(H)$ a partir de l'autòmat flor $\flower(\{v_1, v_2, v_3, v_4\})$, tal i com hem explicat al \Cref{thm: Stallings bijection}. El resultat és l'autòmat següent:
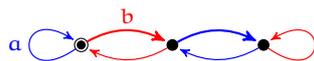
\vspace{-10pt}
\begin{figure}[H]
\centering
  \begin{tikzpicture}[shorten >=1pt, node distance=1.2 and 1.2, on grid,auto,>=stealth']
   \node[state,accepting] (0) {};
   \node[state] (1) [right = of 0]{};
   \node[state] (2) [right = of 1]{};

   \path[->]
        (0) edge[loop left,blue,min distance=10mm,in=145,out=215]
            node[left] {\scriptsize{$a$}}
            (0);
            
    \path[->]
        (0) edge[bend left,red,thick]
            node[above]{$b$} (1);
            
    \path[->]
        (1) edge[bend left,red]
            (0);
            
    \path[->]
        (1) edge[bend left,blue,thick] (2);
            
    \path[->]
        (2) edge[bend left,blue]
            (1);
            
    \path[->]
        (2) edge[loop left,red,min distance=10mm,in=35,out=325]
            (2);

\end{tikzpicture}
\end{figure}
\vspace{-10pt}

Com que els $3$ vèrtexs de $\stallings(H)$ tenen un $a$-arc entrant i un sortint, i un $b$-arc entrant i un sortint, l'autòmat $\stallings(H)$ és saturat i, per tant, $H$ és d'índex finit a~$\Free[2]$, concretament d'índex $\card \Verts\stallings(H)=3$. Prenent com a arbre d'expansió el donat pels dos arcs dibuixats amb traç gruixut, obtenim $\{1, b, ba\}$ com a conjunt de representants de les classes per la dreta mòdul $H$, és a dir, $\Free[2]=H\sqcup Hb\sqcup H(ba)$.
\end{exm} 

Com veiem a continuació, la coneguda fórmula de Schreier, que relaciona l'índex d'un subgrup d'índex finit amb el seu rang, també es dedueix de forma transparent de la Teoria de Stallings.  

\begin{thm}[Fórmula de Schreier]
Sigui $\Free[\kappa]$ el grup lliure de rang $\kappa$, i sigui $H\leqslant \Free[\kappa]$ un subgrup d'índex finit. Aleshores, $\rk(H)-1=[\Free[\kappa] :H](\kappa-1)$. En particular, el subgrup~$H$ és finitament generat si i només si el rang ambient $\kappa$ és finit.
\end{thm}

\begin{proof}
Per la \Cref{prop: propietats St}\ref{item: fi iff St}, si $H$ és d'índex finit, llavors $\stallings(H)$ és $2\kappa$-regular i $\card \Verts\stallings(H)<\infty$. Per tant, $\rk(\stallings(H))<\infty$ (\ie hi ha un nombre finit d'arcs fora de qualsevol arbre d'expansió) si i només si $\kappa<\infty$. En particular, si $\kappa =\infty$ llavors $\rk(H)=\infty$ i ja hem acabat. 

Suposem ara $\kappa <\infty$. Sabem que $H$ és finitament generat i, per tant, $\stallings(H)$ és finit. Sigui $T$ un arbre d'expansió de $\stallings(H)$; tal i com s'ha demostrat al final de la prova del~\Cref{thm: Stallings bijection}, $H$ té una base de $\card(\Edgs^+\stallings(H)\setminus \Edgs T)$ elements; per tant, 
 \begin{align*}
\rk (H)-1 &\,=\,\card(\Edgs^+\stallings(H)\setminus \Edgs T)-1 \,=\, \card \Edgs^+\stallings(H)-\card \Edgs T-1 \\ &\,=\, \card\Edgs^+\stallings(H)-\card \Verts T \,=\, \kappa \card\Verts\stallings(H)-\card\Verts\stallings(H) \\ &\,=\, \ind{H}{\Free[\kappa]}(\kappa-1),
 \end{align*}
on la penúltima igualtat prové de $2\kappa \,\card\Verts\stallings(H)=2\, \card \Edgs^+\stallings(H)$, obtingut sumant els graus de tots els vèrtexs. 
\end{proof}

Combinant les caracteritzacions gràfiques d'índex finit (\Cref{cor: fi iff sat}) i normalitat (\Cref{prop: normal iff}), arribem a una mena de recíproc del lem de Schreier per a subgrups normals no trivials del grup lliure.

\begin{cor}
Un subgrup normal no trivial $H$ de $\Fn$ és finitament generat si i només si té índex finit. \qed
\end{cor}

\begin{cor}
Un subgrup normal no trivial $H$ de $\Free[\aleph_0]$ té sempre rang infinit. \qed
\end{cor}

\medskip

Acabem aquesta secció amb una demostració gràfica particularment neta d'un resultat clàssic de Marshall Hall Jr., que mostra especialment bé la potència de l'enfoc geomètric per obtenir resultats algebraics.

Recordem que donats $H$ i $K$ subgrups d'un grup lliure $\Free$, es diu que $H$ és \defin{factor lliure} de $K$, i ho designem per $H \leqslant\ff K$, si una base (i, per tant, totes les bases) de $H$ es pot ampliar a una base de $K$. Observeu també que si $\Atii$ és un subautòmat de~$\Ati$, aleshores $\gen{\Atii}$ és un factor lliure $\gen{\Ati}$ (ampliant un arbre d'expansió de $\Atii$ a un arbre d'expansió de $\Ati$).

Un resultat ben sabut d'àlgebra lineal diu que, donat un espai vectorial $E$ i un subespai seu $F\leqslant E$, tota base de $F$ es pot ampliar a una base de $E$ (tècnicament, això correspon al fet que $F$ és un sumand directe de $E$, noció de la qual els factors lliures en són la generalització a ambients no commutatius). El corresponent resultat en un grup lliure és fals: sabem que $\Free[n]$ té subgrups $H$ de rang superior a $n$ (inclús de rang infinit) que, clarament, no en seran factors lliures, $H\not\leqslant\ff \Free[n]$. El clàssic teorema de Marshall Hall ens diu que un resultat d'aquesta naturalesa sí que és cert, si ens restringim a subgrups finitament generats, i canviem $\Free[n]$ per un subgrup d'índex finit adequat.

\begin{thm}[\citenr{hall_jr_topology_1950}] \label{thm: Hall vff}
Si $H$ és un subgrup finitament generat de $\Free[A]$, aleshores $H$ és un factor lliure d'un subgrup d'índex finit de $\Free[A]$; és a dir:
 \begin{equation} \label{eq: Hall vff}
H\leqslant\fg \Free \Imp \exists K\st H\leqslant\ff K\leqslant\fin \Free.
 \end{equation}
\end{thm}

\begin{proof}
Si $H$ ja és d'índex finit a $\Free[A]$, el resultat és obvi. En cas contrari, sabem que $\stallings(H)$ és un $A$-autòmat finit i insaturat. Ara és suficient adonar-se del següent fet purament geomètric:\footnote{Es tracta d'una adaptació natural del ben conegut \emph{handshaking lemma}.} en un $A$-autòmat involutiu i finit, per a tot $a\in A$, $\defc[a]{\Ati}=\defc[a^{-1}]{\Ati}$. Per tant, només caldrà saturar $\stallings(H)$ aparellant, per cada $a\in A$, els vèrtexs $a$-deficients amb els $(a^{-1})$-deficients, i afegint un $a$-arc dels primers als segons. L'autòmat $\Ati$ obtingut és finit i saturat (per tant reconeix un subgrup d'índex finit) i conté $\stallings(H)$ com a subautòmat. Per tant, $H = \gen{\stallings(H)} \leqslant\ff \gen{\Ati}$, tal i com volíem demostrar.
Noteu que l'aparellament de vèrtexs deficients de $\stallings(H)$ es pot fer, en general, de diverses maneres i per tant el subgrup obtingut $K$ no és únic, en general.
\end{proof}

\begin{exm}
Sigui $H=\gen{a^{-1} b,a b^2,a^{-2} b^{4},a^{-2}baba^{-1}b^{-1}a^2}\leqslant \Free[\set{a,b}]$. Per a obtenir un subgrup d'índex finit de $\Free[\{a,b\}]$ del qual $H$ en sigui factor lliure, és suficient calcular $\stallings(H)$ (en traç continu a la \Cref{fig: Marshall Hall}) i completar-lo a un autòmat saturat (per exemple, tal com s'indica amb traç discontinu). L'autòmat obtingut reconeix un subgrup $K$ d'índex 6 a $\Free[\{a,b\}]$ que conté $H$ com a factor lliure.
\begin{figure}[H] 
  \centering
  \begin{tikzpicture}[shorten >=1pt, node distance=1.2 and 1.2, on grid,auto,>=stealth']
   \node[state, accepting] (0) {};
   \node[state] (1) [right = of 0]{};
   \node[state] (2) [right = of 1]{};
   \node[state] (3) [right = of 2]{};
   \node[state] (4) [right = of 3]{};
   \node[state] (5) [below = 1 of 2]{};

   \path[->]
        (0) edge[red,dashed, bend left = 50] (2)
            edge[blue, bend right = 25] (5);

    \path[->]
        (1) edge[red, bend left= 25]
        node[below] {\scriptsize{$b$}}
            (0)
            edge[blue, bend right= 25]
         node[above] {\scriptsize{$a$}}
            (0);

    \path[->]
        (2) edge[blue]
            (1)
            edge[red]
            (3);

    \path[->]
        (3) edge[blue] (4)
            edge[red, bend left] (5);
            
    \path[->]
        (4) edge[loop,red,min distance=10mm,in=35,out=-35](4)
        (4) edge[blue,dashed, bend right = 50]
            (2);
            
    \path[->]
        (5) edge[red]
            (1)
            edge[blue,dashed, bend left]
            (3);
\end{tikzpicture}
\caption{Exemple del \Cref{thm: Hall vff} (Marshall Hall Jr.)}
\label{fig: Marshall Hall}
\end{figure}
\end{exm}

\subsection{El problema de la intersecció} \label{ssec: SIP}

En aquesta darrera secció usarem els autòmats de Stallings per resoldre el problema proposat a l'\Cref{exe: interseccio}: el problema de la intersecció de subgrups. Recordem que existeixen grups $G$ contenint parelles de subgrups $H,K\leqslant G$ finitament generats amb intersecció $H\cap K$ no finitament generada (a \cite{delgado_intersection_2022} i \cite{delgado_intersection_2018} trobareu exemples relativament senzills de grups d'aquest tipus; analitzats, a més, usant generalitzacions de les tècniques gràfiques que expliquem en aquesta secció). \citeauthor{howson_intersection_1954} va estudiar aquest fenomen en el grup lliure i el seu nom s'usa per referir-se als grups que no tenen aquest comportament ``estrany'': un grup $G$ es diu que té la \defin{propietat de Howson} (o que és \defin[Howson group]{Howson}) si la intersecció de qualssevol dos subgrups finitament generats $H,K\leqslant G$ és sempre finitament generada (per inducció, el mateix és cert llavors per interseccions d'una quantitat finita de subgrups). Per a un grup qualsevol $G = \pres{A}{R}$ té sentit, doncs, plantejar-se el problema de decisió següent:

\begin{named}[Problema de la intersecció de subgrups a $G = \pres{A}{R}$, $\SIP(G)$]
Donades \\ dues famílies finites de paraules
en els generadors de $G$,

$u_1,\ldots,u_k;v_1,\ldots,v_l \in (A^{\pm})^*$,
decidir si la intersecció dels corresponents subgrups generats
$\gen{u_1,\ldots,u_k} \cap \gen{v_1,\ldots,v_l}$
(a $G$) és finitament generada i, en cas afirmatiu, calcular-ne un sistema de generadors. 
\end{named}

A l'article \parencite{howson_intersection_1954}, Howson va demostrar que els grups lliures satisfan la propietat que porta el seu nom. A continuació, veurem com els autòmats de Stallings aporten una demostració gràfica molt neta i transparent (i, a més, constructiva!) d'aquest fet remarcable: de la pròpia demostració en sortirà una manera efectiva de calcular una base per a la intersecció, i una fita superior per al rang de $H\cap K$ en termes dels rangs de $H$ i de $K$: la famosa conjectura de Hanna Neumann, recentment resolta, de la qual parlarem al final de la secció.  


\begin{defn} \label{def: product of automata}
Siguin $\Ati_{\hspace{-1pt}1}=(\Verts_1, \Edgs_1, \init_1, \term_1, \ell_1, \bp_1)$ i 
$\Ati_{\hspace{-1pt}2}=(\Verts_2, \Edgs_2, \init_2, \term_2, \ell_2, \bp_2)$ dos $A$-autòmats. El \defin{producte}\footnote{o \emph{pull-back}, en terminologia categòrica.} de $\Ati_{\hspace{-1pt}1}$ i $\Ati_{\hspace{-1pt}2}$, designat per $\Ati_{\hspace{-1pt}1} \times \Ati_{\hspace{-1pt}2}$, és el $A$-autòmat amb: 
 \begin{itemize}[beginpenalty=10000]
\item conjunt de vèrtexs el producte cartesià $\Verts_{\!1} \times \Verts_{\!2}$;
\item un arc $(\verti_1,\verti_2) \xarc{a\,} (\vertii_1,\vertii_2)$ per a cada parell d'arcs $\verti_1 \xarc{a\,} \vertii_1$ a~$\Ati_{\hspace{-1pt}1}$, i $\verti_2 \xarc{a\,} \vertii_2$ a~$\Ati_{\hspace{-1pt}2}$ amb la mateixa etiqueta $a\in A$;
\item les funcions naturals $\iota, \tau, \ell$ com a funcions d'incidència i d'etiquetatge; 
\item vèrtex base $\bp=(\bp_1, \bp_2)$;
 \end{itemize}
vegeu la~\Cref{fig doubly enriched arc}.
\end{defn}
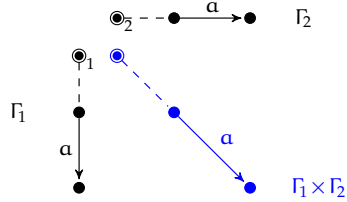
\begin{figure}[h]
     \centering
 \begin{tikzpicture}[shorten >=1pt, node distance=1.25cm and 1.25cm, on grid,auto,>=stealth']
 \newcommand{\dx}{1}
\newcommand{\dy}{1}
   \node (00) {};
   \node[state,accepting] (0) [below = \dy*0.5 of 00]{};
   \node[] (bp1) [below right =0.1 and 0.15 of 0] {$\scriptscriptstyle{1}$};
   \node[state,accepting] (0') [right = \dx*0.5 of 00]{};
   \node[] (bp1) [below right =0.1 and 0.15 of 0'] {$\scriptscriptstyle{2}$};
   \node[state,accepting,blue] (00') [right = \dx*0.5 of 0]{};
   \node[state] (1) [below = \dy*0.75 of 0] {};
   \node[state] (2) [below = \dy of 1] {};
   \node[] (G') [left = 0.8 of 1] {$\scriptstyle{\Ati_{\hspace{-1pt}1}}$};
   \node[state] (1') [right = \dx*0.75 of 0'] {};
   \node[state] (2') [right = \dx of 1'] {};
   \node[] (G) [right = 0.7 of 2'] {$\scriptstyle{\Ati_{\hspace{-1pt}2}}$};
   \node[state,blue] (11') [below right = \dy*0.75 and \dx*0.75 of 00'] {};
   \node[state,blue] (22') [below right = \dy and \dx of 11'] {};
     \node[blue] (GG') [right = 0.9 of 22'] {$\scriptstyle{\Ati_{\hspace{-1pt}1} \times \Ati_{\hspace{-1pt}2}}$};

   \path[dashed]
       (0') edge[] (1');
   \path[dashed]
       (0) edge[] (1);
   \path[dashed,blue]
       (00') edge[] (11');
   \path[->]
       (1') edge[]
             node[pos=0.5,above=-.1mm] {$a$}
             (2');
   \path[->]
         (1) edge[]
             node[pos=0.5,left=-.1mm] {$a$}
             (2);
   \path[->,blue]
         (11') edge[]
             node[pos=0.5,above right] {$a$}
             (22');
 \end{tikzpicture}
     \caption{Esquema del producte (en blau) de dos $A$-autòmats (en negre)}
     \label{fig doubly enriched arc}
 \end{figure}
Per tant, $\Ati_{\hspace{-1pt} 1}\times \Ati_{\hspace{-1pt} 2}$ no és res més que el producte tensorial o categòric\footnote{vegeu \url{https://en.wikipedia.org/wiki/Graph_product}.} de $\Ati_{\hspace{-1pt} 1}$ i $\Ati_{\hspace{-1pt} 2}$, \emph{respectant les etiquetes dels arcs} (\ie no aparellant mai arcs amb etiquetes diferents). Amb un exemple s'entén clarament el funcionament d'aquesta operació ben natural entre $A$-autòmats.  

\begin{exm}\label{exe: producte}
Considerem els subgrups $H=\gen{b, a^3, a^{-1}bab^{-1}a}$ i $K=\gen{ab,a^3, a^{-1}ba}$ de $\Free[\set{\!a,b\!}]$ de l'\Cref{exe: interseccio}. A la~\Cref{fig: pullback} podem veure l'autòmat $\stallings(H)$ dibuixat a la part superior (en format horitzontal), l'autòmat $\stallings(K)$ dibuixat a la part esquerra (en format vertical), i el producte $\stallings(H)\times \stallings(K)$ a la part central (i reorganitzat, a la dreta). 

\begin{figure}[H]
\centering
\begin{tikzpicture}[shorten >=1pt, node distance=1.2cm and 2cm, on grid,auto,auto,>=stealth']

\newcommand{\dx}{1.3}
\newcommand{\dy}{1.2}
\node[] (0)  {};
\node[state,accepting] (a1) [right = \dy-1/3 of 0] {};
\node[state] (a2) [right = \dx of a1] {};
\node[state] (a3) [right = \dx of a2] {};
\node[state] (a4) [right = \dx of a3] {};

\node[state,accepting] (b1) [below = 2*\dy/3 of 0] {};
\node[state] (b2) [below = \dy of b1] {};
\node[state] (b3) [below = \dy of b2] {};

\foreach \y in {1,...,4}
\foreach \x in {1,...,3} 
\node[state] (\x\y) [below right = (\x-1/3)*\dy and (\y-1/3)*\dx of 0] {};
    
\node[state,accepting] () [below right = 2*\dy/3 and 2*\dx/3 of 0] {};

\path[->]
     (a1) edge[red,loop above,min distance=7mm,in=205,out=155]
     node[above = 0.1] {$b$}
     (a1)
     (a1) edge[blue] node[below] {$a$} (a2)
     (a2) edge[blue] (a3)
     (a3) edge[blue,bend right=30] (a1)
     (a3) edge[red] (a4)
     (a4) edge[blue,loop above,min distance=7mm,in=25,out=-25] (a4)

     (b1) edge[blue,bend right=25] (b2)
     (b2) edge[red,bend right=25] (b1)
     (b2) edge[blue] (b3)
     (b3) edge[blue,bend left=40] (b1) 
     (b3) edge[red,loop left,min distance=7mm,in=295,out=245] (b3);
             
 \path[->]   
     (11) edge[blue] (22)
     (22) edge[blue] (33) 
     (33) edge[blue, bend right=25] (11)
            
     (21) edge[blue] (32)
     (32) edge[blue] (13)
     (13) edge[blue] (21)
        
     (12) edge[blue] (23)
     (23) edge[blue] (31) 
     (31) edge[blue] (12)
            
     (14) edge[blue] (24)
     (24) edge[blue] (34)
     (34) edge[blue, bend right=25] (14)
            
     (21) edge[red] (11)
     (23) edge[red] (14)
     (33) edge[red] (34)            
     (31) edge[red,loop left,min distance=7mm,in=250,out=200] (31);
    
 \node[] (i) [right = 1 of 24] {$=$};  
 \node[state,accepting] (1) [above right = 0.7 and 2.25 of i] {};
 \node[state] (1a) [below left = \dy and \dx/3 of 1] {};
 \node[state] (1b) [below right = \dy and \dx/3 of 1] {};
 \node[state] (2) [right = \dx*0.7 of 1b] {};
 \node[state] (2a) [above left = \dy and \dx/3 of 2] {};
 \node[state] (2b) [above right = \dy and \dx/3 of 2] {};
 \node[state] (3) [right = \dx*0.7 of 2b] {};
 \node[state] (3a) [below left = \dy and \dx/3 of 3] {};
 \node[state] (3b) [below right = \dy and \dx/3 of 3] {};
 \node[state] (0) [left = \dx*0.7 of 1a] {};
 \node[state] (0a) [above left = \dy and \dx/3 of 0] {};
 \node[state] (0b) [above right = \dy and \dx/3 of 0] {};

 \path[->]
     (0b) edge[red,thick] (1)
     (1b) edge[red,thick] (2)
     (3) edge[red,thick] (2b)
    
     (0b) edge[blue,thick] (0)
     (0a) edge[blue,thick] (0b)
     (0) edge[blue] (0a)
    
     (1b) edge[blue,thick] (1)
     (1a) edge[blue,thick] (1b)
     (1) edge[blue] (1a)
    
     (2) edge[blue,thick] (2b)
     (2b) edge[blue,thick] (2a)
     (2a) edge[blue] (2)
    
     (3) edge[blue,thick] (3b)
     (3b) edge[blue,thick] (3a)
     (3a) edge[blue] (3)
     (3b) edge[red,loop,min distance=7mm,in=25,out=-25] (3b);
\end{tikzpicture}
\vspace{-5pt}
\caption{
Producte dels autòmats $\stallings(H)$ i $\stallings(K)$}
\label{fig: pullback}
\end{figure}
\end{exm}

La proposició següent recull les propietats principals d'aquest producte d'autòmats, que el relacionen molt de prop amb la intersecció de subgrups. 

\begin{prop} \label{prop: prod props}
Siguin $\Ati_{\hspace{-1pt}1}=(\Verts_1, \Edgs_1, \init_1, \term_1, \ell_1, \bp_1)$ i $\Ati_{\hspace{-1pt}2}=(\Verts_2, \Edgs_2, \init_2, \term_2, \ell_2, \bp_2)$ dos $A$-autòmats. Aleshores, 
 \begin{enumerate}[ind]
\item \label{item: deg prod}
per a tot $(\verti,\vertii) \in \Ati_{\hspace{-1pt}1}\times \Ati_{\hspace{-1pt}2}$, tenim 
$0\leqslant \deg(\verti,\vertii)\leqslant \min \{\deg(\verti),\, \deg(\vertii)\}$;
\item \label{item: int det}
si $\Ati_{\hspace{-1pt}1}$ i $\Ati_{\hspace{-1pt}2}$ són deterministes llavors $\Ati_{\hspace{-1pt}1}\times \Ati_{\hspace{-1pt}2}$ també és determinista, i, en tal cas, $\gen{\Ati_{\hspace{-1pt}1}\times \Ati_{\hspace{-1pt}2}}=\gen{\Ati_{\hspace{-1pt}1}}\cap \gen{\Ati_{\hspace{-1pt}2}}$;
\item \label{item: prod not connected}
encara que $\Ati_{\hspace{-1pt}1}$ i $\Ati_{\hspace{-1pt}2}$ siguin connexos, $\Ati_{\hspace{-1pt}1}\times \Ati_{\hspace{-1pt}2}$ pot no ser connex;
\item \label{item: prod no core}
encara que $\Ati_{\hspace{-1pt}1}$ i $\Ati_{\hspace{-1pt}2}$ siguin cor, $\Ati_{\hspace{-1pt}1}\times \Ati_{\hspace{-1pt}2}$ pot no ser cor. 
 \end{enumerate}
\end{prop}

\begin{proof}
\ref{item: deg prod}. És clar de la \Cref{def: product of automata}, ja que tot $a$-arc sortint de $(\verti,\vertii) \in \Verts (\Ati_{\hspace{-1pt}1} \times \Ati_{\hspace{-1pt}2})$ correspon a un $a$-arc sortint tant de $\verti \in \Verts \Ati_{\hspace{-1pt}1}$ com de $\vertii \in \Verts\Ati_{\hspace{-1pt}2}$.

\ref{item: int det}. Que el producte d'autòmats deterministes és, de nou, determinista, és clar de la definició. Suposem que $\Ati_{\hspace{-1pt}1}$ i  $\Ati_{\hspace{-1pt}2}$ són determinstes; aleshores,
si $w\in \gen{\Ati_{\hspace{-1pt}1}\times \Ati_{\hspace{-1pt}2}}$, hi ha un $\bp$-camí reduït a $\Ati_{\hspace{-1pt}1}\times \Ati_{\hspace{-1pt}2}$ amb etiqueta $w$; projectant-lo a la primera i la segona coordenades, obtenim $\bp$-camins reduïts a $\Ati_{\hspace{-1pt}1}$ i a $\Ati_{\hspace{-1pt}2}$, respectivament, llegint igualment $w$; per tant, $w\in \gen{\Ati_{\hspace{-1pt}1}}\cap \gen{\Ati_{\hspace{-1pt}2}}$. I, recíprocament, si $w\in \gen{\Ati_{\hspace{-1pt}1}}\cap \gen{\Ati_{\hspace{-1pt}2}}$, aleshores hi ha $\bp$-camins reduïts a $\Ati_{\hspace{-1pt}1}$ i a $\Ati_{\hspace{-1pt}2}$, respectivament, ambdós llegint la paraula $w$; `cosint-los' arc per arc, obtenim un $\bp$-camí reduït a $\Ati_{\hspace{-1pt}1}\times \Ati_{\hspace{-1pt}2}$ que llegeix $w$ i, per tant, $w\in \gen{\Ati_{\hspace{-1pt}1}\times \Ati_{\hspace{-1pt}2}}$.

\ref{item: prod not connected} i \ref{item: prod no core}. Els subgrups $H=\gen{a^2, b}$ i $K=\gen{b, aba^{-1}}$ de $\Free[\{a,b\}]$ serveixen de contraexemple (deixem al lector el càlcul dels respectius autòmats de Stallings, i del seu producte). 
\end{proof}

\begin{cor}\label{cor: x}
Per $H,K\leqslant F_A$, tenim $\stallings(H\cap K)=\core (\stallings(H) \times \stallings(K))$. \qed
\end{cor}

Combinat amb la~\Cref{prop: S_T}, això ens demostra la propietat de Howson, la fita de Hanna Neumann, i ens resol el problema de la intersecció per a grups lliures. Definim el \defin{rang reduït} d'un grup lliure $\Free[n]$ com $\rrk(\Free[n])=\max\{0, n-1\}$, és a dir, $\rrk(\Free[n])=n-1$ excepte per al grup trivial, pel que posem $\rrk(1)=0$ (enlloc de~$-1$). 

\begin{thm}[Howson, \cite{howson_intersection_1954}; H. Neumann, \cite{neumann_intersection_1956}]\label{thm: Howson}
Els grups lliures $\Free$ (de qualsevol rang) són Howson. A més, si $H,K\leqslant\fg \Free$, aleshores $\rrk(H\cap K)\leq 2\rrk(H)\rrk(K)$.
\end{thm}

\begin{proof}
Observeu que, com que la propietat de Howson només involucra subgrups finitament generats, és suficient demostrar-la quan l'ambient $\Free$ és de rang finit, cosa que assumirem durant la demostració. Si $H,K\leqslant \Free[n]$ són finitament generats aleshores $\stallings(H)$ i $\stallings(K)$ són finits, i per tant $\stallings(H)\times \stallings(K)$ també serà finit, d'on (per la \Cref{prop: prod props}.\ref{item: int det}) es dedueix que $H\cap K$ torna a ser finitament generada. En conseqüència, $\Free[n]$ (i, per tant, tot grup lliure) és Howson.
 
Per veure la desigualtat, observem primer que si $\Ati$ és un graf connex i finit, i~$\Delta$ és el graf obtingut després d'eliminar successivament vèrtexs de grau $1$ (i les corresponents arestes incidents), aleshores $\rk(\Atii)=\rk(\Ati)$. D'altra banda, per la \Cref{prop: prod props}\ref{item: deg prod}, si $\deg(\verti)\geq 2$ i $\deg(\vertii)\geq 2$, aleshores 
\begin{equation}
\deg(\verti,\vertii)-2
\,\leq\,
\left( \deg(\verti)-2\big)\big( \deg(\vertii)-2 \right).
 \end{equation}

És clar que ens podem restringir al cas $H,K, H\cap K\neq \Trivial$. Siguin ara $\Ati_{\hspace{-1pt}H} = \stallings(H)$, $\Ati_{\hspace{-1pt}K} = \stallings(K)$, $\Ati$ la component connexa de $\Ati_{\hspace{-1pt}H}\times \Ati_{\hspace{-1pt}K}$ contenint el punt base, i~$\Ati^*_{\!H}$, $\Ati^*_{\!K}$ i $\Ati^*\subseteq \Ati^*_{\!H}\times \Ati^*_{\!K}$ els respectius cors reduïts. Com que $\Ati^*$ té tots els vèrtexs $(\verti,\vertii)$ de grau 2 o superior (i, per tant, $\verti$ i $\vertii$ també), usant les observacions anteriors, l'\Cref{eq: rang-degree} i les propietats del cor reduït obtenim:
\begin{align*}
\rk(H\cap K)-1
&\,=\,
\rk (\Ati^*) -1 \\
&\,=\,
\textstyle{\frac{1}{2}\sum_{(\verti,\vertii)\in \Verts\Ati^*} 
( \deg(\verti,\vertii)-2)}\\
&\,\leq\,
\textstyle{\frac{1}{2}\sum_{(\verti,\vertii)\in \Verts\Ati^*} ( \deg(\verti)-2\big)(\deg(\vertii)-2 )}\\
&\,\leq\,
\textstyle{\frac{1}{2}\sum_{(\verti,\vertii)\in \Verts(\Ati^*_{\!H}\times \Ati^*_{\!K})} (\deg(\verti)-2)(\deg(\vertii)-2)}\\
&\,=\,
\textstyle{\frac{1}{2}\big(\sum_{\verti\in \Verts\Ati^*_{\!\!H}} (\deg(\verti)-2) \big)
 \big( \sum_{\vertii\in \Verts\Ati^*_{\!\!K}} ( \deg(\vertii)-2) \big)}\\
&\, =\, 2( \rk(H)-1)( \rk(K)-1). &&\qedhere
\end{align*}
\end{proof}
Amb una demostració una mica més tècnica es pot veure una desigualtat més forta, que va demostrar Walter Neumann a~\cite{kovacs_intersections_1990}: 
 \begin{equation}\label{eq: WN}
\sum\nolimits_{w\in H\backslash \Free[A]/K} \rrk(H^w\cap K)\,\leq\, 2\rrk(H)\rrk(K)
 \end{equation}
on $H^w=w^{-1}Hw$, i el sumatori recorre tots els $w$'s d'un transversal qualsevol del conjunt de classes laterals dobles $H\backslash \Free[A]/K=\{HwK \mid w\in \Free[A]\}$ (es pot veure que els sumands no nuls es corresponen bijectivament amb les components connexes no arbre i no cícliques del producte $\stallings(H)\times \stallings(K)$). 

Arribats a aquest punt, és obligat mencionar que Hanna Neumann primer (sobre la desigualtat al \Cref{thm: Howson}) i Walter Neumann més tard (sobre la desigualtat~\eqref{eq: WN}) van conjecturar que el factor `2' es pot eliminar de les fites respectives. Aquestes són les famoses `Conjectura de Hanna Neumann' i `Conjectura Forta de Hanna Neumann'\footnote{\emph{Hanna Neumann Conjecture} i \emph{Strenghtened Hanna Neumann Conjecture}, en anglès.}, demostrades (independentment i quasi simultània) fa pocs anys per J. Friedman i I. Mineyev; vegeu \cite{friedman_sheaves_2015}, \cite{mineyev_submultiplicativity_2012}, i les remarcables simplificacions obtingudes per W. Dicks, a \cite[Appendix B]{friedman_sheaves_2015} i a \cite{dicks_simplified_2012}, respectivament.

\begin{thm}[Friedman \cite{friedman_sheaves_2015} i I. Mineyev \cite{mineyev_submultiplicativity_2012}]
Sigui $\Free$ un grup lliure i $H,K$ subgrups finitament generats de $\Free$. Aleshores, 
 \begin{equation*}\label{eq: WN2}
\sum\nolimits_{w\in H\backslash \Free[A]/K} \rrk(H^w\cap K)
\,\leq\,
 \rrk(H)\rrk(K).
 \end{equation*}
\end{thm}

Com ja hem justificat, la tècnica del producte d'autòmats també ens permet resoldre el problema de la intersecció.

\begin{thm}
El problema de la intersecció de subgrups per a grups lliures, $\SIP(\Free)$, és resoluble. \qed
\end{thm}

\begin{exm}
Com a exemple, acabem el càlcul començat a l'\Cref{exe: interseccio}. Es tractava de calcular la intersecció de $H=\langle u_1, u_2, u_3\rangle$ i $K=\langle v_1, v_2, v_3\rangle$ donats per les paraules $u_1=b$, $u_2=a^3$, $u_3=a^{-1}bab^{-1}a$, $v_1=ab$, $v_2=a^3$, i $v_3=a^{-1}ba$, com a subgrups de $\Free[A]$, $A=\{a,b\}$. A simple vista ja havíem trobat els elements $a^3,\, b^{-1}a^3 b,\, a^{-1}ba^3 b^{-1}a \in H\cap K$; però no quedava gens clar si aquests tres elements generen $H\cap K$, o en falten de més complicats per descobrir. 

A l'\Cref{exe: producte} hem calculat $\stallings(H)$, $\stallings(K)$ i $\stallings(H)\times \stallings(K)$. Com que aquest últim ja és connex i cor, el~\Cref{cor: x} ens diu que $\stallings(H\cap K)=\stallings(H)\times \stallings(K)$. Prenent com a arbre maximal $T$ l'indicat amb arcs en traç gruixut a la \Cref{fig: pullback}, obtenim la base
 \begin{equation*}
S_T
=
\set{
b^{-1}a^3 b,
a^3,a^{-1}ba^3b^{-1}a,
a^{-1}bab^{-1}a^3ba^{-1}b^{-1}a,
a^{-1}bab^{-1}aba^{-1}ba^{-1}b^{-1}a}
 \end{equation*}
per a la intersecció $H\cap K$.

A més, com que cadascun dels cinc generadors a $S_T$ és l'etiqueta d'un $\bp$-camí reduït a $\stallings(H)\times \stallings(K)$, el podem projectar a la primera i a la segona coordenada, respectivament, i, usant el mètode de la \Cref{ssec: MP}, podem reescriure'l com a paraula en $\{u_1, u_2, u_3\}$ i en $\{ v_1, v_2, v_3\}$:
\begin{equation*}
\begin{array}{rcl}
H\ni u_1^{-1} u_2 u_1 =\ & b^{-1}a^3 b &\ = v_1^{-1}v_2 v_1 \in K, \\
H\ni u_2 =\ & a^3 &\ = v_2 \in K, \\ 
H\ni u_3^3 =\ & a^{-1}ba^3b^{-1}a &\ = v_3v_2 v_3^{-1} \in K, \\
H\ni u_3u_2u_3^{-1} =\ & a^{-1}bab^{-1}a^3ba^{-1}b^{-1}a &\ = v_3v_1^{-1}v_2v_1v_3^{-1}\in K, \\
H\ni u_3u_1u_3^{-1} =\ & a^{-1}bab^{-1}aba^{-1}ba^{-1}b^{-1}a &\ = v_3v_1^{-1}v_2v_3v_2^{-1}v_1v_3^{-1}\in K.
 \end{array}
 \end{equation*}
\end{exm}

\begin{rem}
    Tal i com succeeix amb el problema de la pertinença (veure \cite{mikhailova_occurrence_1958}), no és difícil demostrar que el problema de la intersecció és també indecidible per a productes directes de grups lliures (veure~\cite[Corollary 2.3]{delgado_intersection_2018}), confirmant la indocilitat algorísmica d'aquesta família, en cert sentit propera al grup lliure.
\end{rem}

\medskip

Per acabar, vegem que la mateixa idea és útil també per a resoldre el problema de la intersecció de classes laterals en un grup lliure. Recordem que si $G$ és un grup qualsevol, $u,v\in G$ dos elements, i $H,K\leqslant G$ dos subgrups, llavors les classes laterals $Hu$ i $Kv$ o bé són disjuntes,
o bé s'intersequen exactament en una classe lateral de $H\cap K$, \ie $Hu\cap Kv=(H\cap K)w$. 

\begin{named}[Problema de la intersecció de classes laterals a $G = \pres{A}{R}$, $\CIP(G)$]
Donades dues famílies finites de paraules $u,u_1,\ldots,u_k;v,u_1,\ldots,u_l \in (A^{\pm})^*$, decidir si la intersecció $\gen{u_1,\ldots,u_k}u \cap \gen{v_1,\ldots,v_l}v \subseteq G$ és buida, i, en cas negatiu, calcular-ne un representant.
\end{named}

\begin{thm}
El problema de la intersecció de classes laterals per a grups lliures, $\CIP(\Free)$, és resoluble.
\end{thm}

\begin{proof}
Donades paraules reduïdes $u,u_1,\ldots,u_k;v,v_1,\ldots,v_l \in \Free[A]$, calculem els autòmats de Stallings
de
$H=\gen{u_1,\ldots,u_k}$ i $K=\gen{v_1,\ldots,v_l}$. Ara intentem seguir un camí reduït a $\stallings(H)$ començant a $\bp_H$ i llegint $u=a_{i_1}\cdots a_{i_n}$, (on $a_{i_1},\ldots a_{i_n}\in A^{\pm}$). Si no és possible, perquè després de llegir $a_{i_j}$ arribem a un vèrtex $(a_{i_{j+1}})$-deficient $\vertiii$, ampliem l'autòmat $\stallings(H)$ tot afegint el \emph{pèl} 
 \begin{equation}
\vertiii \! \xarc{\raisebox{0.6ex}{$\scriptstyle{a_{i_{j+1}}}$}} \bullet \xarc{\raisebox{0.6ex}{$\scriptstyle{a_{i_{j+2}}}$}} \bullet \xarc{\raisebox{0.6ex}{$\scriptstyle{a_{i_{j+3}}}$}} \ \cdots \ \xarc{\raisebox{0.6ex}{$\scriptstyle{a_{i_n}}$}} \verti\
 \end{equation}
(per poder completar la lectura de $u$ com a etiqueta d'un camí de $\bp_H$ a, diguem~$\verti$) i designem per $\stallings_u(H)$ l'autòmat resultant. Anàlogament, construïm $\stallings_v(K)$ (assegurant-nos que conté un camí llegint $v$, de $\bp_K$ a, diguem~$\vertii$) i calculem el producte $\Ati =\stallings_u(H)\times \stallings_v(K)$.

Ara, si $(\bp_H, \bp_K)$ i $(\verti,\vertii)$ estan a la mateixa component connexa de $\Ati$, prenem un camí $\walki$ de $(\bp_H, \bp_K)$ a $(\verti,\vertii)$, i el projectem a $\stallings_u(H)$ i a $\stallings_v(K)$, respectivament. Com que ambdues projeccions tenen etiqueta $w=\rlab(\walki)$, i van de $\bp_H$ a $p$, i de $\bp_K$ a $\vertii$, respectivament, concloem que $Hu=Hw$ i $Kv=Kw$; per tant, $w\in Hu\cap Hv$. 

Recíprocament, si $(\bp_H, \bp_K)$ a $(\verti,\vertii)$ són a components connexes diferents no hi ha cap parella de camins reduïts, de $\bp_H$ a $\verti$ (a $\stallings_u(H)$) i de $\bp_K$ a $\vertii$ (a $\stallings_v(K)$), compartit la mateixa etiqueta. Per tant, $Hu\cap Kv\neq \emptyset$ si i només si $(\bp_H, \bp_K)$ i $(\verti,\vertii)$ són a la mateixa component connexa de $\stallings_u(H)\times \stallings_v(K)$. Això completa la demostració. 
\end{proof}

Per acabar  d'i\lgem ustrar la utilitat d'aquesta tècnica del producte de $A$-autòmats, donem un argument força curt i elegant per demostrar, en l'àmbit dels grups lliures, un resultat sobre subgrups força conegut i vàlid en general (la prova del cas general requereix d'eines algebraiques més sofisticades com pot ser el Teorema del Subgrup de Kurosh).

\begin{prop}
Sigui $G$ un grup i siguin $H,K,H',K'$ subgrups de $G$. Si $H\leqff K$ i $H'\leqff K'$, llavors $H\cap H'\leqff K\cap K'$. 
\end{prop}

\begin{proof}
(Fem la demostració només en el cas que $G$ és un grup lliure, $G=\Free$.)

Veurem primer que si $H\leqslant\ff K\leqslant \Free$, i $L\leqslant \Free$, llavors $H\cap L \leqslant\ff K\cap L$. Considerem una base $A$ de $K$ que extengui una base de $H$, i observem que $\stallings(H,A)$ és simplement una rosa d'uns quants pètals (precisament, els etiquetats pels elements de $A$ pertanyents a $H$). Considerem $\stallings(K\cap L, A)$ i calculem ${H\cap L}=H\cap (K\cap L)$ fent el producte d'aquests dos $A$-autòmats: clarament, $\stallings(H,A)\times \stallings(K\cap L,A)$ és el $A$-subautòmat de $\stallings(K\cap L,A)$ determinat pels arcs amb etiqueta a $H$. Per tant, $H\cap L\leqff K\cap L$, tal com volíem. 

Aplicant aquest fet dues vegades, obtenim $H\cap H'\leqff K\cap H'\leqff K \cap K'$ i, per tant, $H\cap H'\leqff K\cap K'$, tal com volíem demostrar. 
\end{proof}


\section{Per saber-ne més}

La teoria dels autòmats de Stallings és un tema molt conegut i ben representat a la literatura. A més de l'article original de Stallings \parencite{stallings_topology_1983} (d'orientació més aviat topològica) existeixen diversos \emph{surveys} amb un punt de vista més proper a l'usat en aquest article (vegeu \eg \cite{kapovich_stallings_2002,bartholdi_rational_2021,delgado_list_2022}). Els articles usant la teoria de Stallings per resoldre altres problemes específics són innumerables i continuen apareixent amb regularitat. Podeu trobar alguns exemples a \parencite{stallings_todd-coxeter_1987,arzhantseva_class_1996,martino_automorphism-fixed_2000,margolis_closed_2001,miasnikov_algebraic_2007,silva_algorithm_2008,bassino_random_2008,bassino_statistical_2013,puder_primitive_2014,delgado_lattice_2020,delgado_relative_2022}.

Arrel de l'immens èxit de la Teoria de Stallings pel grup lliure, també hi ha hagut múltiples extensions d'aquesta idea a àmbits més generals, per exemple: grafs de grups
\parencite{stallings_foldings_1991,
rips_cyclic_1997,
sela_acylindrical_1997,
sela_diophantine_2001,
guirardel_approximations_1998,
guirardel_reading_2000,
kapovich_foldings_2005}; monoides i semigrups \parencite{margolis_free_1993,
margolis_closed_2001,
delgado_combinatorial_2002,
steinberg_inverse_2002}; grups satisfent certes propietats de petita cance\l.lació\footnote{\emph{small cancellation} en anglès.} \parencite{arzhantseva_class_1996}; grups totalment residualment lliures \parencite{myasnikov_fully_2006,
kharlampovich_subgroups_2004,
nikolaev_finite_2011}; productes lliures \parencite{ivanov_intersection_1999}; amalgames de grups finits \parencite{markus-epstein_stallings_2007}; grups actuant lliurement a $\ZZ^n$-arbres \parencite{nikolaev_membership_2012}; grups virtualment lliures \parencite{silva_finite_2016};
subgrups quasi-convexos \parencite{kharlampovich_stallings_2017}; productes directes i semidirectes de grups lliures amb grups lliure-abelians \parencite{delgado_algorithmic_2013,delgado_extensions_2017,delgado_stallings_FTA_2022}; complexos cúbics CAT(0)~\parencite{beeker_stallings_2018}; grups relativament hiperbòlics \parencite{kharlampovich_generalized_2020}; i grups de Coxeter d'angle recte\footnote{\emph{right-angled Coxeter groups} en anglès.}~\parencite{dani_subgroups_2021}, entre d'altres. 

\section*{Agraïments}
Els autors agraïm el suport parcial rebut de l'Agencia Estatal de Investigación, a través del projecte de recerca MTM2017-82740-P (AEI/FEDER, UE). El primer autor va realitzar la primera part d'aquest treball a la Universitat del País Basc (UPV/EHU) amb suport parcial del MINECO a través del projecte PID2019-107444GA-I00, i del govern basc amb el projecte IT974-16.

Agraïm als revisors i editors l'acurada lectura d'aquest text i els pertinents comentaris i suggeriments realitzats.

\newpage

\renewcommand*{\bibfont}{\footnotesize}
\printbibliography


\end{document}